%% file: hessian_averaging.tex
\documentclass[a4paper]{article}
\usepackage[a4paper, total={6.25in, 8.25in}]{geometry}
\usepackage[numbers,sort]{natbib}

\usepackage{amsmath,amssymb,amsbsy,amsfonts,amsthm}

\usepackage{hyperref}
\usepackage{comment}

\theoremstyle{plain}
\newtheorem{theorem}{Theorem}[section]
\newtheorem{lemma}[theorem]{Lemma}
\newtheorem{corollary}[theorem]{Corollary}
\newtheorem{proposition}[theorem]{Proposition}

\newtheorem{assumption}{Assumption}[section]
\newtheorem{condition}{Condition}[section]

\theoremstyle{definition}

\theoremstyle{remark}
\newtheorem{remark}{Remark}[section]

\usepackage{graphicx} 
\usepackage{wrapfig}
\usepackage{pifont}
\usepackage{color}
\usepackage{subcaption}
\usepackage{cancel}
\usepackage{mathtools}
\usepackage{mathabx}

\usepackage[section]{algorithm}
\usepackage[numbered]{algo}
\usepackage{multirow}

\usepackage[dvipsnames]{xcolor}

\usepackage{enumitem}

\usepackage{tikz}
\newcommand*\circled[1]{\tikz[baseline=(char.base)]{
            \node[shape=circle,draw,inner sep=2pt] (char) {#1};}}
\usetikzlibrary{positioning}
\usetikzlibrary{calc,patterns,math,arrows}
\tikzstyle{bag} = [align=center]

\def\E{\mathbb{E}}
\def\R{\mathbb{R}}
\def\N{\mathbb{N}}
\newcommand{\xmark}{\ding{55}}

\numberwithin{equation}{section}

\usepackage{soul}

\begin{document}


\title{Fast Unconstrained Optimization via Hessian Averaging and Adaptive Gradient Sampling Methods}


\author{Thomas O'Leary-Roseberry and Raghu Bollapragada}

\maketitle

\begin{abstract}
\input{sections/abstract}

\end{abstract}



\section{Introduction}
\input{sections/introduction}

\section{Hessian Averaging with Adaptive Gradient Sampling}
\label{sec:alg}
\input{sections/algorithmic_overview}

\section{Deterministic Sampling Analysis} \label{section:det_analysis}
\input{sections/deterministic_analysis}

\section{Stochastic Sampling Analysis} \label{section:stoch_analysis}

\input{sections/expectation_analysis}

\section{Practical Algorithms} \label{sec:practical}
\input{sections/hessian_averaged_newton_methods}

\section{Numerical Experiments} \label{section:numerics}
\input{sections/numerical_experiments}

\section{Conclusions}\label{sec:conclusion}
\input{sections/conclusions}

\section{Acknowledgements}
\input{sections/acknowledgments}

\bibliographystyle{siam}
\bibliography{tom}

\section{Appendices}

\input{sections/appendix_analysis}

\input{sections/appendix_numerics}

\appendix

\end{document}

%% file: sections/abstract.tex
We consider minimizing finite-sum and expectation objective functions via Hessian-averaging based subsampled Newton methods. These methods allow for gradient inexactness and have fixed per-iteration Hessian approximation costs. The recent work (Na et al. 2023) demonstrated that Hessian averaging can be utilized to achieve fast $\mathcal{O}\left(\sqrt{\tfrac{\log k}{k}}\right)$ local superlinear convergence for strongly convex functions in high probability, while maintaining fixed per-iteration Hessian costs. These methods, however, require gradient exactness and strong convexity, which poses challenges for their practical implementation. To address this concern we consider Hessian-averaged methods that allow gradient inexactness via norm condition based adaptive-sampling strategies. For the finite-sum problem we utilize deterministic sampling techniques which lead to global linear and sublinear convergence rates for strongly convex and nonconvex functions respectively. In this setting we are able to derive an improved deterministic local superlinear convergence rate of $\mathcal{O}\left(\tfrac{1}{k}\right)$. For the 
expectation problem we utilize stochastic sampling techniques, and derive global linear and sublinear rates for strongly convex and nonconvex functions, as well as a $\mathcal{O}\left(\tfrac{1}{\sqrt{k}}\right)$ local superlinear convergence rate, all in expectation. We present novel analysis techniques that differ from the previous probabilistic results. Additionally, we propose scalable and efficient variations of these methods via diagonal approximations and derive the novel diagonally-averaged Newton (Dan) method for large-scale problems. Our numerical results demonstrate that the Hessian averaging not only helps with convergence, but can lead to state-of-the-art performance on difficult problems such as CIFAR100 classification with ResNets.

%% file: sections/introduction.tex

We consider finite-sum optimization problems of the form
\begin{equation}\label{eq:prob_det}
    \min_{w \in \R^d} f(w) := \frac{1}{N}\sum_{i = 1}^{N} F_i(w),
\end{equation}
where the objective function $f:\R^d \rightarrow \R$ and the component functions $F_i:\R^d \rightarrow \R$ (for $i \in \{1,2,\dots,N\}$) are twice continuously differentiable functions. Problems of this form are ubiquitous in modern computing applications including machine learning \cite{Goodfellow-et-al-2016,bottou2018optimization} and scientific computing \cite{kouri2018optimization,Luo2024}.
In addition, finite-sum problems commonly arise in stochastic optimization settings when sample average approximations (SAA) of the optimization problems of the form
\begin{equation}\label{eq:prob_exp}
    \min_{w \in \R^d} f(w) := \E_{\zeta}[F(w,\zeta)],
\end{equation}
are considered. Here $\zeta$ is a random variable with associated probability space $(\Omega, \mathcal{F}, P)$, where $\Omega$ is a sample space, $\mathcal{F}$ is an event space, and $P$ is a probability distribution. The function $F: \mathbb{R}^d \times \Omega \rightarrow \mathbb{R}$ is a twice continuously differentiable function, and $\E_{\zeta}[\cdot]$ is the expectation taken with respect to the distribution of $\zeta$. For any given set of random realizations $\{\zeta_1,\zeta_2,\cdots, \zeta_N\}$ generated from the distribution $P$, an SAA problem of the form \eqref{eq:prob_det} is constructed by defining $F_i(\cdot):= F(\cdot,\zeta_i)$. In supervised machine learning, the random variable $\zeta:=(x,y)$ represents input-output data pairs and the function $f$ is a composition of a prediction function and a smooth loss function \cite{bottou2018optimization,Goodfellow-et-al-2016}. The resulting finite-sum problem is referred to as \emph{empirical risk} and the expectation problem is referred to as the \emph{expected risk} minimization problem \cite{vapnik1991principles}. In this paper, we consider algorithms for solving problem \eqref{eq:prob_det} and also suitably adapt them to problem \eqref{eq:prob_exp}. 

Several classes of methods have been proposed to solve \eqref{eq:prob_det} and \eqref{eq:prob_exp} (c.f. \cite{bottou2018optimization,nocedal1999numerical}). In this paper, we focus on subsampled Newton methods that employ gradient and Hessian approximations of the objective function in standard Newton-type methods. 
In these methods, at any given iteration $k \in \N$, the search direction $p_k$ is computed as the solution of the linear system of equations  
\begin{align*}
   \nabla^2 F_{S_k}(w_k) p_k = -\nabla F_{X_k}(w_k),
\end{align*}
where 
\begin{equation} \label{eq:sampling}
\nabla F_{X_k}(w_k) = \frac{1}{|X_k|}\sum_{i \in X_k} \nabla F_i(w_k), \quad 
    \nabla^2 F_{S_k}(w_k) = \frac{1}{|S_k|}\sum_{i \in S_k} \nabla^2 F_i(w_k).
\end{equation}
Here 
the sets $X_k, S_k$ are subsets of the index set $\{1,2,\cdots\}$ and $\zeta_i$'s are drawn independently from each other at random from the distribution $P$. In the case of the finite-sum problem \eqref{eq:prob_det}, the sets $X_k, S_k$ are subsets of the index set $\{1,2,\cdots,N\}$. The choice of the subsets $X_k, S_k$, referred to as sample sets, result in different algorithms. Several recent works have analyzed the theoretical and empirical properties for different choices of sampling sets $X_k$ and $S_k$ 
\cite{bollapragada2019exact,roosta2019sub,agarwal2017second,byrd2011use,byrd2012sample,erdogdu2015convergence,friedlander2012hybrid,pilanci2017newton,carter1991global}. 
Although the computationally efficient choice is to choose Hessian sample sizes ($|S_k|$) to be a fixed small constant, the Hessian approximations in these works require large $|S_k|$ to achieve fast local linear convergence, and increasingly large $|S_k|$ to achieve local superlinear convergence \cite{bollapragada2019exact,roosta2019sub}. However, such large Hessian sample sizes lead to high per-iteration computational cost making them unsuitable for large-scale optimization settings. 


Recently, Na et. al. \cite{na2023hessian}, have proposed and analyzed stochastic Hessian-averaged Newton methods that overcome the high per-iteration Hessian sampling costs of previous methods by instead employing a weighted average of the subsampled Hessians computed at past iterations. This approach increases the accuracy of the Hessian approximations by reducing its variance, albeit at the cost of introducing bias due to utilization of past iterates. Na et. al. established fast superlinear convergence results in high probability with exact gradients and fixed small Hessian sample sizes ($|S_k| = |S_0|$ for all $k$). However, the requirements therein of exact gradient computation at each iteration is not practical in large-scale optimization problems. In this work, we consider the gradients to be inexact and also propose deterministic Hessian-averaging methods where the Hessian samples (with fixed sample size) are chosen in a cyclic order without replacement that lead to faster deterministic superlinear rate of convergence for the finite-sum problem \eqref{eq:prob_det}.

While inexact gradients 
can be utilized in Newton-type methods,
such methods lead to slower rate of convergence which result in a significant increase in the overall number of linear system solves in these methods. Adaptive gradient sampling methods overcome this limitation where the sample sizes employed in the gradient estimation are gradually increased to increase the accuracy in gradient estimation and retain the fast convergence properties of their deterministic counterparts \cite{berahas2022adaptive, byrd2012sample,bollapragada2018adaptive,bollapragada2018progressive,cartis2018global,paquette2020stochastic,pasupathy2018sampling}. Although these methods employ increasingly accurate gradient approximations as the iterations increase ($|X_k|$ increases), they achieve similar overall gradient computational complexity as in the stochastic gradient methods (see \cite[Table 4.1]{byrd2012sample}, Table~\ref{table:complexity}).
These methods typically employ tests to control sample sizes that automatically adapt to problem settings \cite{byrd2012sample,bollapragada2018adaptive,bollapragada2018progressive,carter1991global}. In this work, we extend adaptive gradient sampling tests to the Hessian-averaging setting, where the gradient sample sizes are chosen either deterministically or randomly at each iteration. 
Our goal is a framework for fully inexact (stochastic) Hessian-averaged subsampled Newton 
methods that limit per-iteration Hessian computational costs, maintain rigorous convergence theory (i.e., global convergence with fast (superlinear) local convergence), and at the same time lead to practical algorithms.

While Hessian-based methods have sound theoretical properties, they may not be viable for modern large-scale optimization problems such as those arising in machine learning due to the storage cost of $\mathcal{O}(d^2)$ and ostensible matrix inversion cost of $\mathcal{O}(d^3)$ associated with Hessian matrix. To overcome this limitation, we propose practical diagonally approximated Newton (Dan) algorithms based on the Hessian-averaged Newton methods, similar to the one employed in Adahessian \cite{yao2021adahessian}, requiring only $\mathcal{O}(1)$ per-iteration Hessian-vector products and $\mathcal{O}(d)$ storage, making it appropriate for modern memory-constrained settings. We consider variations of this algorithm based on different weightings, and sampling schemes. Additionally, we illustrate the performance benefits of these methods on numerical examples ranging from stochastic quadratic minimization to large-scale deep learning classification, such as CIFAR100 with ResNets.

\subsection{Related Work}
We provide a concise summary of second-order methods that employ Hessian approximations in Newton-type methods and adaptive gradient sampling methods for solving \eqref{eq:prob_det} and \eqref{eq:prob_exp}. We note that this is not an extensive review of methods but rather closely related works to the methods considered in this work. 

\paragraph{Second-order methods.} Several second-order methods that employ Hessian approximations have been developed in the literature. Subsampled Newton methods are one of the main class in these methods \cite{agarwal2017second,bollapragada2019exact,byrd2011use,byrd2012sample,byrd2016inexact,erdogdu2015convergence, eisenstat1996choosing, martens2010deep,dembo1982inexact,lee2014proximal,pilanci2017newton,wang2015subsampled,xu2016sub,xu2020second}. Roosta et al. have analyzed the theoretical properties of these methods and provided convergence results in probability for the finite-sum problems \cite{roosta2019sub}. Bollapragada et al. have established convergence results in expectation for both finite-sum and expectation problems and employed conjugate gradient method for solving the Newton system of equations inexactly \cite{bollapragada2019exact}. They assumed that the individual component functions to be strongly convex to establish the results in expectation. Both these works have established that Hessian samples have to be increased to achieve superlinear rate of convergence. The empirical performance of subsampled Newton methods has been established on different problems \cite{xu2020second,berahas2020investigation,wang2019utilizing,chen2018comparison}.  Erdogdu and Montanari et al. have developed and analyzed subsampled Newton methods with truncated eigenvalue decomposition \cite{erdogdu2015convergence}. Agarwal et al. have analyzed Newton methods where the Newton system of equations are solved inexactly using a stochastic gradient method at each iteration\cite{agarwal2017second}. Subsampled Newton methods have also been adapted to nonconvex settings using trust region and cubic regularization methods \cite{xu2020newton,xing2023convergence}.

Newton-sketch methods are alternate methods for subsampled Newton methods applied to finite-sum problems \cite{pilanci2017newton,berahas2020investigation,gupta2020oversketched,lacotte2021adaptive,gower2019rsn}. These methods require access to the square root of the Hessian, which is possible when generalized linear models are considered in machine learning. Typically, sketching strategies such as randomized Hadamard transformations provide better Hessian approximations compared to subsampling strategies; however they are typically more computationally expensive due to the high per-iteration cost associated with constructing the linear system of equations \cite{berahas2020investigation}. Derezi\'{n}ski et al. proposed the Newton-Less method that employs sparse sketch matrices to reduce the computational cost of forming the approximate Hessians \cite{derezinski2021newton}. Sketching techniques have also been adapted to the distributed optimization settings \cite{bartan2023distributed}.

Subsampled Newton and Newton sketch methods require increasingly accurate Hessian approximations to achieve superlinear rate of convergence. Na et al. have proposed Hessian-averaging methods that overcome this limitation \cite{na2023hessian}. Jiang et al. have improved the global rate of convergence of \cite{na2023hessian} while maintaining similar superlinear rate of convergence \cite{jiang2024stochastic}. In these works, true gradients are employed in the step computations, thus limiting their deployment in practice. In this work, we consider inexact gradients and also consider deterministic Hessian-averaging methods in addition to stochastic Hessian-averaging methods that could further improve the superlinear rate of convergence. 



\paragraph{Adaptive gradient sampling methods.} 
Stochastic gradient methods are well-known and widely used method in machine learning. This method however suffers from slow sublinear convergence for strongly convex function due to variance in the stochastic gradient estimation. Adaptive gradient sampling methods overcome this limitation by gradually increasing the accuracy in the gradient estimation by controlling the gradient samples $|X_k|$ to ensure similar convergence guarantees as their deterministic counterparts \cite{bollapragada2023adaptive,byrd2012sample,cartis2018global,friedlander2012hybrid,paquette2020stochastic,pasupathy2018sampling,carter1991global}. Several adaptive rules have been developed to choose the gradient accuracy at each iteration within the algorithm and ``norm test" is a popular condition proposed for the unconstrained settings \cite{carter1991global,byrd2012sample,bollapragada2018adaptive,cartis2018global}. These adaptive methods achieve both optimal theoretical convergence results and first-order complexity results to achieve an $\epsilon$-accurate solution for the expectation problem. These methods have also been adapted to other problem settings, including derivative-free optimization \cite{shashaani2018astro,bollapragada2023adaptive,bollapragada2024derivative,bollapragadaWild2023adaptive} and stochastic constrained optimization \cite{beiser2023adaptive,berahas2022adaptive,xie2024constrained,bollapragada2023adaptive}. 
In this work, we will adapt these methods to the Hessian-averaging based subsampled Newton methods.  

\paragraph{Other related methods.} There are several other classes of methods for solving \eqref{eq:prob_det} and \eqref{eq:prob_exp}.  Variants of first-order methods including diagonal scaling and momentum attain good empirical performance on challenging machine learning tasks \cite{KingBa15,duchi2011adaptive,hinton2012neural,loshchilov2017decoupled}. Stochastic quasi-Newton methods that construct quadratic models of the objective function using only stochastic gradient information are a popular class of methods \cite{bollapragada2018progressive,berahas2016multi,mokhtari2015global,schraudolph2007stochastic,berahassampled,yousefi2022efficiency}.
The limited memory variants of these methods are competitive to first-order methods on several machine learning classification problems \cite{bollapragada2018progressive}. Additionally, Kronecker- factored approximate curvature (KFAC) methods have been demonstrated as powerful algorithms in stochastic optimization \cite{martens2015optimizing,ba2017distributed}.

\subsection{Contributions}
The main contributions of our work are as follows.

\begin{enumerate}
    \item We develop an adaptive Hessian-averaging algorithmic framework where we incorporate deterministic and stochastic adaptive generalizations of the ``norm condition" to choose the gradient accuracy at each iteration for solving \eqref{eq:prob_det} and \eqref{eq:prob_exp} respectively. We choose Hessian samples either in a deterministic cyclic fashion without replacement for the finite-sum problem \eqref{eq:prob_det} or randomly from the distribution $P$ for the expectation problem \eqref{eq:prob_exp}. Furthermore, we modify the Hessian-averaging scheme whenever it is not a positive-definite matrix to ensure that the Newton steps are employed at each iteration instead of skipping them (see \cite[Algorithm 1]{na2023hessian}) which ensures a convergence rate for every iteration from the start of the iteration as opposed to having a warm-up phase \cite{jiang2024stochastic}. Furthermore, such modifications automatically become inactive after the iterates enter locally strongly convex regime (see Lemma~\ref{lem:nonconvex}).
    
    \item For the finite-sum problem \eqref{eq:prob_det}, we establish \textit{deterministic} global linear (Theorem~\ref{thm:linear}) and sublinear (Theorem~\ref{thm:global_sublinear_bounded_deterministic}) convergence results by employing appropriately chosen gradient accuracy conditions for strongly convex and nonconvex functions respectively. When the iterates enter a locally strongly convex regime, we further establish \textit{deterministic} local superlinear convergence for the case where the Hessian samples are chosen in a cyclic manner without replacement as opposed to randomly choosing samples at each iteration (Theorem~\ref{thm:super_det}). This choice of Hessian samples leads to an improved superlinear rate of $\mathcal{O}\left(\frac{1}{k}\right)$ as opposed to existing results in the literature (see Table~\ref{table:conv_results}). Moreover this choice produces stronger deterministic results compared to results in probability or expectation. 
    \item We establish theoretical convergence results for the stochastic settings \eqref{eq:prob_exp} where the inaccurate gradient approximations are chosen such that the stochastic gradient accuracy conditions are satisfied and the Hessian samples are chosen randomly from the distribution $P$ at each iteration. We established similar global linear (Theorem~\ref{thm:linear_exp}) and sublinear (Theorem~\ref{thm:global_sublinear_bounded_expectation})  convergence results for strongly convex and nonconvex functions, respectively, as in the case of finite-sum problem. Furthermore, using an additional assumption related to the boundedness of the moment of iterates and local strong convexity of subsampled Hessians, we establish local superlinear convergence where the rate is $\mathcal{O}\left(\frac{1}{\sqrt{k}}\right)$ that matches  existing results in \cite{na2023hessian,jiang2024stochastic}, albeit in expectation as opposed to in probability (see Table~\ref{table:conv_results}). 
    We note that in order to prove the $\mathcal{O}\left(\frac{1}{\sqrt{k}}\right)$ expectation result and the $\mathcal{O}\left(\frac{1}{k}\right)$ deterministic result, we utilize different proof techniques than those employed in \cite{na2023hessian,jiang2024stochastic}. While the probabilistic results therein rely on Freedman's inequality for matrix martingales \cite{tropp2011freedman}, our approach decomposes the Hessian error and places the statistical sampling error at the optimum, allowing us to concentrate the sampling error via direct analysis. We overview the contributions of our convergence results relative to existing methods in Table \ref{table:conv_results}.

    \item We establish total computational complexity for Hessian (Corollary \ref{corr:iter_complexity_det}) and gradient (Corollary \ref{thm:grad_complexity}) computations necessary  to achieve an $\epsilon$-accurate solution. To allow for appropriate comparisons with existing results in the literature, we only considered globally strongly convex functions. Although, the gradient samples are increasing at each iteration, the overall complexity in terms of total gradient samples match with the well-known stochastic gradient method for the expectation problem in terms of the dependence on $\epsilon$. Furthermore, the dependence on the condition number of the problem is improved due to the Hessian-averaging techniques.These results are summarized in Table \ref{table:complexity}.
    
    \item To tackle large-scale problems, we motivate the use of Hessian-matrix products, which can be efficiently implemented via vectorization on GPUs, leading to algorithms with $\mathcal{O}(1)$ Hessian-vector products per iteration and $\mathcal{O}(d)$ memory requirements. In particular we utilize a diagonal approximation, which can be efficiently computed via randomized estimation \cite{MartinssonTropp2020,dharangutte2023tight}, leading to the novel diagonally-averaged Newton (Dan) method and its variants. 
    
    \item We demonstrate the numerical benefits of Hessian averaging on a range of numerical experiments ranging from stochastic quadratics, logistic regression, image classification (CIFAR10 and CIFAR100) as well as neural operator training.
    We demonstrate that Hessian averaging overcomes the inherent instability of fully-subsampled Newton methods.
    In order to target large-scale optimization problems, we propose efficient implementation of Dan using randomized Hutchinson diagonal estimation.    
    In our experiments, we demonstrate that Dan was competitive with Adam and Adahessian (often achieving superior performance); notably Dan does not employ gradient momentum.

\end{enumerate}
\begin{table}[H]
\centering
{\renewcommand{\arraystretch}{1.3}
\begin{tabular}{|c|c|c|c|c|c|c|}
\hline
    result      & objective & gradient & Hessian & global & \begin{tabular}{@{}c@{}}local \\ (superlinear)\end{tabular} & \begin{tabular}{@{}c@{}}result \\ type\end{tabular}\\ \hline
\cite{roosta2019sub} &    \begin{tabular}{@{}c@{}} strong convex \\ finite-sum \end{tabular}      &   subsampled       & subsampled      &   linear    & asymptotic      & $\mathbb{P}$\\ \hline
\cite{bollapragada2019exact}          &     \begin{tabular}{@{}c@{}} strong convex \\ expectation \end{tabular}   &    subsampled      &  subsampled        &   linear     & asymptotic     & $\mathbb{E}$ \\ \hline
\cite{na2023hessian,jiang2024stochastic}         &   \begin{tabular}{@{}c@{}} strong convex \\ finite-sum \end{tabular}    &  exact       & path-averaged      &  linear   &    $\mathcal{O}\left(\sqrt{\tfrac{\log k}{k}}\right)$   & $\mathbb{P}$ \\ \hline
\multirow{2}{*}{\begin{tabular}{@{}c@{}} \enskip \\ Theorem \ref{thm:super_det}\end{tabular}} &     nonconvex      & \multirow{2}{*}{ \begin{tabular}{@{}c@{}} \enskip \\ adaptive\end{tabular}} & \multirow{2}{*}{\begin{tabular}{@{}c@{}} \enskip \\ path-averaged\end{tabular}} &    sublinear    & \multirow{2}{*}{{\begin{tabular}{@{}c@{}} \enskip \\ $\mathcal{O}\left(\frac{1}{k}\right)$\end{tabular}}} & \multirow{2}{*}{{\begin{tabular}{@{}c@{}} \enskip \\ deterministic\end{tabular}}} \\ \cline{2-2} \cline{5-5}
                        &    \begin{tabular}{@{}c@{}} strong convex \\ finite-sum\end{tabular}       &                  &                   &     linear   &                   &                   \\ \hline
\multirow{2}{*}{\begin{tabular}{@{}c@{}} \enskip \\ Theorem \ref{thm:super_stoc}\end{tabular}} &     nonconvex      & \multirow{2}{*}{\begin{tabular}{@{}c@{}} \enskip \\ adaptive\end{tabular}} & \multirow{2}{*}{\begin{tabular}{@{}c@{}} \enskip \\ path-averaged\end{tabular}} &    sublinear    & \multirow{2}{*}{{\begin{tabular}{@{}c@{}} \enskip \\ $\mathcal{O}\left(\frac{1}{\sqrt{k}}\right)$\end{tabular}}} & \multirow{2}{*}{{\begin{tabular}{@{}c@{}} \enskip \\ $\mathbb{E}$\end{tabular}}} \\ \cline{2-2} \cline{5-5}
                        &     \begin{tabular}{@{}c@{}} strong convex \\ expectation\end{tabular}      &                   &                   &    linear    &                   &                   \\ \hline

\end{tabular}
}
\caption{Comparison of convergence results. Here $\mathbb{P}$ denotes that the result is probabilistic, while $\mathbb{E}$ denotes that the result is proven in expectation. }
\label{table:conv_results}
\end{table}

\subsection{Organization}
The paper is organized into six sections. In the remainder of this section, we define notation that is employed throughout the paper. In Section~\ref{sec:alg}, we first describe the Hessian averaging methods and then discuss adaptive gradient-accuracy conditions and the corresponding sample size requirements to satisfy these conditions. In Section~\ref{section:det_analysis}, we establish theoretical results for the finite-sum minimization problem \eqref{eq:prob_det}. We establish global convergence, superlinear local convergence, followed by global to local transition iteration complexity results. In Section~\ref{section:stoch_analysis}, we repeat this analysis for the expectation based sampling methods whose target is the expectation minimization problem \eqref{eq:prob_exp}. We also provide iteration and gradient evaluation complexity results. In Section~\ref{sec:practical}, we discuss practical algorithms that are designed to handle large-scale problems. In Section~\ref{section:numerics}, we illustrate the performance of the proposed algorithms on various problems. Finally, concluding remarks 
are presented in Section~\ref{sec:conclusion}.

\subsection{Notation}

We use the following notation throughout the paper. 
We work with the natural numbers $\mathbb{N}= \{1,2,\cdots\}$, positive integers $\mathbb{Z}^+ = \mathbb{N}\cup\{0\}$, and the reals $\mathbb{R} = (-\infty,+\infty)$.
We consider real-valued ($\mathbb{R}$) spaces $\mathbb{R}^{d}, \mathbb{R}^{d\times r}$ for vectors and matrices, respectively, for dimensions $d,r \in \mathbb{N}$. 
We use the Euclidean $\ell^2$ norm, $\|\cdot \| = \|\cdot\|_2$ for both vectors and matrices, unless otherwise specified.
Assuming $A \in \mathbb{R}^{d\times d}$ is symmetric positive definite, the $A$ weighted norm is defined as $\|w\|_A = \sqrt{w^TAw}$.
Given a symmetric but potentially indefinite matrix $A$, we denote by $\lambda_{\min}(A), \lambda_{\max}(A)$ the smallest and largest eigenvalues of $A$, respectively and $|A|$ denotes the matrix obtained by replacing the negative eigenvalues with their magnitudes.
We denote by $\mathbb{E}_k[\cdot]$ the conditional expectation conditioned on the fact that the algorithm reach iterate $w_k$. When using more specific conditional expectations, definitions will be specified in the text. 
Given a sequence $\{e_k\in \mathbb{R}\}_{k=1}^{\infty}$ with a limit $e^*$, we characterize its rate of convergence as follows. We say that $e_k$ has Q-convergence with order $q$ to $e^* \in \mathbb{R}$ if there exists a constant $C$ such that
\begin{equation*}
    \lim_{k\rightarrow\infty} \frac{e_{k+1}}{e_k^q} = C.
\end{equation*}
For all cases with $q>1$, and the case that $q=1$ and $C = 0$, we say that $e_k$ convergences Q-superlinearly to $e^*$. 
We say that $e_k$ has R-convergence with order $q$ to $e^*$ if there exists a sequence $r_k$, such that $|e_k - e^*| < r_k$ for all $k$ and $r_k$ has Q-convergence with order $q$ to $0$. R-superlinear convergence is attained when $r_k$ has Q-superlinear convergence to $0$.
We note by $\mathcal{\tilde{O}}(\cdot)$ the limiting behavior of a function, disregarding logarithmic factors.

%% file: sections/algorithmic_overview.tex

The generic iterate update form of Hessian-averaged Newton method for solving problems of the form \eqref{eq:prob_det} or \eqref{eq:prob_exp} is given as,
\begin{equation}\label{eq:iter}
 w_{k+1} = w_k - \alpha_k p_k, \quad p_k = \widetilde{H}_k^{-1} g_k,     
\end{equation}
where $\alpha_k > 0$ is the step size parameter, and $g_k \in \R^d, \widetilde{H}_k \in \R^{d \times d}$, are the gradient and Hessian approximations of $\nabla f(w_k)$ and $\nabla^2 f(w_k)$ respectively. We now discuss the algorithmic components of Hessian averaging to compute $\widetilde{H}_k$ and adaptive gradient sampling to compute $g_k$ for each $w_k$, with $k \in \mathbb{Z}^+$.
  
\paragraph{Hessian Averaging.} 

In typical subsampled Newton methods, Hessians are approximated via subsampling where the sample sets are chosen either in a deterministic cyclic fashion or randomly drawn from $P$ as given in \eqref{eq:sampling}. Though these estimates are unbiased, they typically have large variance and require either large or increasing sample sizes $|S_k|$ to achieve fast local linear or superlinear rates of convergence respectively \cite{bollapragada2019exact,roosta2019sub}. Additionally, subsampled Newton iterates that are based on aggressively subsampled Hessian approximations may become unstable, due to the injection of statistical sampling errors in the iterate update. 
Sample size selection procedures that overcome these issues are not computationally viable in large-scale settings. Hessian-averaging approaches overcome this computational hurdle and reduce the variance in the estimation by employing a path-averaged Hessian with coefficients $\gamma_i \in [0,1]$, $\sum_{i=0}^k\gamma_i = 1$, defined as follows:
\begin{equation}\label{eq:avrg_hessian}
	\text{Path-averaged Hessian:} \qquad \widehat{H}_k = \sum_{i=0}^k \gamma_i\nabla^2 F_{S_i}(w_i).
\end{equation}
Here $S_i$ are the independent sample sets drawn either deterministically without replacement in a cyclic fashion or at random at each iteration $i$. The Hessian approximation error corresponding to this path-averaged estimate can be decomposed into two terms as follows:  

\begin{align} \label{eq:hessian_error_decomposition}
	\widehat{H}_k - \nabla^2 f(w_k) = \underbrace{\sum_{i=0}^k \gamma_i \left(\nabla^2 F_{S_i}(w_i) - \nabla^2 F_{S_i}(w_k)\right)}_{\text{Hessian memory error}} + \underbrace{\sum_{i=0}^k \gamma_i\nabla^2 F_{S_i}(w_k) -\nabla^2 f(w_k)}_{\text{sampling error}}.
\end{align}
Choosing $\gamma_i=0$ for all $i=0,\cdots,k-1$ and $\gamma_k = 1$ correspond to the well-known subsampled Newton methods \cite{bollapragada2019exact,roosta2019sub} and results only in the \emph{sampling error} which could be significantly large when sample size $|S_k|$ is small. On the other hand, choosing $\gamma_i = \frac{1}{k+1}$ for all $i=0,\cdots,k,$ reduces the \emph{sampling error} but introduces \emph{Hessian memory error} due to utilization of past information. However, as the iterates converge, both these errors decrease leading to accurate Hessian approximations which is the main motivation for this approach.

In nonconvex settings, the path-averaged subsampled Hessians $\widehat{H}_k$ \eqref{eq:avrg_hessian} may not be positive-definite and so the computation of search direction $\widehat{H}_k^{-1}g_k$ may not be well-defined. To overcome this limitation, we modify the path-averaged subsampled Hessians to ensure it is spectrally lower bounded below. That is, for any given $\tilde{\mu} > 0$, we define the Hessian approximation as  
\begin{equation}\label{eq:htilde_def}
    \widetilde{H}_k = 
    \begin{cases} 
        |\widehat{H}_k| \quad &\mbox{if } \lambda_{\min}(|\widehat{H}_k|) \geq \tilde{\mu} \\
        |\widehat{H}_k| + \left(\tilde{\mu} - \lambda_{\min}(|\widehat{H}_k|)\right)I \quad &\mbox{otherwise},
    \end{cases}
\end{equation}
where $\lambda_{\min}(A)$ is the smallest eigenvalue of symmetric matrix $A\in\R^{d \times d}$, and $|A|$ denotes the matrix obtained by replacing the negative eigenvalues with their magnitudes. That is, for any symmetric matrix $A = U\Lambda U^T$, $U\in\R^{d\times d}$ is the orthogonal matrix, and $\Lambda \in \R^{d\times d}$ is the diagonal matrix with eigenvalues,  $|A| = U|\Lambda| U^T$. This modification ensures that 
\begin{equation}\label{eq:spec_lower_bnd}
    \widetilde{H}_k \succeq \tilde{\mu} I. 
\end{equation}
and in the case where $ \widehat{H}_k \succeq \tilde{\mu} I$, there is no modification to the path-averaged Hessian ($\widetilde{H}_k = \widehat{H}_k$).

\paragraph{Adaptive Gradient Sampling.} The performance of the algorithms based on the iterate update form given in \eqref{eq:iter} also depends on the accuracy of the gradient approximations. In \cite{na2023hessian}, the authors consider exact gradients, which is not practical in large-scale finite-sum problems \eqref{eq:prob_det} or expectation problems \eqref{eq:prob_exp}. Subsampled gradients overcome this limitation, however the sampling errors that they introduce lead to slow convergence. Furthermore, due to the additional costs due to the Hessian evaluations, it is imperative to reduce the number of overall iterations to achieve computational efficiency. Adaptive sampling approaches gradually increase the accuracy in the gradient estimation via increasing sample sizes used in the gradient approximation to achieve fast convergence. These methods ensure that the error in the gradient approximation at each iteration is relatively small compared to the gradient itself. In this work, we combine these approaches with Hessian averaging approaches leading to generalized versions of the norm condition \cite{byrd2012sample,bollapragada2018adaptive,carter1991global} adapted to Newton-type methods. Specifically, we consider the following deterministic and stochastic conditions on the accuracy of the gradient approximations for the finite-sum \eqref{eq:prob_det} and expectation \eqref{eq:prob_exp} problems. 

\begin{condition}{(Gradient sampling conditions).}\label{condition:expected_gradient_bounds}
At each iteration $k \in \mathbb{Z}^+$ and a given symmetric positive-definite matrix $A_k \in \R^{d \times d}$, we utilize the following conditions. 
\begin{subequations}
\begin{enumerate}
\item Deterministic norm condition: For any given $\theta_k \in [0,1]$ and $\iota_k > 0$, the error in the gradient approximation satisfies the following deterministic norm condition. 
\begin{equation}\label{eq:deter_norm}
\|g_k - \nabla f(w_k)\|_{A_k}^2 \leq \theta_k^2\|\nabla f(w_k)\|_{A_k}^2 + \iota_k.
\end{equation}
\item Stochastic norm condition: For any given $\theta_k > 0$ and $\iota_k > 0$, the expected error in the gradient approximation satisfies the following generalized norm condition. 
\begin{equation}\label{eq:expected_norm}
\mathbb{E}[\|g_k - \nabla f(w_k)\|_{A_k}^2 | w_k, A_k] \leq \theta_k^2\|\nabla f(w_k)\|_{A_k}^2 + \iota_k,
\end{equation}
where $\mathbb{E}[\cdot | w_k, A_k]$ denote the conditional expectation conditioned on $A_k$ and that the algorithm reach iterate $w_k$.
\end{enumerate}
\end{subequations}
\end{condition}
\begin{remark}
We note that these gradient sampling conditions reduce to the well-known ``norm condition" for the choice of $A_k = I$ and $\iota_k=0$. In this paper, we consider this choice along with $A_k = \widetilde{H}_k^{-1}$ that leads to better theoretical convergence results (see Theorems ~\ref{thm:linear} and ~\ref{thm:global_sublinear_bounded_expectation}). In addition, the sequence $\iota_k$ further relaxes the norm condition and by suitably driving this sequence to zero, we establish the convergence and rate of convergence results. 
\end{remark}

These gradient sampling conditions are satisfied by choosing the gradient sample sizes $|X_k|$ appropriately. Specifically, under the following assumption, we can establish bounds on $|X_k|$ that satisfy these conditions.
\begin{assumption}{(Gradient approximations).}\label{assum:bnd_var}
For all $w\in\R^d$, the individual component gradients are bounded relative to the gradient of the objective function $f(w)$. That is, 
\begin{subequations}
\begin{enumerate}
    \item For the finite-sum problem: There exists constants $\beta_{1,g},\beta_{2,g} \geq 0$ such that
    \begin{equation}\label{eq:det_grad_var_bnd}
        \|\nabla F_i(w)\|^2 \leq \beta_{1,g} \|\nabla f(w)\|^2 + \beta_{2,g} \quad \forall i\in \{1,2,\cdots,N\}.
    \end{equation}
     \item For the expectation problem: There exists constants $\sigma_{1,g}, \sigma_{2,g} \geq 0$ such that
    \begin{equation}\label{eq:exp_grad_var_bnd}
        \E_{\zeta}[\|\nabla F(w,\zeta) - \nabla f(w)\|^2| w] \leq \sigma_{1,g}^2\|\nabla f(w)\|^2 + \sigma_{2,g}^2.
    \end{equation}
\end{enumerate}
\end{subequations}
\end{assumption}
\begin{remark}
    We note that \eqref{eq:det_grad_var_bnd} is a weaker assumption compared to the assumption where the individual gradient components are absolutely bounded and  
    \eqref{eq:exp_grad_var_bnd} is a standard assumption in stochastic optimization literature \cite{bottou2018optimization}. 
\end{remark}

The following lemma establishes the bounds on gradient sample sizes $|X_k|$ at each iteration $k$.

\begin{lemma}\label{lem:samplesizes}
    Suppose Assumption~\ref{assum:bnd_var} holds. For any $k \in \mathbb{Z}^+$  and $\lambda_{\min}(A_k), \lambda_{\max}(A_k) \in (0,\infty)$ denote the smallest and largest eigenvalues of $A_k$ respectively, we have that
    \begin{subequations}
    \begin{enumerate} 
        \item If \eqref{eq:det_grad_var_bnd} holds and 
        \begin{equation}\label{eq:gradsamp_bnd_det}
            |X_k| \geq N \left(1 - \sqrt{\frac{\theta_k^2\|\nabla f(w_k)\|_{A_k}^2 + \iota_k}{4\lambda_{\max}(A_k) ( \beta_{1,g}\|\nabla f(w_k)\|^2 + \beta_{2,g})}} \right),
        \end{equation}
        then deterministic norm condition \eqref{eq:deter_norm} is satisfied.
        \item If \eqref{eq:exp_grad_var_bnd} holds, $\E[g_k | w_k] = \nabla f(w_k)$, and 
        \begin{equation}\label{eq:gradsamp_bnd_exp}
            |X_k| \geq \frac{\lambda_{\max}(A_k)(\sigma_{1,g}^2\|\nabla f(w_k)\|^2 + \sigma_{2,g}^2)}{\theta_k^2\|\nabla f(w_k)\|_{A_k}^2 + \iota_k},
        \end{equation}
        then stochastic norm condition \eqref{eq:expected_norm} is satisfied.
    \end{enumerate}
    \end{subequations}
\end{lemma}

\begin{proof}
\emph{Deterministic norm condition.} Consider
\begin{align*}
    \|g_k - \nabla f(w_k)\|_{A_k}^2 &\leq \lambda_{\max}(A_k) \|g_k - \nabla f(w_k)\|^2 \\
    &\leq 4 \lambda_{\max}(A_k) \left( \frac{N - |X_k|}{N}\right)^2 \left( \beta_{1,g}\|\nabla f(w_k)\|^2 + \beta_{2,g}\right) \\
    &\leq \theta^2_k  \|\nabla f(w_k)\|_{A_k}^2 + \iota_k.
\end{align*}
where the first inequality is due to $w^TAw \leq \lambda_{\max} (A_k) \|w\|^2$, the second inequality is due to \eqref{eq:det_grad_var_bnd} and the analysis provided in \cite[Section 3.1]{friedlander2012hybrid}, and the third inequality is due to the bound on $|X_k|$ given in \eqref{eq:gradsamp_bnd_det}. 
We provide the derivation of the second inequality in Appendix \ref{appendix:analysis} for completeness. 
 
\emph{Stochastic norm condition.} Following a similar approach as in the deterministic norm condition, and using $\E[g_k | w_k] = \nabla f(w_k)$, we have
\begin{align*}
    \mathbb{E}[\|g_k - \nabla f(w_k)\|_{A_k}^2|w_k,A_k] &\leq \lambda_{\max}(A_k) \mathbb{E}[\|g_k - \nabla f(w_k)\|^2|w_k] \\
    &= \lambda_{\max}(A_k) \frac{\E_{\zeta}[\|\nabla F(w_k,\zeta) - \nabla f(w_k)\|^2|w_k]}{|X_k|} \\
    &\leq \lambda_{\max}(A_k) \frac{\sigma_{1,g}^2\|\nabla f(w_k)\|^2 + \sigma_{2,g}^2}{|X_k|} \\
    &\leq  \theta_k^2 \|\nabla f(w_k)\|_{A_k}^2 + \iota_k.
\end{align*}
\end{proof}
In Section~\ref{subsection:stochastic_gradient_practical}, we provide practical strategies for choosing the gradient sample sizes $|X_k|$ at each iterate $w_k$, $k\in \mathbb{Z}^+$ 
instead of employing these pessimistic theoretical bounds that require accessing unknown problem specific constants such as $\sigma_{1,g}$ and $\sigma_{2,g}$. 

In what follows, we split our convergence theory into two parts: first the deterministic sampling convergence theory where we prove a novel ${\mathcal{O}}\left(\tfrac{1}{k}\right)$ superlinear local convergence rate, followed by a section where we extend this analysis to the stochastic setting, deriving bounds \emph{in expectation}, where our results match the \emph{probabilistic} $\tilde{\mathcal{O}}\left(\tfrac{1}{\sqrt{k}}\right)$ superlinear local convergence rate. In both cases we simultaneously assume gradient and Hessian inexactness, and maintain a fixed per-iteration Hessian computational cost.

%% file: sections/deterministic_analysis.tex

\input{sections/tikz_diagrams/analysis_roadmap_det}

We begin our analysis by focusing on the deterministic sampling-based algorithms for the solution of the finite-sum minimization problem \eqref{eq:prob_det}, where the subsampled gradients satisfy the deterministic norm condition \eqref{eq:deter_norm}, and the subsampled Hessians are deterministically sampled without replacement in a cyclic fashion.
We establish theoretical global linear and sublinear convergence results for strongly convex and nonconvex functions respectively, and local superlinear convergence results when the iterates enter a strongly convex neighborhood of the optimal solution. A schematic of this analysis summarizing the convergence results is given in Figure~\ref{fig:analysis_roadmap_det}. We begin with an assumption about the subsampled functions.

\begin{assumption}{(Spectral upper bounds of subsampled Hessians).} \label{assumption:hessian_spectral_bounds}
The subsampled functions are twice continuously differentiable with the eigenvalues of the subsampled Hessians bounded above where the bound depend on the sample size $|S|$. That is, for all $|S|\in \mathbb{N}$, there exists constants $0 < L_{|S|} < \infty$ such that
\begin{equation}
 \nabla^2 F_{S}(w) \preceq L_{|S|} I, \qquad \forall w \in \mathbb{R}^d.
\end{equation}
Furthermore there exists constant $L$ such that $L_{|S|} \leq L < \infty$ for all $|S| \in \mathbb{N}$.
As a consequence, we have that $\nabla^2 f(w)\preceq L I $ for all $w \in \R^d$.
\end{assumption}
Before establishing theoretical convergence results, we state key inequalities due to the Assumption~\ref{assumption:hessian_spectral_bounds} that are used throughout the analysis.  From Assumption~\ref{assumption:hessian_spectral_bounds} and Taylor's theorem \cite{nocedal1999numerical,taylor1717methodus}, it follows that 
\begin{equation}\label{eq:desc_lemma}
    f(w) \leq f(v) + \nabla f(v)^T(w-v) + \frac{L}{2}\|w-v\|^2, \quad \forall w,v\in \R^d. 
\end{equation}
Furthermore, we also have that the path-averaged Hessians have upper bounded eigenvalues. That is,
due to \eqref{eq:htilde_def} and Assumption~\ref{assumption:hessian_spectral_bounds}, we have
\begin{align}
\widehat{H}_k\preceq \widehat{L}_kI, \enskip &\widehat{L}_k \leq \sum_{i=0}^k\gamma_i L_{|S_i|} \leq L, \label{eq:L_bnd_hat}\\
   \widetilde{H}_k \preceq \tilde{L} I, \enskip &\tilde{L} \leq \sum_{i=0}^{k}\gamma_i L_{|S_i|} + \tilde{\mu} - \lambda_{\min}(|\widehat{H}_k|) \leq L + \tilde{\mu}. \label{eq:L_bnd_tilde}
\end{align}

We begin our analysis by providing a technical lemma that establishes an upper bound on the difference between the objective function values at successive iterations.

\begin{lemma}\label{lem:tech_1}
Suppose Assumption~\ref{assumption:hessian_spectral_bounds} holds. For any $w_0$, let $\{w_k: k \in \N\}$ be iterates generated by \eqref{eq:iter} with the Hessian approximation given in \eqref{eq:htilde_def}.  If the step size $\alpha_k$ at each iteration $k$ is chosen such that $\alpha_k \leq \frac{\tilde{\mu}}{L}$. Then, for all $k \in \mathbb{Z}^+$, it follows that,
\begin{equation}\label{eq:desclemma}
    f(w_{k+1}) \leq f(w_k) - \frac{\alpha_k}{2}\nabla f(w_k)^T\widetilde{H}_k^{-1}\nabla f(w_k) + \frac{\alpha_k}{2}\delta_k^T \widetilde{H}_k^{-1} \delta_k, 
\end{equation}
where $\delta_k = g_k - \nabla f(w_k)$.
\end{lemma}

\begin{proof}
    Using \eqref{eq:desc_lemma} and the definition of $\delta_k$, we have
\begin{align}
    f(w_{k+1}) &\leq f(w_k) - \alpha_k \nabla f(w_k)^T\widetilde{H}_k^{-1} g_k + \frac{L\alpha_k^2}{2}\|\widetilde{H}_k^{-1}g_k\|^2 \nonumber \\
    &=  f(w_k) - \alpha_k \nabla f(w_k)^T \widetilde{H}_k^{-1}(\nabla f(w_k) + \delta_k) + \frac{L\alpha_k^2}{2}\|\widetilde{H}_k^{-1}(\nabla f(w_k) + \delta_k)\|^2 \label{eq:proof_part1} \\
    & = f(w_k) - \alpha_k \nabla f(w_k)^T \widetilde{H}_k^{-1} \nabla f(w_k) - \alpha_k \nabla f(w_k)^T\widetilde{H}_k^{-1}\delta_k \nonumber \\
    &\quad + \frac{L\alpha_k^2}{2}\left[\|\widetilde{H}_k^{-1} \nabla f(w_k)\|^2 + \|\widetilde{H}_k^{-1}\delta_k\|^2 + 2 \nabla f(w_k)^T\widetilde{H}_k^{-2} \delta_k \right] \nonumber \\
     &= f(w_k) - \alpha_k \nabla f(w_k)^T\widetilde{H}_k^{-1}\nabla f(w_k)+ \frac{L\alpha_k^2}{2}\left[\|\widetilde{H}_k^{-1} \nabla f(w_k)\|^2 + \|\widetilde{H}_k^{-1}\delta_k\|^2\right] \nonumber \\
     &\quad -\alpha_k (\widetilde{H}_k^{-1/2}\nabla f(w_k))^T\left[I - L\alpha_k \widetilde{H}_k^{-1}\right](\widetilde{H}_k^{-1/2}\delta_k). \label{eq:proof_ineq_1}
\end{align}
Substituting $I - L\alpha_k\widetilde{H}_k^{-1} \succeq 0$ due to $\alpha_k \leq \tfrac{\tilde{\mu}}{L}$ and using the fact that $-w^TAv \leq \tfrac{1}{2}w^TAw + \tfrac{1}{2}v^TAv$ for any $w,v\in \R^d$ and $0 \preceq A \in \R^{d \times d}$ in \eqref{eq:proof_ineq_1} yields
\begin{align}
    f(w_{k+1}) &\leq f(w_k) - \alpha_k \nabla f(w_k)^T\widetilde{H}_k^{-1}\nabla f(w_k)+ \frac{L\alpha_k^2}{2}\left[\|\widetilde{H}_k^{-1} \nabla f(w_k)\|^2 + \|\widetilde{H}_k^{-1}\delta_k\|^2\right] \nonumber \\
     &\quad +\frac{\alpha_k}{2}(\widetilde{H}_k^{-1/2}\nabla f(w_k))^T[I - L\alpha_k\widetilde{H}_k^{-1}](\widetilde{H}_k^{-1/2}\nabla f(w_k)) \nonumber \\
     &\quad +\frac{\alpha_k}{2}(\widetilde{H}_k^{-1/2}\delta_k))^T[I - L\alpha_k\widetilde{H}_k^{-1}](\widetilde{H}_k^{-1/2}\delta_k) \nonumber \\
     &= f(w_k) - \frac{\alpha_k}{2} \nabla f(w_k)^T\widetilde{H}_k^{-1}\nabla f(w_k)+ \frac{\alpha_k}{2} \delta_k^T\widetilde{H}_k^{-1}\delta_k. \nonumber 
\end{align}
 
\end{proof}

\subsection{Global Convergence}\label{sec:global}
In this section we derive global convergence rates. We first start with strongly convex functions, where we make the following assumption about the objective function.
\begin{assumption}{(Global strong convexity).} \label{assumption:strng_cnvx}
The eigenvalues of the Hessians are all positive and are bounded away from zero. That is, there exists constant $\mu > 0$ such that
\begin{equation}
	\nabla^2 f(w) \succeq \mu I.
\end{equation}  
\end{assumption}

From Assumption~\ref{assumption:strng_cnvx}, we have
\begin{equation}\label{eq:strngcnvx_grdbnd}
    \|\nabla f(w)\|^2 \geq 2\mu(f(w) - f(w^*)), \qquad \forall w\in\R^d,
\end{equation}
where $w^*$ is the unique optimal solution of \eqref{eq:prob_det} or \eqref{eq:prob_exp} (see \cite{bottou2018optimization} for the proof). 
We are now ready to provide the global linear convergence result for strongly convex functions.

\begin{theorem}{(Global linear convergence, deterministic sampling).}\label{thm:linear} Suppose Assumptions~\ref{assumption:hessian_spectral_bounds} and ~\ref{assumption:strng_cnvx} hold. For any $w_0\in\R^d$, let $\{w_k: k \in \N\}$ be iterates generated by \eqref{eq:iter} where the Hessian approximation is given in \eqref{eq:htilde_def} and the gradient approximations $g_k$ satisfies the Condition~\ref{condition:expected_gradient_bounds} with $\iota_{k+1} = \iota_k a_g$ for some $\iota_0 > 0$ and $a_g \in [0,1)$. Then, if $g_k$ satisfies deterministic norm condition \eqref{eq:deter_norm} with $A_k = \widetilde{H}_k^{-1}$ and $\theta_k  = \tilde{\theta}_g \in [0,1)$, and the step size is chosen such that $\alpha_k = \alpha \leq \frac{\tilde{\mu}}{L}$, 
 \begin{align}\label{eq:lin_det_reslt}
     f(w_{k}) - f(w^*) &\leq \tilde{C}_1\tilde{\rho}_1^k,\\
     \quad \tilde{C}_1 :=\max\left\{f(w_0) - f(w^*), \frac{\tilde{L}\iota_0}{\mu\left(1-\tilde{\theta}_g^2\right)} \right\}&, \mbox{ and } \tilde{\rho}_1 := \max\left\{1 - \frac{\alpha\mu\left(1-\tilde{\theta}_g^2\right)}{2\tilde{L}}, a_g\right\}\nonumber.
 \end{align}

\end{theorem}

\begin{proof}
From \eqref{eq:deter_norm} and \eqref{eq:desclemma}, we have 
\begin{align}
    f(w_{k+1}) &\leq f(w_k) - \frac{\alpha_k}{2} \nabla f(w_k)^T\widetilde{H}_k^{-1}\nabla f(w_k)+ \frac{\alpha_k}{2} \delta_k^T\widetilde{H}_k^{-1}\delta_k\nonumber \\
     &\leq f(w_k) - \frac{\alpha_k}{2} \nabla f(w_k)^T\widetilde{H}_k^{-1}\nabla f(w_k)+ \frac{\alpha_k \theta_k^2}{2} \nabla f(w_k)^T\widetilde{H}_k^{-1}\nabla f(w_k) + \frac{\alpha_k \iota_k}{2}\nonumber\\
    &\leq f(w_k) - \frac{\alpha_k(1 - \theta_k^2)}{2 \tilde{L}} \|\nabla f(w_k)\|^2 + \frac{\alpha_k \iota_k}{2}, \label{eq:gradnorm_det_reslt}
\end{align}
where the last inequality is due to \eqref{eq:L_bnd_tilde}. Subtracting $f(w^*)$ from both sides of \eqref{eq:gradnorm_det_reslt} and using \eqref{eq:strngcnvx_grdbnd}, it follows that 
\begin{equation}\label{eq:proof_det_ind}
    f(w_{k+1}) - f(w^*) \leq \left(1 - \alpha_k\frac{\mu(1-\theta_k^2)}{\tilde{L}} \right)\left(f(w_k) - f(w^*) \right) + \frac{\alpha_k \iota_k}{2}.
\end{equation}
We will use induction to show the rest of the proof. Substitute $\alpha_k = \alpha$ and $\theta_k = \tilde{\theta}_g$ in \eqref{eq:proof_det_ind}. 
We note that \eqref{eq:lin_det_reslt} trivially holds for $k=0$. Now, suppose that \eqref{eq:lin_det_reslt} holds for some $k$. Considering \eqref{eq:proof_det_ind}, we get,  
\begin{align*}
f(w_{k+1}) - f(w^*) &\leq \left(1 - \alpha\frac{\mu(1-\tilde{\theta}_g^2)}{\tilde{L}} \right) \tilde{C}_1\tilde{\rho}_1^k + \frac{\alpha \iota_0 a_g^k}{2}\\
&= \tilde{C}_1\tilde{\rho}_1^k\left(1 - \alpha\frac{\mu(1-\tilde{\theta}_g^2)}{\tilde{L}} + \frac{\alpha \iota_0}{2 \tilde{C}_1}\left(\frac{a_g}{\tilde{\rho}_1}\right)^k\right)\\
&\leq \tilde{C}_1\tilde{\rho}_1^k\left(1 - \alpha\frac{\mu(1-\tilde{\theta}_g^2)}{\tilde{L}} + \frac{\alpha \iota_0}{2 \tilde{C}_1}\right)\\
&\leq \tilde{C}_1\tilde{\rho}_1^k\left(1 - \alpha\frac{\mu(1-\tilde{\theta}_g^2)}{2\tilde{L}}\right) =  \tilde{C}_1\tilde{\rho}_1^{k+1},
\end{align*}
where the inequalities are due to the definitions of $\tilde{C}_1$ and $\tilde{\rho}_1$. 

\end{proof}
\begin{remark}\label{remark_glob_lin} We make the following remarks about the Theorem~\ref{thm:linear}
    \begin{itemize}
    \item For $\tilde{\theta}_g=0$ and $a_g=0$ ($\iota_k=0$), Theorem~\ref{thm:linear} recovers the classical global linear convergence result for Newton's method. We note that the rate constant $\left(1 - \frac{\alpha \mu}{2\tilde{L}}\right)$ is worse than that of steepest descent method, which is an artifact of global convergence analysis of Newton's method. As is typical with Newton methods, the global convergence bounds are more pessimistic than similar bounds for first-order methods due to taking into account worst-case spectral bounds of the Hessian. We note that in practice we do not expect significantly worse convergence rates for second-order methods, indeed our numerical results demonstrate that Hessian-averaged Newton methods are able to take larger steps than first-order methods.
\item We do not consider other cases of deterministic norm condition in this setting of $A_k = I$ as it enforces stringent restrictions on the choice of $\tilde{\theta}_g$.  
\end{itemize}
\end{remark}

Next we consider general nonconvex functions where we provide the following sublinear convergence results.

\begin{theorem}{(Global sublinear convergence, deterministic sampling).} \label{thm:global_sublinear_bounded_deterministic}
Suppose Assumption~\ref{assumption:hessian_spectral_bounds} holds and the objective function $f$ is bounded below by $f_\text{min}$. For any $w_0$, let $\{w_k: k \in \N\}$ be iterates generated by \eqref{eq:iter} 
where the gradient approximations $g_k$ satisfy the Condition~\ref{condition:expected_gradient_bounds}with $\sum_{i=0}^{\infty} \iota_k = \tilde{\iota} < \infty$. Then, if $g_k$ satisfies deterministic norm condition \eqref{eq:deter_norm} with  $A_k = \widetilde{H}_k^{-1}$ and $\theta_k  = \tilde{\theta}_g\in [0,1)$, and the step size is chosen such that $\alpha_k = \alpha \leq \frac{\tilde{\mu}}{L}$, then $\lim_{k\rightarrow \infty} \|\nabla f(w_k)\|^2 = 0$ and for any positive integer $T$, 
    \begin{equation}\label{eq:sublinear_det}
         \min\limits_{0\leq k \leq T-1} \|\nabla f(w_k)\|^2 \leq \frac{\tilde{L}}{(1-\tilde{\theta}_g^2)T}\left(\frac{2(f(w_0) - f_\text{min})}{\alpha} + \tilde{\iota}\right).
    \end{equation}

\end{theorem}
\begin{proof}
Substituting $\alpha_k=\alpha$ and $\theta_k=\theta_g$ in \eqref{eq:gradnorm_det_reslt}, rearranging the terms, and summing up the inequalities from $k=0$ to $T-1$ yields
\begin{align*}
\sum_{k=0}^{T-1}\|\nabla f(w_k)\|^2
&\leq \frac{\tilde{L}}{(1-\tilde{\theta}_g^2)}\left(\frac{2(f(w_0) - f(w_T))}{\alpha} + \sum_{k=0}^{T-1}\iota_k\right)\leq \frac{\tilde{L}}{(1-\tilde{\theta}_g^2)}\left(\frac{2(f(w_0) - f_\text{min})}{\alpha} + \tilde{\iota}\right),
\end{align*}
where the last inequality is due to $f(w_T) \geq f_\text{min}$ and $\sum_{k=0}^{\infty}\iota_k = \tilde{\iota}$. Taking the limits on $T$ yields $\lim_{k\rightarrow \infty} \|\nabla f(w_k)\|^2 = 0$. In addition, we have 
\begin{align*}
\min\limits_{0 \leq k \leq T-1} \|\nabla f(w_k)\|^2 &\leq \frac{1}{T} \sum_{k=0}^{T-1}\|\nabla f(w_k)\|^2 \leq \frac{\tilde{L}}{(1-\tilde{\theta}_g^2)T}\left(\frac{2(f(w_0) - f_\text{min})}{\alpha} + \tilde{\iota}\right).
\end{align*}

\end{proof}

\subsection{Local Convergence}
We now provide local superlinear rates of convergence results for the iterates generated by \eqref{eq:iter} when unit step size is eventually employed. We make the following standard assumption in local analysis of Newton-type methods that the Hessians are Lipschitz continuous. That is,     

\begin{assumption}{(Lipschitz continuous Hessians).} \label{assumption:lipschitz_hessian}
For any $|S| \in \mathbb{N}$, there exists a constant $0 < M_{|S|}<\infty$ such that
\begin{equation}
	\|\nabla^2 F_{S}(w) - \nabla^2 F_{S}(v)\| \leq M_{|S|}\|w-v \| \qquad \forall w,v \in \mathbb{R}^d.
\end{equation}
Furthermore, there exists constant $M$ such that $M_{|S|} < M < \infty$ for all $|S| \in \N$. As a consequence, we have that $\|\nabla^2 f(w) - \nabla^2 f(v)\| \leq M\|w -v\|$ for all $w,v \in \R^d$. 
\end{assumption}
We start by presenting a fundamental lemma that establishes a generalized linear-quadratic bound on the iterate distance to optimality when unit step size is employed at that iteration ($\alpha_k=1$). 

\begin{lemma}\label{lemma:tech_local}
Suppose Assumptions~\ref{assumption:hessian_spectral_bounds} and~\ref{assumption:lipschitz_hessian} hold. For any $w_0\in\R^d$, let $\{w_k: k \in \N\}$ be iterates generated by \eqref{eq:iter} where the Hessian approximation is given in \eqref{eq:htilde_def} and the gradient approximations $g_k$ satisfies the Condition~\ref{condition:expected_gradient_bounds} with $\lambda_{\min}(A_k) \geq \lambda_{A}$ and $\frac{\lambda_{\max}(A_k)}{\lambda_{\min}(A_k)} \leq \kappa_{A}$ for some positive constants $\lambda_A, \kappa_{A} < \infty$. If at any iteration $k \in \N$, unit step size is chosen ($\alpha_k=1$). Then, if $g_k$ satisfies deterministic norm condition \eqref{eq:deter_norm}  
    \begin{align}
    \|w_{k+1} - w^*\| &\leq \frac{M}{2\tilde{\mu}}\|w_k-w^*\|^2 + \frac{1}{\tilde{\mu}}\|(\widetilde{H}_k - \nabla^2 f(w_k))(w_k - w^*)\| \nonumber \\
    &\quad + \frac{L\sqrt{\kappa_{A}}}{\tilde{\mu}}\theta_k\|w_k - w^*\| + \frac{\sqrt{\iota_k}}{\tilde{\mu}\sqrt{\lambda_A}},\label{eq:twoterm_det}
    \end{align}

where $w^*$ is an optimal solution. 
\end{lemma}
\begin{proof}
We proceed by decomposing the iterate update \eqref{eq:iter}  into three terms: Newton update term, Hessian error term, and gradient error term as follows. 
\begin{align}
	&\|w_{k+1} - w^*\| \nonumber \\
 &\quad \leq \|\widetilde{H}_k^{-1}\|\|\widetilde{H}_k(w_k - w^*) - g_k\| \nonumber \\
 &\quad \leq \frac{1}{\tilde{\mu}} \left(\underbrace{\|\nabla^2 f(w_k)(w_k - w^*) - \nabla f(w_k)\|}_\text{Newton update} + \underbrace{\|(\widetilde{H}_k - \nabla^2 f(w_k))(w_k - w^*)\|}_{\text{Hessian error}} + \underbrace{\|g_k - \nabla f(w_k)\|}_{\text{gradient error}}\right). \label{eq:three_terms} 
\end{align}
The Newton update term is standard in Newton-type methods and has been analyzed in many prior works and we provide it here for the sake of completeness. Using Assumption~\ref{assumption:lipschitz_hessian}, it follows that
\begin{align}
    \|\nabla^2 f(w_k)(w_k - w^*) - \nabla f(w_k)\| &= \left\|\nabla^2 f(w_k)(w_k - w^*) - \int_{t=0}^{1}\nabla^2 f(w_k + t(w^* - w_k))(w_k - w^*)dt\right\| \nonumber \\
    &\leq \|w_k-w^*\|\int_{t=0}^{1}\|\nabla^2 f(w_k) - \nabla^2 f(w_k + t(w^* - w_k))\|dt \nonumber \\ 
    &\leq M\|w_k-w^*\|^2\int_{t=0}^{1}tdt  =\frac{M}{2}\|w_k-w^*\|^2. \label{eq:quad_term}
\end{align}
Next, we analyze the gradient error term. If $g_k$ satisfies deterministic norm condition \eqref{eq:deter_norm}, then we have that

\begin{align}
\sqrt{\lambda_{\min}(A_k)}\|g_k - \nabla f(w_k)\| &\leq \|g_k - \nabla f(w_k)\|_{A_k} \nonumber \\
&\leq \sqrt{\theta_k^2 \|\nabla f(w_k)\|_{A_k}^2 + \iota_k} \nonumber \\
&\leq \sqrt{\lambda_{\max}(A_k)}\theta_k\|\nabla f(w_k)\| + \sqrt{\iota_k}.
\end{align}
Rearranging the terms in the above inequality and using Assumption~\ref{assumption:hessian_spectral_bounds}, we get
\begin{equation}
    \|g_k - \nabla f(w_k)\| \leq L\sqrt{\frac{\lambda_{\max}(A_k)}{\lambda_{\min}(A_k)}}\theta_k\|w_k-w^*\| + \sqrt{\frac{\iota_k}{\lambda_{\min}(A_k)}}. \label{eq:grad_term}
\end{equation}
Combining \eqref{eq:three_terms}, \eqref{eq:quad_term}, and \eqref{eq:grad_term} yields \eqref{eq:twoterm_det}. 

\end{proof}

Lemma~\ref{lemma:tech_local} establishes the dependence of the iterate distance to optimality on the Hessian approximation error. In the following lemmas, we derive bounds for individual terms, in the service of establishing the local rate of convergence. We first decompose the error into different error terms as stated in the following lemma.   

\begin{lemma}\label{lem:hess_error}
Suppose Assumption~\ref{assumption:lipschitz_hessian} holds. For any iteration $k \in \mathbb{Z}^+$, the error in the Hessian approximation is upper bounded as
\begin{align}\label{eq:total_hessianerror}
    &\|(\widetilde{H}_k - \nabla^2 f(w_k))(w_k - w^*)\| \nonumber \\
    &\quad \leq \|(\widetilde{H}_k - \widehat{H}_k)(w_k - w^*)\| + 3M\|w_k - w^*\|^2 + \sum_{i=0}^k \gamma_i  \|(\nabla^2 F_{S_i} (w_i) - \nabla^2 F_{S_i}(w^*))(w_k - w^*)\| \nonumber \\
    &\qquad \quad 
    + \left\|\left(\sum_{i=0}^k \gamma_i\nabla^2 F_{S_i}(w^*) - \nabla^2f(w^*)\right)(w_k - w^*)  \right\|.
\end{align}
\end{lemma}
\begin{proof}
    Let $e_k = w_k - w^*$. The Hessian approximation error can be decomposed into two terms.
    \begin{align}\label{eq:hessian_error_terms}
    (\widetilde{H}_k - \nabla^2 f(w_k))e_k = \underbrace{(\widetilde{H}_k - \widehat{H}_k)e_k}_{\text{nonconvex error}} +  (\widehat{H}_k - \nabla^2 f(w_k))e_k,
    \end{align}
    where the first term is arising due to the nonconvexity. Now, using the decomposition of the second term given in \eqref{eq:hessian_error_decomposition}, we get
    \begin{align}\label{eq:proof_twoterms}
        &\|(\widehat{H}_k - \nabla^2 f(w_k))e_k\| \nonumber \\
        &\quad \leq \left\|\left(\sum_{i=0}^k \gamma_i \left(\nabla^2 F_{S_i}(w_i) - \nabla^2 F_{S_i}(w_k)\right)e_k\right)\right\| + \left\|\left(\sum_{i=0}^k \gamma_i\nabla^2 F_{S_i}(w_k) -\nabla^2 f(w_k)\right)e_k\right\|. 
    \end{align}
    Considering the first term in \eqref{eq:proof_twoterms}, using $\sum_{i=0}^{k}\gamma_i =1$ and Assumption~\ref{assumption:lipschitz_hessian}, we get
    \begin{align}
        &\left\|\left(\sum_{i=0}^k \gamma_i \left(\nabla^2 F_{S_i}(w_i) - \nabla^2 F_{S_i}(w_k)\right)\right)e_k\right\| \nonumber \\
        &\quad \leq \sum_{i=0}^k\gamma_i\|(\nabla^2 F_{S_i}(w_i) - \nabla^2 F_{S_i}(w^*))e_k\| + \|(\nabla^2 F_{S_i}(w^*) - \nabla^2 F_{S_i}(w_k))e_k\| \nonumber \\
        &\quad \leq \sum_{i=0}^k\gamma_i\|(\nabla^2 F_{S_i}(w_i) - \nabla^2 F_{S_i}(w^*))e_k\| + M\|w_k - w^*\|^2. \label{eq:proof_t1}
    \end{align}
    Considering the second term in \eqref{eq:proof_twoterms}, we get
    \begin{align}
        &\left\|\left(\sum_{i=0}^k \gamma_i\nabla^2 F_{S_i}(w_k) -\nabla^2 f(w_k)\right)e_k\right\| \nonumber \\
        &\quad \leq \left\|\sum_{i=0}^k \gamma_i(\nabla^2 F_{S_i}(w_k) - \nabla^2 F_{S_i}(w^*))e_k\right\| + \left\|\left(\sum_{i=1}^k \gamma_i\nabla^2 F_{S_i}(w^*) - \nabla^2f(w^*)\right)e_k  \right\| \nonumber \\
        &\quad \quad + \|(\nabla^2 f(w_k) - \nabla^2 f(w^*))e_k\|\nonumber \\
        &\quad \leq 2M\|w_k-w^*\|^2 + \left\|\left(\sum_{i=0}^k \gamma_i\nabla^2 F_{S_i}(w^*) - \nabla^2f(w^*)\right)e_k  \right\|. \label{eq:proof_t2}
    \end{align}
    Combining \eqref{eq:hessian_error_terms}, \eqref{eq:proof_twoterms}, \eqref{eq:proof_t1}, and \eqref{eq:proof_t2} yields \eqref{eq:total_hessianerror}. 
\end{proof}
We will analyze each term in the upper bound on the Hessian approximation error established in Lemma~\ref{lem:hess_error}. We achieve this by utilizing a key assumption in the local analysis of Newton methods where we assume that the iterates generated by \eqref{eq:iter} eventually enter a locally strongly convex regime. 
That is, we make the following assumption about the iterates generated by \eqref{eq:iter} with $\alpha_k$ specified in Section~\ref{sec:global} and $g_k$ satisfying deterministic norm condition \eqref{eq:deter_norm}. 

\begin{assumption}{(Local strong convexity).} \label{assum:local_strng_cnvx}
For any $w_0$, there exists $\nu > 0$ such that for all $k\in \{j \in \N \enskip |\enskip \|\nabla f(w_k)\| \leq \nu\}$, we have that $\nabla^2 f(w_k) \succeq \mu I$, where \{$w_k: k\in \N$\} are iterates generated by \eqref{eq:iter}. Moreover, we also assume that the following well-known inequalities associated with strong convexity also hold with respect to a local solution $w^*$ ($\|\nabla f(w^*)\| = 0$).
\begin{align}\label{eq:local_ineq_1}
    \|\nabla f(w_k)\|^2 &\geq 2\mu (f(w_k) - f(w^*)) \geq \mu^2 \|w_k - w^*\|^2. 
\end{align}
\end{assumption}

\begin{remark}
 We note that similar assumptions have been made in the constrained setting \cite[Assumption 5.1]{berahas2024modified}. This assumption is required not for all iterates, but only for those obtained after running the algorithm for a sufficiently large number of iterations, such that the iterates enter a locally strongly convex regime. This assumption is trivially satisfied when the functions are globally strongly convex (see Assumption~\ref{assumption:strng_cnvx}).
Due to the global convergence results established in Section~\ref{sec:global}, this assumption also implies that the iterates are indeed converging to a second-order stationary point ($\nabla f(w^*) = 0$ and $\nabla^2 f(w^*) \succ 0$). Such assumptions are commonly employed in local analysis of Newton-type methods, albeit in the form of proximity to a second-order stationary point $w^*$. That is, $\|w_k-w^*\| \leq \nu$. 
\end{remark}

In the next lemma, we establish upper bounds for the terms in  Lemma~\ref{lem:hess_error}. The main approach in the proof of this lemma is that using global convergence results, the iterates will enter the locally strongly convex regime after the global sublinear phase established in Theorem~\ref{thm:global_sublinear_bounded_deterministic}. Furthermore, once the iterates enters this phase, they will remain in this regime thereby achieving the linear convergence as established in Theorem~\ref{thm:linear}.   

\begin{lemma}\label{lemma:memory_det}
    Suppose Assumptions~\ref{assumption:hessian_spectral_bounds},~\ref{assumption:lipschitz_hessian} and ~\ref{assum:local_strng_cnvx} hold. For any $w_0\in\R^d$, let $\{w_k: k \in \N\}$ be iterates generated by \eqref{eq:iter} where the Hessian approximation is given in \eqref{eq:htilde_def} with $\gamma_i=\frac{1}{k+1}$ for all $i=0,\cdots,k$, and the gradient approximations $g_k$ satisfies the Condition~\ref{condition:expected_gradient_bounds} with $\sum_{i=0}^{\infty} \iota_k = \tilde{\iota} < \infty$. If $A_k$ and the corresponding step size $\alpha_k$ are chosen according to Theorem~\ref{thm:global_sublinear_bounded_deterministic}. Then there exists $k_{\text{lin}} \geq 0$ such that for any $k \geq k_{\text{lin}}$ if $\iota_{k+1} = \iota_{k} a_g$ for some $\iota_{k_\text{lin}} \geq 0$ and $a_g \in [0,1)$, we have that if $g_k$ satisfies deterministic norm condition \eqref{eq:deter_norm}, then there exists a constant $C_{p,d}$ such that
    \begin{align}\label{eq:lemma36_det}
        \sum_{i=0}^k \gamma_i  \|(\nabla^2 F_{S_i} (w_i) - \nabla^2 F_{S_i}(w^*))(w_k - w^*)\|&\leq \frac{C_{p,d}}{k+1}\|w_k - w^*\|.
    \end{align}

\end{lemma}
\begin{proof}
Since the conditions of Theorem~\ref{thm:global_sublinear_bounded_deterministic} are satisfied, from \eqref{eq:sublinear_det} we have that, for any positive integer $T$ with $\alpha_k=\alpha=\frac{\tilde{\mu}}{L}$, 
\begin{equation*}
         \min\limits_{0\leq k \leq T-1} \|\nabla f(w_k)\|^2 \leq \frac{\tilde{L}}{\left(1-\tilde{\theta}_g^2\right)T}\left(\frac{2(f(w_0) - f_\text{min})}{\alpha} + \tilde{\iota}\right).
\end{equation*}
Choosing 
\begin{equation}\label{eq:klin_d}
T\geq \tilde{k}_{\text{lin}} := \frac{\tilde{L}L}{\left(1-\tilde{\theta}_g^2\right)\mu\nu^2}\left(\frac{2(f(w_0) - f_\text{min})}{\alpha} + \tilde{\iota}\right),
\end{equation} 
we get
$$\min\limits_{0\leq k \leq T-1} \|\nabla f(w_k)\|^2 \leq \frac{\nu^2\mu}{L} \leq \nu^2.$$ 
Let $k_{\text{lin}}\leq \tilde{k}_{\text{lin}}$ be the first iterate at which $\|\nabla f(w_{k_{\text{lin}}})\|^2 \leq \nu^2$. Due to Assumption~\ref{assum:local_strng_cnvx}, it follows that $\nabla^2 f(w_{k_{\text{lin}}}) \succeq \mu I$. We will now show that starting with this iterate $w_{k_{\text{lin}}}$, all the following iterates will remain in this locally strongly convex phase. That is, we will show that $\|\nabla f(w_k)\| \leq \nu$ for all $k \geq k_{\text{lin}}$. We use induction to prove this statement. Note that for the base case of $k=k_{\text{lin}}$, it is trivially satisfied. Let us assume that the statement is true till some iteration $k-1>k_{\text{lin}}$. Using the strongly convex results established in Theorem~\ref{thm:linear} with $\alpha_k=\alpha=\frac{\tilde{\mu}}{L}$, starting with iterate $w_{k_{\text{lin}}}$, from \eqref{eq:lin_det_reslt}, we get,
\begin{align*}
    \|\nabla f(w_k)\|^2 &\leq 2L(f(w_k) - f(w^*))\leq 2L\tilde{C}_1\tilde{\rho}_1^{k - k_{\text{lin}}}\leq 2L\tilde{C}_1,
\end{align*}
where the first inequality is a well-known result for functions with Lipschitz continuous gradients \cite[Eq (3.16)]{bollapragada2018adaptive}.
Moreover, choosing $\iota_{k_{\text{lin}}} \leq \frac{\mu\left(1-\tilde{\theta}_g^2\right)\nu^2}{2L\tilde{L}}$, we have,
\begin{align}
    \|\nabla f(w_k)\|^2  \leq 2L\tilde{C}_1 &= 2L\max\left\{f(w_{k_{\text{lin}}}) - f(w^*), \frac{\tilde{L}\iota_{k_{\text{lin}}}}{\mu(1-\tilde{\theta}_g^2)} \right\} \nonumber \\
    &\leq 2L\max\left\{f(w_{k_{\text{lin}}}) - f(w^*), \frac{\nu^2}{2L} \right\} \nonumber \\
    &\leq 2L\max\left\{\frac{\|\nabla f(w_{k_{\text{lin}}})\|^2}{2\mu}, \frac{\nu^2}{2L} \right\} \leq \nu^2, \label{eq:proof_sub_-1}
\end{align}
where the third inequality is due to \eqref{eq:local_ineq_1} and the last inequality is due to $\|\nabla f(w_{k_{\text{lin}}})\|^2 \leq \frac{\nu^2\mu}{L}$. 

Therefore, for all $k\geq k_{\text{lin}}$, we conclude that $\|\nabla f(w_k)\|\leq \nu$ and consequently from Assumption~\ref{assum:local_strng_cnvx}, we have that 
$\nabla^2 f(w_k) \succeq \mu I$. 
Let $e_k = w_k - w^*$, and consider
\begin{align}
    \sum_{i=0}^k \gamma_i  &\|(\nabla^2 F_{S_i} (w_i) - \nabla^2 F_{S_i}(w^*))e_k\| \nonumber \\
    &=\frac{1}{k+1} \left(\sum_{i=0}^{k_{\text{lin}}-1}\|(\nabla^2 F_{S_i} (w_i) - \nabla^2 F_{S_i}(w^*))e_k\| + \sum_{i=k_{\text{lin}}}^{k}\|(\nabla^2 F_{S_i} (w_i) - \nabla^2 F_{S_i}(w^*))e_k\|\right) \nonumber \\
    &\leq \frac{2Lk_{\text{lin}}}{k+1}\|w_k - w^*\| + \frac{M}{k+1}\sum_{i=k_{\text{lin}}}^{k}\|w_i - w^*\|\|w_k - w^*\|, \label{eq:proof_sub_0}
\end{align}
where the first term in the inequality is due to $\|\nabla^2 F_{S_i}(\cdot)\| \leq L$ and the second term is due to Assumption~\ref{assumption:lipschitz_hessian}.
Now, starting with $k=k_{\text{lin}}$, the iterates are in locally strongly convex regime. Therefore, from \eqref{eq:lin_det_reslt}, \eqref{eq:local_ineq_1}, and \eqref{eq:proof_sub_-1}, we have for any $k \geq k_{\text{lin}}$,
\begin{align}\label{eq:proof_lin_1}
    \|w_k - w^*\|^2 \leq \frac{2}{\mu}(f(w_k) - f(w^*))\leq \frac{2\tilde{C}_1\tilde{\rho}_1^{k-k_{\text{lin}}}}{\mu} \leq \frac{\nu^2 \tilde{\rho}_1^{k - k_{\text{lin}}}}{\mu^2}.
\end{align}

Summing this inequality from $k_{\text{lin}}$ to $k$ yields,
\begin{align}\label{eq:proof_sum_1}
    \sum_{i=k_{\text{lin}}}^{k}\|w_i - w^*\| &\leq \sum_{i=k_{\text{lin}}}^{k}\frac{\nu(\sqrt{\tilde{\rho}_1})^{i-k_{\text{lin}}}}{\mu } 
    \leq \frac{\nu}{\mu} \sum_{i=0}^{k - k_{\text{lin}}} (\sqrt{\tilde{\rho}_1})^{i} < \frac{\nu}{\mu \left(1-\sqrt{\tilde{\rho}_1}\right)}.
\end{align}
Substituting \eqref{eq:proof_sum_1} in \eqref{eq:proof_sub_0} and choosing
\begin{equation}\label{eq:Cpd}
C_{p,d} := 2Lk_{\text{lin}} + \frac{\nu M}{\mu(1-\sqrt{\tilde{\rho}_1})}, 
\end{equation}
yields

\begin{align}
   \sum_{i=0}^k \gamma_i  \|(\nabla^2 F_{S_i} (w_i) - \nabla^2 F_{S_i}(w^*))e_k\| &\leq \frac{\|e_k\|}{k+1}\left(2Lk_{\text{lin}} + \frac{\nu M}{\mu\left(1-\sqrt{\tilde{\rho}_1}\right)}\right) = \frac{C_{p,d}}{k+1}\|w_k - w^*\|.
\end{align}

\end{proof}
\begin{remark}
If the functions are globally strongly convex (Assumption~\ref{assumption:strng_cnvx} holds) then there is no sublinear convergent phase in the algorithm. That is, $k_{\text{lin}} = 0$ in this setting. 
\end{remark}

We will now establish that the nonconvex error term in Lemma~\ref{lem:hess_error}  vanishes for sufficiently large number of iterations.

\begin{lemma}
    \label{lem:nonconvex}
    Suppose conditions of Lemma~\ref{lemma:memory_det} hold. Let $\tilde{\mu} \leq \frac{\mu}{2}$, and $g_k$ satisfies deterministic norm condition \eqref{eq:deter_norm} and the sample sets $S_k$ are chosen deterministically without replacement in a cyclic fashion such that $n|S_k| = N$ for some $n \in \N$ with $n\geq 1$. Let $\lambda_{\min}(\nabla^2 F_{S_i}(w_i)) \geq -\hat{\lambda}$ for some $\hat{\lambda} \in [0,\infty)$. Then there exists $k_{\text{non}} \geq k_{\text{lin}}$ such that for all $k \geq k_{\text{non}}$, $\widetilde{H}_k = \widehat{H}_k$. 
\end{lemma}
\begin{proof}

Any iteration $k\geq k_{\text{lin}}$ can be written as $k = k_{\text{lin}} + nm -1 + k_{\text{rem}}$ where $m \in \mathbb{Z}^+$, $k_{\text{rem}} \in \N$ and $1\leq k_{\text{rem}} \leq n-1$. Let $v \in \R^n$ be any vector and consider 
\begin{align}
    v^T\widehat{H}_kv &= \frac{1}{k+1}\sum_{i=0}^{k}v^T\nabla^2 F_{S_i}(w_i)v \nonumber \\
    &=\frac{1}{k+1}\sum_{i=0}^{k_{\text{lin}} -1}v^T\nabla^2 F_{S_i}(w_i)v + \frac{1}{k+1}\sum_{i=k_{\text{lin}}}^{k}v^T\nabla^2 F_{S_i}(w_i)v \nonumber \\
    &\geq \frac{-\hat{\lambda}k_{\text{lin}}}{k+1}\|v\|^2 + \frac{1}{k+1}\sum_{i=k_{\text{lin}}}^{k}v^T\nabla^2 F_{S_i}(w_i)v \label{eq:proof_noncon_0} \\
    &= \frac{-\hat{\lambda}k_{\text{lin}}}{k+1}\|v\|^2 + \frac{1}{k+1}\sum_{i=k_{\text{lin}}}^{k}v^T(\nabla^2 F_{S_i}(w_i) - \nabla^2 F_{S_i}(w^*))v + \frac{1}{k+1}\sum_{i=k_{\text{lin}}}^{k}v^T\nabla^2 F_{S_i}(w^*)v. \label{eq:proof_noncon_1} 
\end{align}
Using \eqref{eq:proof_sum_1}, we get
\begin{align}
    \frac{1}{k+1}\sum_{i=k_{\text{lin}}}^{k}v^T(\nabla^2 F_{S_i}(w_i) - \nabla^2 F_{S_i}(w^*))v &\geq - \frac{1}{k+1}\sum_{i=k_{\text{lin}}}^{k}\|\nabla^2 F_{S_i}(w_i) - \nabla^2 F_{S_i}(w^*)\|\|v\|^2 \nonumber \\
    &\geq \frac{-M\|v\|^2}{k+1}\sum_{i=k_{\text{lin}}}^{k}\|w_i - w^*\| \nonumber \\
    &\geq \frac{-M\nu}{(k+1)\mu(1 - \sqrt{\tilde{\rho}_1})}\|v\|^2. \label{eq:proof_noncon_2} 
\end{align}
Now consider
\begin{align}
    \frac{1}{k+1}&\sum_{i=k_{\text{lin}}}^{k}v^T\nabla^2 F_{S_i}(w^*)v \nonumber \\
    &=  \frac{1}{k+1}\sum_{i=k_{\text{lin}}}^{k_{\text{lin}} + nm - 1}v^T\nabla^2 F_{S_i}(w^*)v +  \frac{1}{k+1}\sum_{i=k_{\text{lin}} + nm }^{k_{\text{lin}} + nm - 1 + k_{\text{rem}}}v^T\nabla^2 F_{S_i}(w^*)v \nonumber \\
    &=\frac{nm}{k+1}v^T\nabla^2f(w^*)v + \frac{1}{k+1}\sum_{i=k_{\text{lin}} + nm }^{k_{\text{lin}} + nm -1 + k_{\text{rem}}}v^T\nabla^2 F_{S_i}(w^*)v \nonumber \\
    &\geq \frac{nm \mu}{k+1}\|v\|^2 - \frac{\hat{\lambda}k_{\text{rem}}}{k+1}\|v\|^2, \label{eq:proof_noncon_3}
\end{align}
where the second equality is due to the fact that sample sets $S_k$ are chosen deterministically without replacement in a cyclic fashion which implies that $\frac{1}{n}\sum_{i = j}^{j+n-1}\nabla^2 F_{S_i}(w^*) = \nabla^2 f(w^*)$ for any $j\in\N$. Let 
\begin{equation}\label{eq:k_non}
k_{\text{non}}:= \max\left\{4(1 + \tfrac{\hat{\lambda}}{\mu})(k_{\text{lin}} + n - 1), \tfrac{4M\nu}{\mu^2 (1 - \sqrt{\tilde{\rho}_1})}\right\} - 1.
\end{equation}
Combining \eqref{eq:proof_noncon_1}, \eqref{eq:proof_noncon_2}, \eqref{eq:proof_noncon_3}, and using \eqref{eq:k_non}, we get
\begin{align}
    \frac{v^T\widehat{H}_kv}{\|v\|^2} &\geq \frac{nm \mu}{k+1} -  \frac{\hat{\lambda}(k_{\text{lin}} + k_{\text{rem}})}{k+1} - \frac{M\nu}{(k+1)\mu(1 - \sqrt{\tilde{\rho}_1})} \nonumber \\
    &=\mu \left(\frac{k+1 -(k_{\text{lin}} + k_{\text{rem}})}{k+1} - \frac{\hat{\lambda}(k_{\text{lin}} + k_{\text{rem}})}{\mu(k+1)} - \frac{M\nu}{(k+1)\mu^2(1 - \sqrt{\tilde{\rho}_1})} \right) \nonumber \\
    &\geq \mu (1 - \tfrac{1}{4} -\tfrac{1}{4}) = \frac{\mu}{2} \geq \tilde{\mu}.
\end{align}
Therefore, $\lambda_{\min}(\widehat{H}_k) \geq \tilde{\mu}$ and $\widetilde{H}_k = \widehat{H}_k$ for all $k\geq k_{\text{non}}$. 

\end{proof}

To analyze the last term in Lemma~\ref{lem:hess_error}, we make the following standard assumption about individual Hessian components. 

\begin{assumption}{(Hessian approximations, deterministic case).}\label{assum:hessian_var}
The individual component Hessians are bounded relative to the Hessian of the objective function $f$ at the optimal solution $w^*$. That is, for the finite-sum problem, there exist constants $\beta_{1,H},\beta_{2,H} \geq 0$ such that
\begin{equation}\label{eq:det_hess_var_bnd}
    \|\nabla^2 F_i(w^*)\|^2 \leq \beta_{1,H}\|\nabla^2 f(w^*)\|^2 + \beta_{2,H}. 
\end{equation}
\end{assumption}

\begin{remark}
We note that Assumption~\ref{assum:hessian_var} is relatively weak compared to similar assumptions made in the literature \cite{bollapragada2019exact}, as it requires the Hessian components (or its variance) to be bounded only at the optimal solution instead at all iterates $w\in\R^d$. 
\end{remark}

\begin{lemma}\label{lem:sampling error}
    Suppose Assumptions~\ref{assumption:lipschitz_hessian} and~\ref{assum:hessian_var} hold. If $\gamma_i=\frac{1}{k+1}$ for all $i=0,\cdots,k$, and the sample sets $S_k$ are chosen deterministically without replacement in a cyclic fashion such that $n|S_k| = N$ for some $n \in \N$.
        \begin{align}\label{eq:lemma37_det}
            \left\|\left(\sum_{i=0}^k \gamma_i\nabla^2 F_{S_i}(w^*) - \nabla^2f(w^*)\right)(w_k - w^*)  \right\| &\leq \frac{C_{s,d}}{k+1}\frac{(n-1)^2}{n}\|w_k - w^*\|.
        \end{align}

\end{lemma}
\begin{proof}
Any iteration can be written as $k = nm - 1 + k_{\text{rem}}$ where $m \in \mathbb{Z}^{+}$, $k_{\text{rem}} \in \N$ and $1\leq k_{\text{rem}} \leq n-1$.  Consider, 
\begin{align}
     &\sum_{i=0}^k \gamma_i\nabla^2 F_{S_i}(w^*) - \nabla^2f(w^*) \nonumber \\
     &\quad = \frac{1}{k+1}\sum_{i=0}^{nm - 1} \nabla^2 F_{S_{i}}(w^*) - \frac{nm}{k+1}\nabla^2 f(w^*) + \frac{1}{k+1}\sum_{i=nm }^{k} \nabla^2 F_{S_i}(w^*) - \frac{k_{\text{rem}}}{k+1}\nabla^2 f(w^*) \nonumber \\
     &\quad= \frac{nm}{k+1}\left(\frac{1}{nm}\sum_{i=0}^{nm - 1} \nabla^2 F_{S_{i}}(w^*) - \nabla^2 f(w^*)\right) + \frac{k_{\text{rem}}}{k+1}\left(\frac{1}{k_{\text{rem}}}\sum_{i=nm}^{k} \nabla^2 F_{S_{i}}(w^*) - \nabla^2 f(w^*)\right) \nonumber \\
     &\quad= \frac{k_{\text{rem}}}{k+1}\left(\nabla^2 F_{S^\dagger}(w^*) - \nabla^2 f(w^*)\right) \leq \frac{n-1}{k+1}\left(\nabla^2 F_{S^\dagger}(w^*) - \nabla^2 f(w^*)\right), \label{proof:sampling_term_det}
\end{align}
where the third equality is due to the fact that $\frac{1}{n}\sum_{i=0}^{n-1}\nabla^2 F_{S_i}(w^*) = \nabla^2 f(w^*)$ and $S^\dagger = \cup_{i=(n-1)m + 1}^k S_i$. Following a similar approach in establishing the deterministic bounds on gradient approximation error given in \cite[Section 3.1]{friedlander2012hybrid} (which we restate in Appendix \ref{appendix:analysis} for completeness) and using $|S^\dagger| \geq |S_k|$, we get  
\begin{align}\label{eq:deterministic_hess_sample_bound}
    \|\nabla^2 F_{S^\dagger}(w^*) - \nabla^2 f(w^*)\|^2 &\leq 4\left(\frac{N - |S^\dagger|}{N}\right)^2\left(\beta_{1,H}\|\nabla^2 f(w^*)\|^2 + \beta_{2,H}\right) \nonumber \\
    &\leq 4\left(1 - \frac{1}{n}\right)^2\left(\beta_{1,H}L^2 + \beta_{2,H}\right).
\end{align}
Substituting \eqref{eq:deterministic_hess_sample_bound} in \eqref{proof:sampling_term_det} and choosing \begin{equation}\label{eq:Csd}
C_{s,d}:= 2\sqrt{\beta_{1,H}L^2 + \beta_{2,H}}, 
\end{equation}yields 
\begin{align*}
    \left\|\left(\sum_{i=0}^k \gamma_i\nabla^2 F_{S_i}(w^*) - \nabla^2f(w^*)\right)(w_k - w^*)  \right\| 
    &\leq \frac{2\sqrt{\beta_{1,H}L^2 + \beta_{2,H}}(n-1)^2}{(k+1)n}\|w_k - w^*\| \\
    &= \frac{C_{s,d}}{k+1}\frac{(n-1)^2}{n}\|w_k - w^*\|.
\end{align*}
	 
\end{proof}
\begin{remark}
	Lemma~\ref{lem:sampling error} establishes the bound on the sampling error in terms of the sample size $|S_0|$ and the iteration number $k$. Deterministic sampling without replacement (cyclic manner) has rate $\mathcal{O}\left(\tfrac{1}{k}\right)$, instead of the $\mathcal{O}\left(\tfrac{1}{\sqrt{k}}\right)$ rate for stochastic Hessians; this result is proven probabilistically in \cite{na2023hessian,jiang2024stochastic}, and in expectation in the next section. Moreover, the sampling error becomes zero after finishing every cycle in the deterministic sampling case. 
\end{remark}

Before proceeding to the main result, we provide the following technical lemma that establishes the linear and superlinear convergence of generic sequences where each term is bounded above by the previous term in a relaxed linear-quadratic manner.

\begin{lemma}\label{lemma:superlinear}
	Suppose $\{z_k: k \in \mathbb{Z}^+\}$ is a non-negative sequence that satisfies
	\begin{equation}\label{eq:lemma_seq}
    z_{k+1} \leq qz_k^2 + \tau_{k}z_k + o_k
    \end{equation}
    for any given non-negative constant $q$, non-negative sequence $\{\tau_k: k \in \mathbb{Z}^+\}$ with $\tau_{k+1}\leq \tau_k$ for all $k \in \N$, and non-negative sequence $\{o_k: k \in \mathbb{Z}^+\}$. If 
    $z_0 \leq  \tfrac{\upsilon}{3q}$, $\tau_0 \leq  \tfrac{\upsilon}{3}$ and $o_0 \leq \tfrac{\upsilon^2}{9 q}$ and $o_{k+1} = o_k \upsilon t_{k+1}^2$  where $\upsilon \in [0,1)$ is any given constant and $t_k$ is a non-negative sequence with $t_1\leq 1$ and $t_{k+1} \leq t_k$ for all $k \in \N$.
    Then, for all $k\in \N$,
    \begin{align}\label{eq:proof_lemma_tech0_sup}
         z_{k} \leq r_{k}, \quad r_{k+1} = \max\left\{ r_k \rho_k, r_0\upsilon^{k+1} \prod_{i=0}^{k}t_{i+1}\right\}, \quad r_0=\max\{z_0, \tfrac{3o_0}{\upsilon}\}, \quad \rho_k = qr_k + \tau_k + \tfrac{o_0}{r_0}\prod_{i=0}^{k-1}t_{i+1} \in [0,\upsilon].
    \end{align}
     Therefore, $z_k \rightarrow 0$ at an R-linear rate. Furthermore, if $\lim\limits_{k \rightarrow \infty} t_k = 0$ then $z_k \rightarrow 0$ at an R-superlinear rate.
\end{lemma}
\begin{proof}
Let $b_{k+1} = b_k t_{k+1}$ for all $k\in \N$ with $b_0 = \tfrac{o_0}{r_0} \leq \tfrac{\upsilon}{3}$. We have that $b_k \leq b_0 \leq \tfrac{\upsilon}{3}$ and $\tfrac{o_{k+1}}{b_{k+1}} = \tfrac{o_{k}}{b_k}\upsilon t_{k+1} \leq \tfrac{o_{k}}{b_k}$ for all $k \in \N$. We will use induction to prove that 
\begin{align}\label{eq:proof_lemma_tech_gen}
z_{k} \leq r_{k}, \quad r_{k+1} = \max\left\{ r_k \rho_k, \tfrac{o_{k+1}}{b_{k+1}}\right\}, \quad r_0=\max\{z_0, \tfrac{3o_0}{\upsilon}\}, \quad \rho_k = qr_k + \tau_k + b_k \in [0,\upsilon].
\end{align}
Note that the base case of $k=0$ is trivially satisfied since $z_0 \leq r_0$. Suppose that this result is true for some $k$.  From \eqref{eq:lemma_seq}, we get
\begin{align*}
    z_{k+1} &\leq qz_k^2 + \tau_{k}z_k + o_k \leq r_k \left(q r_k + \tau_k + \frac{o_k}{r_k}\right) \leq r_k (q r_k + \tau_k + b_k) = r_k \rho_k\leq r_{k+1}.
\end{align*}
Next, we will use induction to show that $r_k \leq \frac{\upsilon}{3q}$ and $\rho_k \leq \upsilon$ for all $k\in \N$. Note the base case of $k=0$ is satisfied since $r_0 \leq \frac{\upsilon}{3q}$ and $\rho_0 = q r_0 + \tau_0 + b_0 \leq \upsilon$. Let us assume that this result is true for some $k$. Consider,
\begin{align*}
    r_{k+1} &=  \max\{ r_k \rho_k, \tfrac{o_k}{b_k}\} \leq \max\{ r_k, \tfrac{o_0}{b_0}\} \leq \tfrac{\upsilon}{3q} \nonumber \\
    \rho_{k+1} &= qr_{k+1} + \tau_{k+1} + b_k \leq \tfrac{\upsilon}{3} + \tau_0 + \tfrac{\upsilon}{3} \leq \upsilon.
\end{align*}
Therefore, from \eqref{eq:proof_lemma_tech_gen}, we have that
\begin{align*}
    \frac{r_{k+1}}{r_k} =  \max\left\{\rho_k, \frac{o_{k+1}}{b_{k+1}r_k}\right\} \leq \max\left\{\rho_k, \frac{o_{k+1}b_k}{b_{k+1}o_k}\right\} = \max\left\{\rho_k, \upsilon t_{k+1}\right\} \leq \upsilon < 1.   
\end{align*}
Hence, $r_k \rightarrow 0$ at a Q-linear rate and consequently $z_k \rightarrow 0$ at an R-linear rate. Furthermore, if $\lim\limits_{k \rightarrow \infty} t_k = 0$ and $\lim\limits_{k \rightarrow \infty} \tau_k = 0$, then 
\begin{align}\label{eq:lin_constant}
    \lim_{k \rightarrow \infty} \rho_k \leq \lim_{k \rightarrow \infty} qr_k + \tau_k + \tfrac{o_0}{r_0}t_1^k = 0. 
\end{align}
Therefore, 
\begin{align*}
    \lim_{k \rightarrow \infty} \frac{r_{k+1}}{r_k} = \lim_{k \rightarrow \infty} \max\left\{\rho_k, \upsilon t_{k+1}\right\} = 0.
\end{align*}    
Hence, $r_k \rightarrow 0$ at a Q-superlinear rate and consequently $z_k \rightarrow 0$ at an R-superlinear rate.
\end{proof}

\newpage

We are now ready to provide the main theoretical results in this section. 
\begin{theorem} {(Deterministic local linear and superlinear convergence).}\label{thm:super_det}
Suppose Assumptions~\ref{assumption:hessian_spectral_bounds}, ~\ref{assumption:lipschitz_hessian},~\ref{assum:local_strng_cnvx}, and ~\ref{assum:hessian_var} hold. For any $w_0\in\R^d$, let $\{w_k: k \in \N\}$ be iterates generated by \eqref{eq:iter} where the Hessian approximation is given in \eqref{eq:htilde_def} with $\tilde{\mu} \leq \tfrac{\mu}{2}$, the sample sets $S_k$ chosen deterministically without replacement in a cyclic fashion such that $n|S_k|=N$ for some $n\in\N$ with $n\geq1$, and $\gamma_i=\frac{1}{k+1}$ for all $i=0,\cdots,k$. Furthermore, the gradient approximation $g_k$ satisfies the deterministic norm condition \eqref{eq:deter_norm} with $A_k = \widetilde{H}_k^{-1}$ for all $k \in \N$, $\sum_{i=0}^{\infty}\iota_i = \tilde{\iota} < \infty$, and $\alpha_k =\alpha \leq \frac{\tilde{\mu}}{L}$ for all $k < k_{\text{sup}}$ for some $k_{\text{sup}} \in \N$. Furthermore, for all $k_{\text{lin}}\leq k < k_{sup}$, $\iota_{k+1} = \iota_{k} a_g$ for some $\iota_{k_\text{lin}} \geq 0$ and $a_g \in [0,1)$, where $k_{\text{lin}}$ is given in Lemma~\ref{lemma:memory_det}. Let $\tilde{\theta}_l = \frac{\sqrt{a_l}}{6}\left(\tfrac{\tilde{\mu}}{\tilde{L}}\right)^{3/2}$, $q= \tfrac{7M}{2\tilde{\mu}}$, and $\tau_k := \frac{1}{\tilde{\mu}}\left(\frac{nC_{p,d} + C_{s,d}(n-1)^2}{n(k+1)} + \theta_{k}L\sqrt{\tfrac{\tilde{L}}{\tilde{\mu}}}\right)$ for some $a_l \in [0,1)$, and let $r_k$ be a sequence such that $r_{k_{\text{sup}}}=\max\left\{\|w_{k_{\text{sup}}} - w^*\|, \tfrac{\sqrt{\tilde{L}\iota_{k_{\text{sup}}}}}{\tilde{\mu}}\right\} \leq \tfrac{\sqrt{a_l}}{3q}$.
\begin{enumerate}
\item If $\alpha_k= 1$, $\theta_k = \tilde{\theta}_l$, and $\iota_{k+1} = \iota_{k}a_l$ for all $k \geq k_{\text{sup}}$. Then 
\begin{align}\label{eq:loc_lin_det}
        \|w_{k} - w^*\|  \leq r_{k}, \quad r_{k+1} = \max\left\{ r_k \rho_k, r_{k_{\text{sup}}}(\sqrt{a_l})^{k - k_{\text{sup}} +1} \right\}, \quad \rho_k = qr_k + \tau_k + \tfrac{\sqrt{\tilde{L}\iota_{k_{\text{sup}}}}}{r_{k_{\text{sup}}}\tilde{\mu}} \in [0,\sqrt{a_l}].
\end{align}
Therefore, $\|w_{k} - w^*\| \rightarrow 0$ at an R-linear rate with rate constant  upper bounded by $\sqrt{a_l} \in [0,1)$.
\item If $\alpha_k=1$, $\theta_k = \frac{\tilde{\theta}_l}{k+1}$, and $\iota_{k+1} = \iota_{k}a_l t_{k+1}^4$, $t_k = \tfrac{1}{k}$  for all $k \geq k_{\text{sup}}$. Then 
\begin{align}
\|w_{k} - w^*\| &\leq r_{k}, \quad r_{k+1} = \max\left\{ r_k \rho_k, r_{k_{\text{sup}}}(\sqrt{a_l})^{k-k_{\text{sup}}+1} \prod_{i=k_{\text{sup}}}^{k}t_{i+1}\right\}, \nonumber \\
&\rho_k = qr_k + \tau_k + \tfrac{\sqrt{\tilde{L}\iota_{k_{\text{sup}}}}}{r_{k_{\text{sup}}}\tilde{\mu}}\prod_{i=k_{\text{sup}}}^{k-1}t_{i+1} = \mathcal{O}\left(\tfrac{1}{k}\right)\in [0,\sqrt{a_l}]. \label{eq:loc_suplin_det}
\end{align}
Therefore, $\|w_{k} - w^*\| \rightarrow 0$ at an R-superlinear rate with rate constant upper bounded by $\max\{\rho_k, \sqrt{a_l}t_{k+1}\} =\mathcal{O}\left(\tfrac{1}{k}\right)$. 
\end{enumerate}
\end{theorem}
\begin{proof}
Let 
\begin{align}\label{eq:ksup_det}
k_{\text{sup}} = \bigg\lceil \max \bigg\{&4\left(1 + \tfrac{\hat{\lambda}}{\mu}\right)(k_{\text{lin}} + n - 1), \tfrac{4M\nu}{\mu^2 (1 - \sqrt{\tilde{\rho}_1})}, \tfrac{6(nC_{p,d} + C_{s,d}(n-1)^2)}{n\tilde{\mu} \sqrt{a_l}},  \nonumber \\
&~~~k_{\text{lin}} + 2\log_{1/\tilde{\rho}_1}\left(\tfrac{3q\nu}{\mu \sqrt{a_l}}\right), k_{\text{lin}} + \log_{1/a_g}\left(\tfrac{81\tilde{L}q^2 \iota_{k_{\text{lin}}}}{a_l^2\tilde{\mu}^2}\right)\bigg\}\bigg\rceil, 
\end{align}
where $k_{\text{lin}}$ is defined in \eqref{eq:klin_d}, $\hat{\lambda}$ is defined in Lemma~\ref{lem:nonconvex}, $C_{p,d}, C_{s,d}$ are given in \eqref{eq:Cpd} and \eqref{eq:Csd} respectively, and $\tilde{\rho}_1$ is the linear convergence rate defined in Theorem~\ref{thm:linear}.
We note that $k_{\text{sup}} \geq k_{\text{non}}$ due to the first and second terms in \eqref{eq:ksup_det}  where $k_{\text{non}}$ is defined in \eqref{eq:k_non}. 
Using ~\eqref{eq:twoterm_det},~\eqref{eq:total_hessianerror},~\eqref{eq:lemma36_det}, and~\eqref{eq:lemma37_det}, we get for all $k \geq k_{\text{sup}}$,
\begin{align}
    \|w_{k+1} - w^*\| &\leq \frac{7M}{2\tilde{\mu}}\|w_k - w^*\|^2 + \frac{1}{\tilde{\mu}}\left(\frac{nC_{p,d} + C_{s,d}(n-1)^2}{n(k+1)} + \theta_kL\sqrt{\frac{\tilde{L}}{\tilde{\mu}}}\right)\|w_k - w^*\| + \frac{\sqrt{\tilde{L}\iota_k}}{\tilde{\mu}} \nonumber \\
    &= q\|w_{k} - w^*\|^2 + \tau_k \|w_{k} - w^*\| + o_k, \nonumber 
\end{align}
where 
$o_k := \frac{\sqrt{\tilde{L}\iota_k}}{\tilde{\mu}}$. Using the third term in \eqref{eq:ksup_det}, $L\leq \tilde{L}$, and $\theta_k =\tilde{\theta}_l = \frac{\sqrt{a_l}}{6}\left(\frac{\tilde{\mu}}{\tilde{L}}\right)^{3/2}$, we get $\tau_k \leq \frac{\sqrt{a_l}}{3}$ for $k = k_{\text{sup}}$. 
In addition, using the fourth term in \eqref{eq:ksup_det} and \eqref{eq:proof_lin_1}, we get $\|w_k - w^*\| \leq \frac{\sqrt{a_l}}{3q}$
for $k = k_{\text{sup}}$. Moreover, $\iota_{k_{\text{sup}}} = \iota_{k_{\text{lin}}} a_g^{k_{\text{sup}}-k_{\text{lin}}} \leq \tfrac{a_l^2\tilde{\mu}^2} {81\tilde{L}q^2}$ due to the fifth term in \eqref{eq:ksup_det} which implies that $o_{k} \leq \tfrac{a_l}{9q}$ for $k=k_{\text{sup}}$.

Therefore, starting with $k=k_{\text{sup}}$ and using Lemma~\ref{lemma:superlinear} with $\upsilon = \sqrt{a_l}$ and $\iota_{k+1} = \iota_k a_l$ yields \eqref{eq:loc_lin_det}. Moreover, from \eqref{eq:lin_constant}, we have that $\tfrac{r_{k+1}}{r_k} \leq \sqrt{a_l}<1$. 

Similarly starting with $k=k_{\text{sup}}$ and using Lemma~\ref{lemma:superlinear} with $\upsilon = \sqrt{a_l}$ and $\iota_{k+1} = \tfrac{\iota_k a_l}{(k+1)^4}$ yields \eqref{eq:loc_suplin_det}. Moreover, from \eqref{eq:lin_constant}, we get
\begin{align*}
\tfrac{r_{k+1}}{r_k}\leq \max\left\{\rho_k, \sqrt{a_l} t_{k+1}\right\}&\leq \max\left\{ qr_k + \tau_k + \tfrac{o_{k_{\text{sup}}}}{r_0}\prod_{i=k_{\text{sup}}}^{k-1}t_{i+1},\tfrac{\sqrt{a_l}}{k+1} \right\} \nonumber \\
&< \max\left\{ qr_{k_{\text{sup}}}(\sqrt{a_l})^{k-k_{\text{sup}}} + \tfrac{\tau_{k_{\text{sup}}}(k_{\text{sup}} + 1)}{k+1} + \tfrac{o_{k_{\text{sup}}}}{r_{k_{\text{sup}}} (k+1)},\tfrac{\sqrt{a_l}}{k+1} \right\} = \mathcal{O}\left(\tfrac{1}{k}\right).
\end{align*}
\end{proof}

\begin{remark}
    We make the following remarks about this result. 
    \begin{itemize} 
        \item \textbf{Case:} $\theta_k = 0$, $\iota_k = 0, \enskip   (t_k=0)$ (exact gradient). When exact gradients are employed, using Lemma~\ref{lemma:superlinear}, we get deterministic Q-superlinear convergence where the rate constant is $\mathcal{O}\left(\tfrac{1}{k}\right)$. We note that this is an improvement over the final phase superlinear convergence results established in probability established for stochastic Hessian sampling with mean zero sub-exponential Hessian noise where the rate constant is $\mathcal{O}\left(\sqrt{\tfrac{\log{k}}{k}}\right)$ \cite{na2023hessian,jiang2024stochastic}.   
        
        \item \textbf{Case:} $\theta_k \neq 0$, $\iota_k=0$ (inexact adaptive gradient). When inexact gradients are employed where the gradient accuracies are chosen solely relative to the gradient norm itself, we get Q-linear and Q-superlinear convergence results based on the choice of the $\theta_k$ parameter.

        \item \textbf{Rate constant}. We note that the local linear convergence rate constant ($\sqrt{a_l} \in [0,1)$) is a hyperparameter that doesn't depend on the problem characteristics. Therefore, this local linear convergence result is better than global linear convergence results that typically depend on the condition number of the problem. This result is similar to other local linear convergence results established in the literature \cite{bollapragada2019exact, roosta2019sub}, although those results are established either in probability or in expectation as opposed to the deterministic result presented here. 

        \item \textbf{Step size}. Two different step sizes are chosen in the the global phase ($\alpha_k = \tfrac{\tilde{\mu}}{L}$) and local superlinear phase ($\alpha_k = 1$). While such two phase approaches are common for Newton-type methods \cite{bollapragada2019exact}, these results can be unified using an inexact line search approach that employs inexact function evaluations where the unit step size is automatically selected in the second (local) phase with an appropriately modified Armijo sufficient decrease condition \cite{bollapragada2018progressive,paquette2020stochastic}. 
     
    \end{itemize}
\end{remark}
We will now characterize the number of iterations required to transition from one convergent phase of the algorithm to the other. For the sake of simplicity and to make it possible to compare our results with other existing results in the literature, we will only consider global strongly convex functions (see Assumption~\ref{assumption:strng_cnvx}). Therefore, the algorithm only encounters two phases: Global linear convergence and local linear or superlinear convergence depending on the choice of gradient accuracies as established in Theorem~\ref{thm:super_det}. It is possible to account for global sublinear convergence phase too by analyzing $k_{\text{lin}}$ given in \eqref{eq:klin_d}. However, for strongly convex functions $k_{\text{lin}} = 0$. 
\begin{corollary}\label{cor:ksup_d}
    Suppose Assumption~\ref{assumption:strng_cnvx} and conditions of Theorem~\ref{thm:super_det} hold where we choose $a_g = 1 - \tfrac{\mu \tilde{\mu}}{2L\tilde{L}} \in [0,1)$. Then, the number of iterations required for the iterates to reach local linear or superlinear convergence phase is given as 
    \begin{align}\label{eq:order_ksup}
        k_{\text{sup}} = \tilde{\mathcal{O}}\left(\max\left\{\tfrac{N}{|S_0|}(1 +\tfrac{ \hat{\lambda}}{\mu}), \kappa^2(1 + \tfrac{M}{\mu^{3/2}})\right\}\right),
    \end{align}
    where $\hat{\lambda}$ is such that $\lambda_{\min}(\nabla^2 F_{S_i}(w_i))\succeq -\hat{\lambda}$ for all $i \in \N$  and  $\kappa \leq \tfrac{\tilde{L}}{\tilde{\mu}}$. 
\end{corollary}
\begin{proof}
Since the function is strongly convex, we no longer require the iterates to enter the basin where $\|\nabla f(w_k)\|^2 \leq \nu^2$ for invoking local strong convex properties. Therefore, \eqref{eq:proof_sum_1} is updated using \eqref{eq:lin_det_reslt} as
\begin{align}\label{eq:proof_update_sum_1}
    \sum_{i=0}^{k}\|w_i - w^*\| \leq \sqrt{\frac{2}{\mu}}\sqrt{(f(w_k) - f(w^*))}\leq \sqrt{\frac{2\tilde{C}_1}{\mu}}\sum_{i=0}^{k}(\sqrt{\tilde{\rho}_1})^{i} < \frac{\sqrt{2\tilde{C}_1}}{\sqrt{\mu}(1 - \sqrt{\tilde{\rho}_1})}.
\end{align}
Using $k_{\text{lin}} = 0$, \eqref{eq:proof_update_sum_1} is updated as $
C_{p,d} := \frac{M\sqrt{2\tilde{C}_1}}{\sqrt{\mu}(1-\sqrt{\tilde{\rho}_1})}$. Using these update formulae, we get 
\begin{align}\label{eq:ksup_det_update}
k_{\text{sup}} = \bigg\lceil \max \bigg\{&4(1 + \tfrac{\hat{\lambda}}{\mu})(n - 1), \tfrac{4M\sqrt{2\tilde{C}_1}}{\mu^{3/2} (1 - \sqrt{\tilde{\rho}_1})}, \tfrac{6(nC_{p,d} + C_{s,d}(n-1)^2)}{n\tilde{\mu} \sqrt{a_l}},  \nonumber \\
&~~~2\log_{1/\tilde{\rho}_1}\left(\tfrac{3\sqrt{2}q\sqrt{\tilde{C}_1}}{\sqrt{\mu a_l}}\right), \log_{1/a_g}\left(\tfrac{81\tilde{L}q^2 \iota_{0}}{a_l^2\tilde{\mu}^2}\right)\bigg\}\bigg\rceil, 
\end{align}
Now, consider the second term in \eqref{eq:ksup_det_update}, we note that
\begin{align}\label{eq:term2_det}
    \tfrac{4M\sqrt{2\tilde{C}_1}}{\mu^{3/2} (1 - \sqrt{\tilde{\rho}_1})} \leq \tfrac{8M\sqrt{2\tilde{C}_1}}{\mu^{3/2} (1 - \tilde{\rho}_1)} \leq \tfrac{16\sqrt{2} \tilde{C}_1L^2 M}{\mu^{7/2}} = \mathcal{\tilde{O}}(\tfrac{\kappa^2 M}{\mu^{3/2}}).
\end{align}
The third term in \eqref{eq:ksup_det_update} is given as, 
\begin{align}\label{eq:term3_det}
    \tfrac{6(nC_{p,d} + C_{s,d}(n-1)^2)}{n\tilde{\mu} \sqrt{a_l}} &= \tfrac{6M\sqrt{2\tilde{C}_1}}{\mu^{3/2}(1-\sqrt{\tilde{\rho}_1})\sqrt{a_l}} + \tfrac{12\sqrt{\beta_{1,H}L^2 + \beta_{2,H}}(n-1)^2}{n\tilde{\mu}\sqrt{a_l}} = \mathcal{\tilde{O}}(\tfrac{\kappa^2 M}{\mu^{3/2}}).
\end{align}
Now considering the last two terms in \eqref{eq:ksup_det_update} and using $a_g = 1 - \tfrac{\mu \tilde{\mu}}{2L\tilde{L}}$, $\log(1-x) \approx -x$ for small $x \in (0,1)$, we have that
\begin{align}\label{eq:term45_det}
\max \bigg\{
2\log_{1/\tilde{\rho}_1}\left(\tfrac{3\sqrt{2}q\sqrt{\tilde{C}_1}}{\sqrt{\mu a_l}}\right), \log_{1/a_g}\left(\tfrac{81\tilde{L}q^2 \iota_{0}}{a_l^2\tilde{\mu}^2}\right)\bigg\} &= \tfrac{2L\tilde{L}}{\mu \tilde{\mu}} \max \bigg\{
2\log\left(\tfrac{3\sqrt{2}q\sqrt{\tilde{C}_1}}{\sqrt{\mu a_l}}\right), \log\left(\tfrac{81\tilde{L}q^2 \iota_{0}}{a_l^2\tilde{\mu}^2}\right)\bigg\} = \mathcal{\tilde{O}}(\kappa^2).  
\end{align}
Combining \eqref{eq:term2_det}, \eqref{eq:term3_det}, \eqref{eq:term45_det}, and using $n = \tfrac{N}{|S_0|}$ yields the desired result. 
\end{proof}

\begin{remark}
    We note that in the case where the subsampled functions are convex, i.e. $\lambda_{\min}(\nabla^2 F_{S_i}(x_i)) \succeq 0$, we have $\hat{\lambda} = 0$. Therefore, 
    \begin{equation}\label{eq:ksup_cnvx}
        k_{\text{sup}} = \tilde{\mathcal{O}}\left(\max\left\{\tfrac{N}{|S_0|}, \kappa^2(1+\tfrac{M}{\mu^{3/2}})\right\}\right).
    \end{equation}
    These transition phases are better than the final transition phases established for uniform weighted scheme when $N < \kappa^6$ and comparable to nonuniform weighted scheme when $N = \mathcal{O}(\kappa^2)$ and $\tfrac{M}{\mu^{3/2}}$ is small in \cite{na2023hessian}. Furthermore, one could speed up the transition to local phase by modifying Newton's method using proximal extra gradient methods, at additional per-iteration costs \cite{jiang2024stochastic}.
\end{remark}

%% file: sections/tikz_diagrams/analysis_roadmap_det.tex
\begin{figure}[H]
\center
\begin{tikzpicture}[scale = 0.95, transform shape,every node/.style={draw,outer sep=0pt,thick},box/.style={draw, rectangle, minimum width=1cm,}]
\node[box,fill=lightgray!5] at (0,0)  [minimum width=3.5cm,minimum height=8cm, label=\textbf{Nonconvex functions}] (NCBB) {};
  
\node[bag, fill=white] (GlobalSub) at (0,0) [minimum width=2cm,minimum height=2cm,label=\textbf{Global sublinear}] {Rate: $\mathcal{O}\left(\tfrac{1}{k}\right)$\\ \enskip \\$\alpha_k = \mathcal{O}\left(\tfrac{1}{\kappa}\right)  $\\$\theta_k = \tilde{\theta}_g \in [0,1)$\\$\sum\iota_k < \infty$};

\node[box,fill=lightgray!25] at ($(NCBB.east)+(6.5,0)$)  [minimum width=13cm,minimum height=8cm, label=\textbf{Strongly convex functions}] (SCBB) {};

\node[bag, fill=white] (GlobalLin) at ($(GlobalSub.east)+(3.5,0)$) [minimum width=2cm,minimum height=2cm,label=\textbf{Global linear}] {Rate constant:\\ $\max\left\{\left(1 - \frac{\left(1-\tilde{\theta}_g^2\right)}{2\kappa^2}\right), a_g\right\}$\\\enskip \\$\alpha_k = \mathcal{O}\left(\tfrac{1}{\kappa}\right)  $\\$\theta_k = \tilde{\theta}_g \in [0,1)$\\$\iota_k = \iota_0a_g^k, \enskip a_g \in [0,1)$ };

\node[bag, fill=white] (LocalLin) at ($(GlobalLin.east)+(4.5,2)$) [minimum width=6cm,minimum height=2cm,label=\textbf{Local linear}] {Rate constant: $\sqrt{a_l}\in [0,1)$\\
\enskip\\
$\alpha_k = 1$\\
$\theta_k = \mathcal{O}\left(\sqrt{\frac{a_l}{\kappa^3}}\right)$\\$\iota_k = \iota_0a_l^k$ };

\node[bag, fill=white] (LocalSup) at ($(GlobalLin.east)+(4.5,-2)$) [minimum width=6cm,minimum height=2cm,label=\textbf{Local superlinear}] {Asymptotic rate constant: $\mathcal{O}\left(\frac{1}{k}\right)$\\
\enskip\\
$\alpha_k = 1$\\
$\theta_k = \mathcal{O}\left(\sqrt{\frac{a_l}{k^2\kappa^3}}\right)$, $a_l \in [0,1)$\\$\iota_k = \frac{\iota_0a_l^k}{k^4}$ };

\draw[thick, -stealth] (GlobalSub.east) -- (GlobalLin.west);
\draw[thick, -stealth] (GlobalLin.east) -- (LocalLin.west);
\draw[thick, -stealth] (GlobalLin.east) -- (LocalSup.west);

\node[bag, fill=cyan!12.5] (Obj) at ($(GlobalSub.east)+(0.75,2.875)$) [minimum width=4cm,minimum height=1.5cm] {\textbf{Objective Function}\\\enskip\\$\min_w f(w) = \frac{1}{N}\sum_{i=1}^NF_i(w)$};


\node[bag, fill=green!12.5] (Obj) at ($(GlobalLin)+(0.8,-2.75)$) [minimum width=4cm,minimum height=1.5cm] {\textbf{Global-local}\\ \textbf{transition complexity}\\\enskip\\$\tilde{\mathcal{O}}\left(\max\left\{\tfrac{N}{|S_0|}, \kappa^2(1+\tfrac{M}{\mu^{3/2}})\right\}\right)$};

\node[box,fill=cyan!12.5] at ($(NCBB.east)+(4.75,-4.5)$)  [minimum width=16.5cm,minimum height=1cm] (NCBB) {\textbf{Constants}: $N$, $\kappa$, $\mu$, $M$ \qquad  \textbf{Hyperparameters}: $\alpha_k$, $\theta_k$, $\iota_k$, $|S_0|$};
\end{tikzpicture}

\caption{Overview of the results presented in this section. We characterize the main results for global and local convergence results, and their relationship to the problem constants and algorithmic hyperparameters. Here, $N$ denotes the number of data,  $\mu$ is the Hessian spectral lower bound, $\kappa = \frac{L}{\mu}$ is a condition-number like constant, and $M$ is the Hessian Lipschitz constant.}\label{fig:analysis_roadmap_det}
\end{figure}
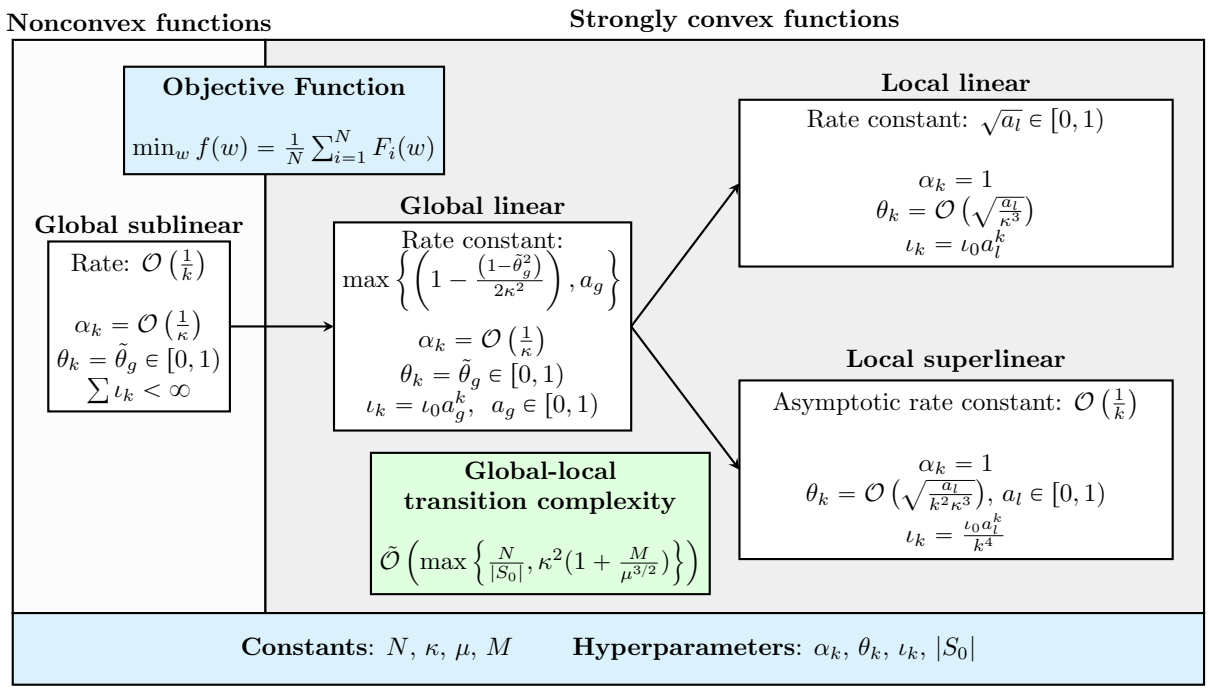

%% file: sections/expectation_analysis.tex
\input{sections/tikz_diagrams/analysis_roadmap_exp}

We continue our analysis by focusing on the stochastic sampling-based algorithms for the solution of the expectation problem \eqref{eq:prob_exp}. Our analysis mirrors the previous section and builds off of many of its assumption and derivations. A schematic for the analysis in this section is given in Figure~\ref{fig:analysis_roadmap_exp}. In addition to building on the results of the previous section, we also provide total number of gradients evaluated (gradient complexity) and Hessians computed to achieve an $\epsilon$-accurate solution for the strongly convex functions.

We begin our analysis by providing the stochastic analog of Lemma \ref{lem:tech_1}, which establishes an upper bound on the difference between the objective function values at successive iterations, albeit in conditional expectation.

\begin{lemma}\label{lem:tech_1_stoch}
Suppose Assumption~\ref{assumption:hessian_spectral_bounds} holds. For any $w_0$, let $\{w_k: k \in \mathbb{Z}^+\}$ be iterates generated by \eqref{eq:iter} with the Hessian approximation given in \eqref{eq:htilde_def}.  If the step size $\alpha_k$ at each iteration $k$ is chosen such that $\alpha_k \leq \frac{\tilde{\mu}}{L}$, and if $g_k$ is an unbiased estimator of $\nabla f(w_k)$. Then, for all $k \in \mathbb{Z}^+$, it follows that,
\begin{equation}\label{eq:expdesclemma}
    \E_k[f(w_{k+1})] \leq f(w_k) - \alpha_k\left(1 - \frac{L\alpha_k}{2\tilde{\mu}}\right)\nabla f(w_k)^T\widetilde{H}_k^{-1}\nabla f(w_k) + \frac{L\alpha_k^2}{2\tilde{\mu}}\E_k[\delta_k^T \widetilde{H}_k^{-1} \delta_k]. 
\end{equation}
\end{lemma}

\begin{proof}
Using \eqref{eq:proof_part1}, taking conditional expectation, and using $\E_k[g_k] = \nabla f(w_k)$, we get   
 \begin{align}
    \E_k[f(w_{k+1})]
    &\leq  f(w_k) - \alpha_k \nabla f(w_k)^T\widetilde{H}_k^{-1} \nabla f(w_k) + \frac{L\alpha_k^2}{2}\E_k[(\widetilde{H}_k^{-1/2}\nabla f(w_k))^T\widetilde{H}_k^{-1}(\widetilde{H}_k^{-1/2}\nabla f(w_k)] \nonumber \\
    &\quad +\frac{L\alpha_k^2}{2}\E_k[(\widetilde{H}_k^{-1/2}\delta_k)^T\widetilde{H}_k^{-1}(\widetilde{H}_k^{-1/2}\delta_k)]. \nonumber 
\end{align}
Using the fact that $ \widetilde{H}_k^{-1} \preceq \frac{1}{\tilde{\mu}} I$ in the above inequality yields \eqref{eq:expdesclemma}. 
\end{proof}

\subsection{Global convergence}
We denote the expectation with respect to all the random variables as $\E[\cdot]$. That is,
\begin{equation*}
    \E[f(w_k)] = \E_0\E_1\cdots\E_{k-1}[f(w_k)].
\end{equation*}

We proceed with the stochastic analog of Theorem \ref{thm:linear} (in expectation). 

\begin{theorem}{(Global linear convergence).}\label{thm:linear_exp} Suppose Assumptions~\ref{assumption:hessian_spectral_bounds} and ~\ref{assumption:strng_cnvx} hold. For any $w_0\in\R^d$, let $\{w_k: k \in \N\}$ be iterates generated by \eqref{eq:iter} where the Hessian approximation is given in \eqref{eq:htilde_def} and the gradient approximations $g_k$ satisfies the Condition~\ref{condition:expected_gradient_bounds} with $\iota_{k+1} = \iota_k a_g$ for some $\iota_0 > 0$ and $a_g \in [0,1)$. Then, if $g_k$ satisfies stochastic norm condition \eqref{eq:expected_norm} with
 \begin{subequations}
 \begin{enumerate}
 \item $A_k = \widetilde{H}_k^{-1}$, $\theta_k = \tilde{\theta}_g \in [0,1)$ and $\alpha_k = \alpha \leq \frac{\tilde{\mu}}{L}$: 
 \begin{align}\label{eq:lin_exp_reslt}
     \mathbb{E}[f(w_{k}) - f(w^*)] &\leq \tilde{C}_1\tilde{\rho}_1^k,\\
     \quad \tilde{C}_1 :=\max\left\{f(w_0) - f(w^*), \frac{\tilde{L}\iota_0}{\mu\left(1-\tilde{\theta}_g^2\right)} \right\}&, \mbox{ and } \tilde{\rho}_1 := \max\left\{1 - \frac{\alpha\mu\left(1-\tilde{\theta}_g^2\right)}{2\tilde{L}}, a_g\right\}\nonumber.
 \end{align}
 \item  $g_k$ being an unbiased estimator of $\nabla f(w_k)$, $\theta_k = \tilde{\theta}_g \geq 0$, and either of the following two conditions hold: 
 \begin{enumerate}
    \item  $A_k = \widetilde{H}_k^{-1}$, $\alpha_k = \alpha \leq \frac{\tilde{\mu}}{L\left(1+\tilde{\theta}_g^2\right)}$, (or)
    \item $A_k = I$, $\alpha_k = \alpha \leq \tfrac{\tilde{\mu}}{L\left(1+\tfrac{\tilde{\theta}_g^2\tilde{L}}{\tilde{\mu}}\right)}$,
 \end{enumerate}
 then, 
 \begin{align}\label{eq:lin_exp_reslt_2}
     \mathbb{E}[f(w_{k}) - f(w^*)] &\leq \tilde{C}_2\tilde{\rho}_2^k,\\
     \quad \tilde{C}_2 :=\max\left\{f(w_0) - f(w^*), \frac{L\tilde{L} \alpha \iota_0}{\mu \tilde{\mu}\min\{1, \tilde{\mu}\}} \right\}&, \mbox{ and } \tilde{\rho}_2 := \max\left\{1 - \frac{\alpha\mu}{2\tilde{L}}, a_g\right\} \nonumber.
 \end{align}
 \end{enumerate}
 \end{subequations}

\end{theorem}

\begin{proof}
\begin{subequations}
\noindent Case (1) From \eqref{eq:expected_norm} and \eqref{eq:desclemma}, we have
\begin{align}
    \E_k[f(w_{k+1})] &\leq f(w_k) - \frac{\alpha_k}{2} \nabla f(w_k)^T\widetilde{H}_k^{-1}\nabla f(w_k)+ \frac{\alpha_k}{2} \E_k\left[\E\left[\delta_k^T\widetilde{H}_k^{-1}\delta_k| w_k, \widetilde{H}_k^{-1}\right]\right]\nonumber \\
     &\leq f(w_k) - \frac{\alpha_k}{2} \nabla f(w_k)^T\widetilde{H}_k^{-1}\nabla f(w_k)+ \frac{\alpha_k \theta_k^2}{2} \E_k[\nabla f(w_k)^T\widetilde{H}_k^{-1}\nabla f(w_k)] + \frac{\alpha_k\iota_k}{2} \nonumber\\
    &\leq f(w_k) - \frac{\alpha_k(1 - \theta_k^2)}{2 \tilde{L}} \|\nabla f(w_k)\|^2 + \frac{\alpha_k\iota_k}{2}. \label{eq:gradnorm_stoch_reslt}
\end{align}

\noindent Case (2) From \eqref{eq:expected_norm} and \eqref{eq:expdesclemma}, and using condition $(a)$, we have
\begin{align}
 \E_k[f(w_{k+1})] &\leq f(w_k) - \alpha_k\left(1 - \frac{L\alpha_k}{2\tilde{\mu}}(1 + \theta_k^2)\right)\nabla f(w_k)^T\widetilde{H}_k^{-1}\nabla f(w_k) + \frac{L\alpha_k^2\iota_k}{2\tilde{\mu}} \nonumber \\
 &\leq f(w_k) - \frac{\alpha_k}{2\tilde{L}}\|\nabla f(w_k)\|^2 + \frac{L\alpha_k^2\iota_k}{2\tilde{\mu}}. \nonumber
\end{align}
Similarly, from \eqref{eq:expected_norm} and \eqref{eq:expdesclemma}, and using condition $(b)$, we have 
\begin{align}
 \E_k[f(w_{k+1})] &\leq f(w_k) - \frac{\alpha_k}{\tilde{L}}\left(1 - \frac{L\alpha_k}{2\tilde{\mu}}\right)\|\nabla f(w_k)\|^2 + \frac{L\alpha_k^2}{2\tilde{\mu}^2}\E_k[\|\delta_k\|^2] \nonumber \\
 &\leq f(w_k) - \frac{\alpha_k}{\tilde{L}}\left(1 - \frac{L\alpha_k}{2\tilde{\mu}}\left(1 + \tfrac{\theta^2\tilde{L}}{\tilde{\mu}}\right)\right)\|\nabla f(w_k)\|^2 + \frac{L\alpha_k^2\iota_k}{2\tilde{\mu}^2}\nonumber \\
 &\leq f(w_k) - \frac{\alpha_k}{2\tilde{L}}\|\nabla f(w_k)\|^2 + \frac{L\alpha_k^2\iota_k}{2\tilde{\mu}^2}. \nonumber
\end{align}
Combining the above two results yield
\begin{align}
 \E_k[f(w_{k+1})] 
 &\leq f(w_k) - \frac{\alpha_k}{2\tilde{L}}\|\nabla f(w_k)\|^2 + \frac{L\alpha_k^2\iota_k}{2\tilde{\mu}\min\{1, \tilde{\mu}\}}. \label{eq:gradnorm_stoch_reslt2}
\end{align}
The rest of the proof to attain  \eqref{eq:lin_exp_reslt} and \eqref{eq:lin_exp_reslt_2} follows by using similar arguments as in the deterministic norm condition analysis give in the proof of Theorem \ref{thm:linear}.  
\end{subequations}
\end{proof}
As was discussed in Remark \ref{remark_glob_lin}, we note that the rate constant $\left(1 - \frac{\alpha \mu}{2\tilde{L}}\right)$ being worse than that of steepest descent is an artifact of the analysis and not essentially algorithmic in nature. We do not observe deteriorated global convergence in our numerical experiments. Additionally we do not consider the biased stochastic norm condition with $A_k=I$, as it leads to restrictive choices of $\tilde{\theta}_g$.

\begin{theorem}{(Global sublinear convergence, stochastic case).} \label{thm:global_sublinear_bounded_expectation}
Suppose Assumption~\ref{assumption:hessian_spectral_bounds} holds and the objective function $f$ is bounded below by $f_\text{min}$. For any $w_0$, let $\{w_k: k \in \N\}$ be iterates generated by \eqref{eq:iter} 
where the gradient approximations $g_k$ satisfy the Condition~\ref{condition:expected_gradient_bounds}with $\sum_{i=0}^{\infty} \iota_k = \tilde{\iota} < \infty$. Then, for any positive integer $T$, if $g_k$ satisfies stochastic norm condition \eqref{eq:expected_norm} with
 \begin{subequations}
 \begin{enumerate}
    \item $A_k = \widetilde{H}_k^{-1}$, $\theta_k = \tilde{\theta}_g \in [0,1)$ and $\alpha_k = \alpha \leq \frac{\tilde{\mu}}{L}$:
    \begin{equation}\label{eq:sublinear_result_1}
         \min\limits_{0\leq k \leq T-1} \E[\|\nabla f(w_k)\|^2] \leq \frac{\tilde{L}}{(1-\tilde{\theta}_g^2)T}\left(\frac{2(f(w_0) - f_\text{min})}{\alpha} + \tilde{\iota}\right).
    \end{equation}
    \item  $g_k$ being an unbiased estimator of $\nabla f(w_k)$, $\theta_k = \tilde{\theta}_g \geq 0$, and either of the following two conditions hold: 
 \begin{enumerate}
    \item  $A_k = \widetilde{H}_k^{-1}$, $\alpha_k = \alpha \leq \frac{\tilde{\mu}}{L(1+\tilde{\theta}_g^2)}$, (or)
    \item $A_k = I$, $\alpha_k = \alpha \leq  \tfrac{\tilde{\mu}}{L\left(1+\tfrac{\tilde{\theta}_g^2\tilde{L}}{\tilde{\mu}}\right)}$,
 \end{enumerate}
 then, 
 \begin{equation}\label{eq:sublinear_result_2}
 \min\limits_{0 \leq k \leq T-1} \mathbb{E}[\|\nabla f(w_k)\|^2] \leq \frac{\tilde{L}}{T}\left(\frac{2(f(w_0) - f_\text{min})}{\alpha} + \frac{L\alpha \tilde{\iota}}{\tilde{\mu}\min\{1, \tilde{\mu}\}}\right).
 \end{equation}
 \end{enumerate}
\end{subequations}
Moreover, $\sum_{k=0}^{\infty}\|\nabla f(w_k)\|^2 < \infty$ almost surely and consequently $\|\nabla f(w_k)\|^2 \rightarrow 0$ as $k \rightarrow \infty$ almost surely.  
\end{theorem}
\begin{proof}
Starting with inequalities \eqref{eq:gradnorm_stoch_reslt} and \eqref{eq:gradnorm_stoch_reslt2} and following the same procedure as in the deterministic case, albeit in expectation,  yields \eqref{eq:sublinear_result_1} and \eqref{eq:sublinear_result_2} respectively. 
Furthermore, applying Robbins-Siegmund Theorem \cite{robbins1970boundary} to \eqref{eq:gradnorm_stoch_reslt} or \eqref{eq:gradnorm_stoch_reslt2} yields $\sum_{k=0}^{\infty}\|\nabla f(w_k)\|^2 < \infty$ almost surely. Using this inequality, it is not difficult to show that $\|\nabla f(w_k)\|^2 \rightarrow 0$ as $k \rightarrow \infty$ almost surely.  
\end{proof}

\subsection{Local convergence} \label{subsec:local_stoch_conv}

We now provide local superlinear rates of convergence results for the iterates generated by \eqref{eq:iter} when unit step size is eventually employed. We begin by extending Lemma \ref{lemma:tech_local} to the stochastic setting.

\begin{lemma}\label{lemma:tech_local_stoch}
Suppose Assumptions~\ref{assumption:hessian_spectral_bounds} and~\ref{assumption:lipschitz_hessian} hold. For any $w_0\in\R^d$, let $\{w_k: k \in \N\}$ be iterates generated by \eqref{eq:iter} where the Hessian approximation is given in \eqref{eq:htilde_def} and the gradient approximations $g_k$ satisfies the Condition~\ref{condition:expected_gradient_bounds} with $\lambda_{\min}(A_k) \geq \lambda_{A}$ and $\frac{\lambda_{\max}(A_k)}{\lambda_{\min}(A_k)} \leq \kappa_{A}$ for some positive constants $\lambda_A, \kappa_{A} < \infty$. If at any iteration $k \in \N$, unit step size is chosen ($\alpha_k=1$), then, if $g_k$ satisfies stochastic norm condition \eqref{eq:expected_norm}
 \begin{align}
    \E\left[\|w_{k+1} - w^*\|\right] &\leq \frac{M}{2\tilde{\mu}}\E\left[\|w_k-w^*\|^2\right] + \frac{1}{\tilde{\mu}}\E\left[\|(\widetilde{H}_k - \nabla^2 f(w_k))(w_k - w^*)\|\right] \nonumber \\
    &\quad + \frac{L\sqrt{\kappa_{A}}}{\tilde{\mu}}\theta_k\E\left[\|w_k - w^*\|\right] + \frac{\sqrt{\iota_k}}{\tilde{\mu}\sqrt{\lambda_A}}, \label{eq:twoterm_exp}
    \end{align}
where $w^*$ is an optimal solution. 
\end{lemma}
\begin{proof}
We use a similar approach to the proof of Lemma \ref{lemma:tech_local}, see the proof there for the decomposition of errors. When $g_k$ satisfies stochastic norm condition \eqref{eq:expected_norm}, taking conditional expectations of the gradient term in \eqref{eq:three_terms} and using Jensen's inequality $\E[\|g_k - \nabla f(w_k)\|_{A_k}|w_k, A_k]\leq (\E[\|g_k - \nabla f(w_k)\|^2_{A_k}|w_k,A_k])^{1/2}$ followed by full expectation yields \eqref{eq:twoterm_exp}.  
\end{proof}

Lemma~\ref{lemma:tech_local_stoch} establishes the dependence of the iterate distance to optimality on the Hessian approximation error. Recall that Lemma \ref{lem:hess_error} decomposes the error into several terms, one of which includes the sum of errors in the difference between subsampled Hessian at each iterate, and the subsampled Hessian at the optimum, $\nabla^2 F_{S_i}(w_i) - \nabla^2F_{S_i}(w^*)$. In the next lemma, we establish upper bounds for the terms in  Lemma~\ref{lem:hess_error}, similar to the bounds derived in the deterministic sampling case as was done in Lemma \ref{lemma:memory_det}.

\begin{lemma}\label{lemma:memory_exp}
    Suppose Assumptions~\ref{assumption:hessian_spectral_bounds},~\ref{assumption:lipschitz_hessian} and ~\ref{assum:local_strng_cnvx} hold. For any $w_0\in\R^d$, let $\{w_k: k \in \N\}$ be iterates generated by \eqref{eq:iter} where the Hessian approximation is given in \eqref{eq:htilde_def} with $\gamma_i=\frac{1}{k+1}$ for all $i=0,\cdots,k$, and the gradient approximations $g_k$ satisfies the Condition~\ref{condition:expected_gradient_bounds} with $\sum_{i=0}^{\infty} \iota_k = \tilde{\iota} < \infty$. If $A_k$ and the corresponding step size $\alpha_k$ are chosen according to Theorem~\ref{thm:global_sublinear_bounded_expectation}. Then there exists $k_{\text{lin}} \geq 0$ such that for any $k \geq k_{\text{lin}}$ if $\iota_{k+1} = \iota_{k} a_g$ for some $\iota_{k_\text{lin}} \geq 0$ and $a_g \in [0,1)$, we have that if $g_k$ satisfies stochastic norm condition \eqref{eq:expected_norm} with any of the choices for $(A_k, \theta_k,\alpha_k)$ provided in Theorem~\ref{thm:global_sublinear_bounded_expectation}. In addition, suppose either $\sum_{i=0}^{\infty} \|\nabla f(w_i)\|^2 < \infty$ or Assumption~\ref{assumption:strng_cnvx} holds, then there exists a constant $C_{p,s}$ such that
        \begin{align}\label{eq:lemma36_exp}
            \sum_{i=0}^k \gamma_i  \E\left[\left\|\left(\nabla^2 F_{S_i} (w_i) - \nabla^2 F_{S_i}(w^*)\right)(w_k - w^*)\right\|\right]&\leq \frac{C_{p,s}}{k+1}\left(\E\left[\left\|w_k - w^*\right\|^2\right]\right)^{1/2}.
        \end{align}
\end{lemma}
\begin{proof}
In the stochastic setting, we assume that $\sum_{i=0}^{\infty} \|\nabla f(w_i)\|^2 < \infty$ or Assumption~\ref{assumption:strng_cnvx} holds. If $\sum_{i=0}^{\infty} \|\nabla f(w_i)\|^2 < \infty$ then there exists a $k_{\text{lin}} \in \N$ such that for all $k\geq k_{\text{lin}}$, $\|\nabla f(w_k)\| \leq \nu$. 
Therefore, for all $k\geq k_{\text{lin}}$ the iterates are all in locally strongly convex regime. Taking expectations in \eqref{eq:proof_sub_0} we get, 
\begin{align}
&\sum_{i=0}^k \gamma_i  \E\left[\left\|\left(\nabla^2 F_{S_i} (w_i) - \nabla^2 F_{S_i}(w^*)\right)e_k\right\|\right] \nonumber \\
&\quad \leq  \frac{2Lk_{\text{lin}}}{k+1}\E\left[\|w_k - w^*\|\right] + \frac{M}{k+1}\sum_{i=k_{\text{lin}}}^{k}\E\left[\|w_i - w^*\|\|w_k - w^*\|\right] \nonumber \\
&\quad \leq  \frac{2Lk_{\text{lin}}}{k+1}\left(\E\left[\|w_k - w^*\|^2\right]\right)^{1/2} + \frac{M}{k+1}\sum_{i=k_{\text{lin}}}^{k}\left(\E\left[\|w_i - w^*\|^2\right]\right)^{1/2}\left(\E\left[\|w_k - w^*\|^2\right]\right)^{1/2},
\label{eq:proof_sub_exp}
\end{align}
where the last inequality is due to the fact that $(\E[a])^2 \leq \E[a^2]$ and $(\E[ab])^2 \leq \E[a^2]\E[b^2]$ for any $a,b > 0$. 
Now using the global convergence results given in Theorem~\ref{thm:global_sublinear_bounded_expectation} and following the similar approach as in the deterministic norm condition analysis, albeit in expectation, we get, for an appropriate choice of  $\iota_{k_{\text{lin}}}$ that $\|\nabla f(w_k)\|^2 \leq \nu^2$ for $k\geq k_{\text{lin}}$. Furthermore, from Jensen's inequality, we also have
\begin{align}\label{eq:proof_lin_1_exp}
    \left(\E\left[\|w_k - w^*\|\right]\right)^2 \leq  \E\left[\|w_k - w^*\|^2\right] \leq \tfrac{\nu^2 \tilde{\rho}^{k-k_{\text{lin}}}}{\mu^2},
\end{align}
and 
\begin{align}\label{eq:proof_sum_1_exp}
    \sum_{i=k_{\text{lin}}}^{k}\left(\E\left[\|w_i - w^*\|^2\right]\right)^{1/2} & < \frac{\nu}{\mu (1-\sqrt{\tilde{\rho}})},
\end{align}
where $\tilde{\rho}$ is the rate constant that depends on the choice of $(A_k,\theta_k, \alpha_k)$ given in Theorem~\ref{thm:linear_exp}. Substituting \eqref{eq:proof_sum_1_exp} in \eqref{eq:proof_sub_exp} yields
\begin{align}
\sum_{i=0}^k \gamma_i  \E\left[\left\|\left(\nabla^2 F_{S_i} (w_i) - \nabla^2 F_{S_i}(w^*)\right)e_k\right\|\right] &\leq  \frac{1}{k+1}\left(2Lk_{\text{lin}} + \frac{\nu M}{\mu (1-\sqrt{\tilde{\rho}})}\right)\left(\E\left[\|w_k - w^*\|^2\right]\right)^{1/2} \nonumber \\
&= \frac{C_{p,s}}{k+1}\left(\E\left[\|w_k - w^*\|^2\right]\right)^{1/2}, \nonumber 
\end{align}
where 
\begin{equation}\label{eq:Cps}
C_{p,s} := 2Lk_{\text{lin}} + \frac{\nu M}{\mu(1-\sqrt{\tilde{\rho}})}.
\end{equation}
\end{proof}

\begin{remark}
If the functions are globally strongly convex (Assumption~\ref{assumption:strng_cnvx} holds) then there is no sublinear convergent phase in the algorithm. That is, $k_{\text{lin}} = 0$ in this setting. Furthermore, we made the assumption that $\sum_{i=0}^{\infty}\|\nabla^2 f(w_i)\|^2 < \infty$. This assumption, although made on the iterates which are stochastic, is necessary to ensure that the iterates eventually lie within a locally strongly convex regime. Moreover, from Theorem~\ref{thm:global_sublinear_bounded_expectation}, we have that $\sum_{i=0}^{\infty}\|\nabla^2 f(w_i)\|^2 < \infty$ almost surely, making this assumption relatively weak in this setting. 
\end{remark}

We will now establish that the nonconvex error term in Lemma~\ref{lem:hess_error}  vanishes for sufficiently large number of iterations. To achieve this, for the expectation problem, we need an additional assumption about the subsampled functions when the iterates enter the strongly convex regime. 
\begin{assumption}
    \label{assum:strng_cnv_components}
    In the expectation problem \eqref{eq:prob_exp}, for all $k\geq k_{\text{lin}}$, where $k_{\text{lin}}$ is defined in Lemma~\ref{lemma:memory_exp}, $\nabla^2 F_{S_k}(w_k) \succeq \mu I$. 
\end{assumption}
\begin{remark}
Assumption~\ref{assum:strng_cnv_components} implies that when the iterates enter the locally strongly convex regime, the subsampled functions are also strongly convex. In the case when the objective function is strongly convex, this assumption is typically satisfied due to regularziation in machine learning problems.  Moreover, such assumptions have been previously made in \cite{bollapragada2019exact,berahas2020investigation} and are required to establish strong convergence results in expectation.   
\end{remark}

\begin{lemma}
    \label{lem:nonconvex_stoch}
    Suppose conditions of Lemma~\ref{lemma:memory_exp} hold. Let $\tilde{\mu} \leq \frac{\mu}{2}$, and  $g_k$ satisfies stochastic norm condition \eqref{eq:expected_norm} and suppose that Assumption~\ref{assum:strng_cnv_components} holds. Then, there exists $k_{\text{non}} \geq k_{\text{lin}}$ such that for all $k \geq k_{\text{non}}$, $\widetilde{H}_k = \widehat{H}_k$. 
\end{lemma}
\begin{proof}
Let 
\begin{align}\label{eq:k_non_stoch}
k_{\text{non}} = 2 k_{\text{lin}}\left(1 + \tfrac{\hat{\lambda}}{\mu}\right).
\end{align}
Then, from \eqref{eq:proof_noncon_0}, for all $k\geq k_{\text{non}}$,  we have
\begin{align*}
    v^T\widehat{H}_kv &\geq \frac{-\hat{\lambda}k_{\text{lin}}}{k+1}\|v\|^2 + \frac{\mu (k + 1 - k_{\text{lin}})}{k+1}\|v\|^2 \geq \frac{\mu}{2}\|v\|^2 \geq \tilde{\mu}\|v\|^2,
\end{align*}
where we used $\tilde{\mu} \leq \tfrac{\mu}{2}$. Therefore, $\lambda_{\min}(\widehat{H}_k) \geq \tilde{\mu}$ and $\widetilde{H}_k = \widehat{H}_k$ for all $k\geq k_{\text{non}}$.  
\end{proof}

To analyze the last term in Lemma~\ref{lem:hess_error}, we make the following standard assumption about individual (stochastic) Hessian components. 

\begin{assumption}{(Hessian approximations, stochastic case).}\label{assum:hessian_var_stoch}
The individual component Hessians are bounded relative to the Hessian of the objective function $f$ at the optimal solution $w^*$. That is, 
\begin{subequations}
for the expectation problem, there exists constant $\sigma^2_{H} \geq 0$ such that 
    \begin{equation}\label{eq:exp_hess_var_bnd}
        \E_{\zeta}\left[\left\|\nabla^2 F(w^*,\zeta) - \nabla^2 f(w^*)\right\|^2\right] \leq \sigma^2_{H}. 
    \end{equation}

\end{subequations}
\end{assumption}

\begin{remark}
As with Assumption~\ref{assum:hessian_var}, we note that Assumption~\ref{assum:hessian_var_stoch} is relatively weak compared to similar assumptions made in the literature \cite{bollapragada2019exact}, as it only requires bounded variance at the optimum, instead at all possible iterates $w\in\R^d$. In addition, we also note that this assumption is trivially satisfied with the bound $\sigma_{H} = 2L$ when the functions are double differentiable with Lipschitz continuous gradients. 

\end{remark}

\begin{lemma}\label{lem:sampling error_stoch}
    Suppose Assumptions~\ref{assumption:lipschitz_hessian} and~\ref{assum:hessian_var_stoch} hold. If $\gamma_i=\frac{1}{k+1}$ for all $i=0,\cdots,k$, 
    and the sample sets $S_k$ are randomly chosen such that $|S_k| = |S_0| \in \N$. Then 
    \begin{align}\label{eq:lemma37_exp}
        \E\left[\left\|\left(\sum_{i=0}^k \gamma_i\nabla^2 F_{S_i}(w^*) - \nabla^2f(w^*)\right)(w_k - w^*)  \right\|\right] &\leq \frac{\sigma_H}{\sqrt{(k+1)|S_0|}}(\E[\|w_k - w^*\|^2])^{1/2}.
    \end{align}
\end{lemma}
\begin{proof}
   Consider, 
	\begin{align}
		&\E\left[\left\|\sum_{i=0}^k \gamma_i\nabla^2 F_{S_i}(w^*) - \nabla^2f(w^*)(w_k - w^*)  \right\|\right] \nonumber \\ &\quad \leq   \left(\E\left[\left\|\sum_{i=0}^k \gamma_i\nabla^2 F_{S_i}(w^*) - \nabla^2f(w^*)\right\|^2\right]\right)^{1/2}\left(\E\left[\|w_k - w^*\|^2\right]\right)^{1/2}. \label{eq:proof_2_exp_1}
	\end{align}
where the inequality is due to 	$(\E[ab])^2 \leq \E[a^2]\E[b^2]$ for any $a,b > 0$. Let $S^\dagger = \cup_{i=0}^k S_i$ with $|S^\dagger| = (k+1)|S_0|$. Using $\nabla^2 F_{S^\dagger}(w^*) = \frac{1}{(k+1)|S_0|}\sum_{i \in S^{\dagger}}\nabla^2 F_{i} (w^*)$, we get
\begin{align}
	\E\left[\left\|\nabla^2 F_{S^\dagger}(w^*) - \nabla^2 f(w^*)\right\|^2\right] &= \frac{1}{(k+1)^2|S_0|^2}\E\left[\left\|\sum_{i \in S^{\dagger}} (\nabla^2 F_i(w^*) - \nabla^2 f(w^*))\right\|^2\right] \nonumber \\
	&=\frac{1}{(k+1)^2|S_0|^2}\sum_{i \in S^{\dagger}}\E\left[\left\|\nabla^2 F_i(w^*) - \nabla^2 f(w^*)\right\|^2\right] \nonumber \\
	&= \frac{1}{(k+1)|S_0|}\E_{\zeta}\left[\left\|\nabla^2 F(w^*,\zeta) - \nabla^2 f(w^*)\right\|^2\right] \leq \frac{\sigma^2_H}{(k+1)|S_0|}, \label{eq:proof_2_exp_2}
\end{align}
where the second and third equalities are due to the fact that the sample sets $S_k$'s consists of i.i.d samples of $\zeta$, and the inequality is due to Assumption~\ref{assum:hessian_var}. Substituting \eqref{eq:proof_2_exp_2} in \eqref{eq:proof_2_exp_1} yields the desired result.    
\end{proof}
\begin{remark}
	Lemma~\ref{lem:sampling error_stoch} establishes the bound on the sampling error in terms of the sample size $|S_0|$ and the iteration number $k$. We note that deterministic sampling without replacement in a cyclic manner has a better dependence on $k$ as compared to the stochastic subsampled Hessian (see Lemma~\ref{lem:sampling error}).
\end{remark}

To prove similar results for the expectation problem, we need an additional assumption on the second moments of the iterates as employed in stochastic second-order methods \cite{bollapragada2019exact,berahas2024modified}. 
\begin{assumption}
There exists a non-negative constant $\eta$ such that for all $k\in \mathbb{Z}^+$, 
$$\mathbb{E}\left[\|w_k - w^* \|^2\right] \leq \eta \left( \mathbb{E}[\|w_k - w^* \|] \right)^2.$$
\label{assum.second.moment.bound}
\end{assumption}
\begin{remark}
Although this assumption seems to be restrictive, it is imposed on non-negative numbers and is less restrictive than assuming that the iterates are bounded. It might be stronger than the sub-exponential assumption on the stochastic Hessian \cite{na2023hessian,jiang2024stochastic}, it is however required to establish results in expectation instead of results in probability. This assumption has been employed in other works to the same effect \cite{bollapragada2019exact,berahas2024modified}.
\end{remark}

Finally we provide linear and superlinear local convergence rates for the expectation problem. This theorem is the stochastic analog of Theorem \ref{thm:super_det}, and the proof is similar. We provide the entire proof for completeness.

\begin{theorem} {(Expectation local linear and superlinear convergence).}\label{thm:super_stoc}
Suppose Assumptions~\ref{assumption:hessian_spectral_bounds}, ~\ref{assumption:lipschitz_hessian},~\ref{assum:local_strng_cnvx},~\ref{assum:strng_cnv_components}, ~\ref{assum:hessian_var_stoch}, and~\ref{assum.second.moment.bound} hold. For any $w_0\in\R^d$, let $\{w_k: k \in \N\}$ be iterates generated by \eqref{eq:iter} where the Hessian approximation is given in \eqref{eq:htilde_def} with $\tilde{\mu} \leq \tfrac{\mu}{2}$, the sample sets $S_k$ are randomly chosen such that $|S_k| = |S_0|$ for all $k\in\N$, and $\gamma_i=\frac{1}{k+1}$ for all $i=0,\cdots,k$. Furthermore, the gradient approximations $g_k$ satisfies stochastic norm condition \eqref{eq:expected_norm} with $\sum_{i=0}^{\infty}\iota_i = \tilde{\iota} < \infty$ and any of the choices for $(A_k, \theta_k,\alpha_k)$ with $\frac{\lambda_{\max}(A_k)}{\lambda_{\min}(A_k)} \leq \kappa_{A} > 0$ provided in Theorem~\ref{thm:global_sublinear_bounded_expectation} for all $k < k_{\text{sup}}$ for some $k_{\text{sup}} \in \N$. 
In addition, suppose either $\sum_{i=0}^{\infty} \|\nabla f(w_i)\|^2 < \infty$ or Assumption~\ref{assumption:strng_cnvx} holds. Furthermore, for all $k_{\text{lin}}\leq k < k_{sup}$, $\iota_{k+1} = \iota_{k} a_g$ for some $\iota_{k_\text{lin}} \geq 0$ and $a_g \in [0,1)$, where $k_{\text{lin}}$ is given in Lemma~\ref{lemma:memory_exp}. Let $\tilde{\theta}_l = \frac{\sqrt{a_l}\tilde{\mu}}{9\sqrt{\kappa_A}L}$, $q= \tfrac{7M \eta}{2\tilde{\mu}}$, and $\tau_k := \frac{1}{\tilde{\mu}}\left(\tfrac{C_{p,s}\sqrt{\eta}}{k+1} + \tfrac{\sigma_{H}\sqrt{\eta}}{\sqrt{(k+1)|S_0|}} + \theta_kL\sqrt{\kappa_{A}}\right)$ for some $a_l \in [0,1)$, and let $r_k$ be a sequence with and $r_{k_{\text{sup}}}=\max\left\{\E[\|w_{k_{\text{sup}}} - w^*\|], \tfrac{3\sqrt{\iota_{k_{\text{sup}}}}}{\tilde{\mu}\sqrt{\lambda_{A}a_l}}\right\} \leq \tfrac{\sqrt{a_l}}{3q}$. 
\begin{enumerate}
\item If $\alpha_k= 1$, $\theta_k = \tilde{\theta}_l$, and $\iota_{k+1} = \iota_{k}a_l$ for all $k \geq k_{\text{sup}}$. Then 
\begin{align}\label{eq:loc_lin_exp}
        \E[\|w_{k} - w^*\|]  \leq r_{k}, \quad r_{k+1} = \max\left\{ r_k \rho_k, r_{k_{\text{sup}}}(\sqrt{a_l})^{k-k_{\text{sup}}+1} \right\}, \quad \rho_k = qr_k + \tau_k + \tfrac{\sqrt{\iota_{k_{\text{sup}}}}}{\tilde{\mu}\sqrt{\lambda_A}r_{k_{\text{sup}}}} \in [0,\sqrt{a_l}].
\end{align}
Therefore, $\E[\|w_{k} - w^*\|] \rightarrow 0$ at an R-linear rate with rate constant upper bounded by $\sqrt{a_l} \in [0,1)$.
\item If $\alpha_k=1$, $\theta_k = \frac{\tilde{\theta}_l}{\sqrt{k+1}}$, and $\iota_{k+1} = \iota_{k}a_l t_{k+1}^4$, $t_k = \tfrac{1}{\sqrt{k}}$  for all $k \geq k_{\text{sup}}$. Then 
\begin{align}
\E[\|w_{k} - w^*\|] &\leq r_{k}, \quad r_{k+1} = \max\left\{ r_k \rho_k, r_{k_{\text{sup}}}(\sqrt{a_l})^{k-k_{\text{sup}}+1} \prod_{i=k_{\text{sup}}}^{k}t_{i+1}\right\}, \nonumber \\
&\rho_k = qr_k + \tau_k + \tfrac{\sqrt{\iota_{k_{\text{sup}}}}}{\tilde{\mu}\sqrt{\lambda_A}r_{k_{\text{sup}}}}\prod_{i=k_{\text{sup}}}^{k-1}t_{i+1} = \mathcal{O}(\tfrac{1}{\sqrt{k}})\in [0,\sqrt{a_l}]. \label{eq:loc_suplin_exp}
\end{align}
Therefore, $\E[\|w_{k} - w^*\|] \rightarrow 0$ at an R-superlinear rate with rate constant upper bounded by $\max\{\rho_k, \sqrt{a_l}t_{k+1}\} =\mathcal{O}\left(\tfrac{1}{\sqrt{k}}\right)$. 
\end{enumerate}
\end{theorem}
\begin{proof}
Let 
\begin{align}\label{eq:ksup_exp}
k_{\text{sup}} = \bigg\lceil \max \bigg\{&2 k_{\text{lin}}\left(1 + \tfrac{\hat{\lambda}}{\mu}\right),  \tfrac{9C_{p,s}\sqrt{\eta}}{\tilde{\mu}\sqrt{a_l}}, \tfrac{81\sigma^2_H\eta}{\tilde{\mu}^2|S_0|a_l},  \nonumber \\
&~~~k_{\text{lin}} + 2\log_{1/\tilde{\rho}}\left(\tfrac{3q\nu}{\mu \sqrt{a_l}}\right), k_{\text{lin}} + \log_{1/a_g}\left(\tfrac{81q^2 \iota_{k_{\text{lin}}}}{a_l^2\tilde{\mu}^2\lambda_{A}}\right)\bigg\}\bigg\rceil,
\end{align}
where $k_{\text{lin}}$ is defined in Lemma~\ref{lemma:memory_exp}, $\hat{\lambda}$ is defined in Lemma~\ref{lem:nonconvex}, $C_{p,s}$ is defined in \eqref{eq:Cps}, and $\tilde{\rho}$ is the linear convergence rate defined in Theorem~\ref{thm:linear_exp}.
We note that $k_{\text{sup}} \geq k_{\text{non}}$ due to the first term in \eqref{eq:ksup_exp}  where $k_{\text{non}}$ is defined in \eqref{eq:k_non_stoch}.
Using ~\eqref{eq:twoterm_exp},~\eqref{eq:total_hessianerror},~\eqref{eq:lemma36_exp}, and~\eqref{eq:lemma37_exp}, and Assumption~\ref{assum.second.moment.bound} we get for all $k \geq k_{\text{sup}}$,
\begin{align}
    \E[\|w_{k+1} - w^*\|] &\leq \tfrac{7M\eta}{2\tilde{\mu}}(\E[\|w_k - w^*\|])^2 + \tfrac{1}{\tilde{\mu}}\left(\tfrac{C_{p,s}\sqrt{\eta}}{k+1} + \tfrac{\sigma_{H}\sqrt{\eta}}{\sqrt{(k+1)|S_0|}} + \theta_kL\sqrt{\kappa_{A}}\right)\E[\|w_k - w^*\|] \nonumber \\
    &\qquad + \tfrac{\sqrt{\iota_k}}{\tilde{\mu}\sqrt{\lambda_A}} \label{eq:stoc_three_term}\\
    &= q\|w_{k} - w^*\|^2 + \tau_k \|w_{k} - w^*\| + o_k, \nonumber 
\end{align}
where $o_k := \tfrac{\sqrt{\iota_k}}{\tilde{\mu}\sqrt{\lambda_A}}$. Using the second and third terms in \eqref{eq:ksup_exp} and $\theta_k =\tilde{\theta}_l= \frac{\sqrt{a}_l\tilde{\mu}}{9\sqrt{\kappa_A}L}$, we get $\tau_k \leq \frac{\sqrt{a}_l}{3}$ for $k = k_{\text{sup}}$. In addition, using the fourth term in \eqref{eq:ksup_exp} and \eqref{eq:proof_lin_1_exp}, we get $\E[\|w_k - w^*\|] \leq \frac{\sqrt{a_l}}{3q}$
for $k = k_{\text{sup}}$. Moreover, $\iota_{k_{\text{sup}}} = \iota_{k_{\text{lin}}} a_g^{k_{\text{sup}}-k_{\text{lin}}} \leq \tfrac{a_l^2\tilde{\mu}^2\lambda_{A}} {81q^2}$ due to the fifth term in \eqref{eq:ksup_det} which implies that $o_{k} \leq \tfrac{a_l}{9q}$ for $k = k_{\text{sup}}$.


Therefore, starting with $k=k_{\text{sup}}$ and using Lemma~\ref{lemma:superlinear} with $\upsilon = \sqrt{a}_l$ and $\iota_{k+1} = \iota_k a_l$ yields \eqref{eq:loc_lin_exp}. Moreover, from \eqref{eq:lin_constant}, we have that $\tfrac{r_{k+1}}{r_k} \leq \sqrt{a_l}<1$. 

Similarly starting with $k=k_{\text{sup}}$ and using Lemma~\ref{lemma:superlinear} with $\upsilon = \sqrt{a_l}$ and $\iota_{k+1} = \tfrac{\iota_k a_l}{(k+1)^2}$ yields \eqref{eq:loc_suplin_exp}. Moreover, from \eqref{eq:lin_constant}, we get
\begin{align*}
\tfrac{r_{k+1}}{r_k} \leq \max\left\{\rho_k, \sqrt{a_l} t_{k+1}\right\} 
&\leq \max\left\{ qr_k + \tau_k + \tfrac{o_{k_{\text{sup}}}}{r_0}\prod_{i=k_{\text{sup}}}^{k-1}t_{i+1},\sqrt{\tfrac{a_l}{k+1}} \right\} \nonumber \\
&< \max\left\{ qr_{k_{\text{sup}}}(\sqrt{a_l})^{k-k_{\text{sup}}} + \tau_k  + \tfrac{o_{k_{\text{sup}}}}{r_{k_{\text{sup}}} \sqrt{k+1}},\sqrt{\tfrac{a_l}{k+1}} \right\} = \mathcal{O}(\tfrac{1}{\sqrt{k}}).
\end{align*}
\end{proof}

\begin{remark}
We note that unlike the deterministic Hessian sampling, stochastic sampling leads to a slower rate of superlinear convergence $\mathcal{O}\left(\tfrac{1}{\sqrt{k}}\right)$ similar to the final phase results established in \cite{na2023hessian,jiang2024stochastic}. Moreover, for the exact gradient settings, although we only presented the final superlinear convergence results, from \eqref{eq:stoc_three_term}, it can be seen that there are two phases in the superlinearly convergent regime, where initially rate is $\mathcal{O}\left(\tfrac{1}{k}\right)$ and the later (final) phase is $\mathcal{O}\left(\tfrac{1}{\sqrt{k}}\right)$, similar to results in \cite{na2023hessian,jiang2024stochastic}, albeit in expectation instead of high probability. We also note that the proof techniques employed in analyzing the Hessian approximations in this analysis are different from those established in \cite{na2023hessian,jiang2024stochastic} where we analyze the variance in the Hessian sampling error at the optimal solution.
\end{remark}

\subsection{Complexity Analysis}
In this section, we establish the iteration, gradient, and Hessian computational complexity bounds, i.e., the total number of iterations, individual gradient evaluations and Hessian evaluations required to get an $\epsilon$-accurate solution, characterized as any iterate $w_k$ satisfying 
\begin{align}
    \Big(\E[\|w_k - w^*\|]\Big)^2 \leq \epsilon,
\end{align}
where $w^*$ is an optimal solution. For the sake of simplicity, we will only consider global strongly convex functions (see Assumption~\ref{assumption:strng_cnvx}). The global sublinear convergence results for general nonconvex functions established in Theorem~\ref{thm:global_sublinear_bounded_expectation} are required only to establish that the iterates will eventually enter a basin with $\|\nabla f(w_k)\|^2 \leq \nu^2$ and therefore, a complete complexity analysis can be performed by including the analysis for this sublinear phase too.  Furthermore, to make our analysis simple and make it possible to compare our results with other methods in the literature, we limit our analysis to specific settings where we choose $\theta_k = 0$ and $\iota_k = \iota_0a_g^k$ for all $k < k_{\text{sup}}$ and $\iota_k = \iota_0a_l^k$ for all $k \geq k_{\text{sup}}$ which leads to fast local linear convergence (see Theorem~\ref{thm:super_stoc}).

\begin{corollary}
    \label{corr:iter_complexity_det}
Suppose Assumption~\ref{assumption:strng_cnvx} and conditions of Theorem~\ref{thm:super_stoc} hold with $\theta_k = 0$, $\iota_k = \iota_0a_g^k$ for all $k < k_{\text{sup}}$ where $a_g = 1 - \tfrac{\mu \tilde{\mu}}{2L\tilde{L}}$, and $\iota_k = \iota_0a_l^k$ for all $k \geq k_{\text{sup}}$ with $a_{\text{l}} \in [0,1)$, and $A_k = I$. Then, for any given (sufficiently small) $\epsilon > 0$,  we  get an $\epsilon$-accurate solution after performing $K_{\epsilon}$ iterations where 
\begin{align}\label{eq:kepsilob_stoc}
K_{\epsilon} = k_{\text{sup}} + \left\lceil \log_{1/\sqrt{a_l}}\left(\tfrac{r_{k_{\text{sup}}}}{\sqrt{\epsilon}}\right) \right\rceil = \mathcal{\tilde{O}}\left(\kappa^2\left(1 + \tfrac{M}{\mu^{3/2}}\right) + \log\left(\tfrac{1}{\sqrt{\epsilon}}\right)\right),
\end{align} 
where $r_{k_\text{sup}}$ is defined as in Theorem~\ref{thm:super_stoc}.
\end{corollary}
\begin{proof}
From \eqref{eq:loc_lin_exp}, we have that for all $k\geq k_{\text{sup}}$, 
\begin{align}\label{eq:thm_final_rlocal_lin}
\E[\|w_k - w^*\|] \leq r_{k_{\text{sup}}}(\sqrt{a_l})^{k - k_{\text{sup}}}.
\end{align}

We use induction to prove this statement. It is trivially satisfied for $k=k_{\text{sup}}$. Suppose this statement is true for some $k\geq k_{\text{sup}}$. Then, we have
\begin{align*}
    \E[\|w_{k+1} - w^*\|] &\leq \max\left\{ r_k \rho_k, r_{k_{\text{sup}}}(\sqrt{a_l})^{k-k_{\text{sup}}+1} \right\} \leq \max\left\{ r_{k_{\text{sup}}} (\sqrt{a_l})^{k - k_{\text{sup}}}  \rho_k, r_{k_{\text{sup}}}(\sqrt{a_l})^{k-k_{\text{sup}}+1} \right\} \\&= r_{k_{\text{sup}}}(\sqrt{a_l})^{k-k_{\text{sup}}+1},
\end{align*}
where $\rho_k$ is defined as in Theorem~\ref{thm:super_stoc}.
Substituting $K_{\epsilon}$ in \eqref{eq:thm_final_rlocal_lin}, we get
\begin{align*}
    \E[\|w_k - w^*\|] \leq r_{k_{\text{sup}}}(\sqrt{a_l})^{\left\lceil \log_{1/\sqrt{a_l}}\left(\tfrac{r_{k_{\text{sup}}}}{\sqrt{\epsilon}}\right) \right\rceil} \leq \sqrt{\epsilon}. 
\end{align*}
We will now analyze $k_{\text{sup}}$. Following similar steps in the proof of Corollary~\ref{cor:ksup_d}, we have that $k_{\text{lin}} = 0$ and also utilizing 
\begin{align}\label{eq:proof_sum_1_exp_update}
    \sum_{i=k_{0}}^{k}(\E[\|w_i - w^*\|^2])^{1/2} & < \frac{\sqrt{2\tilde{C}_2}}{\sqrt{\mu} (1-\sqrt{\tilde{\rho}_2})},
\end{align}
we get $C_{p,s} : = \frac{M\sqrt{2\tilde{C}_2}}{\sqrt{\mu}(1-\sqrt{\tilde{\rho}_2})}$. Therefore $k_{\text{sup}}$ in \eqref{eq:ksup_exp}, is updated as, 
\begin{align}
k_{\text{sup}} &= \bigg\lceil \max \bigg\{  \tfrac{9C_{p,s}\sqrt{\eta}}{\tilde{\mu}\sqrt{a_l}}, \tfrac{81\sigma^2_H\eta}{\tilde{\mu}^2|S_0|a_l}, 2\log_{1/\tilde{\rho}_2}\left(\tfrac{3q\sqrt{2\tilde{C}_2}}{\sqrt{\mu} \sqrt{a_l}}\right), \log_{1/a_g}\left(\tfrac{81q^2 \iota_{0}}{a_l^2\tilde{\mu}^2}\right)\bigg\}\bigg\rceil \label{eq:ksup_exp_update} \\
&= \mathcal{\tilde{O}}\left(\max \bigg\{\tfrac{\kappa^2 M}{\mu^{3/2}}, \tfrac{\sigma_H^2}{\tilde{\mu}^2|S_0|}, \kappa^2, \kappa^2  \bigg\} \right) =  \mathcal{\tilde{O}}\left(\kappa^2\left(1 + \tfrac{M}{\mu^{3/2}}\right)\right), \nonumber
\end{align}
where we employed $\log(1 - x) \approx - x$ for sufficiently small $x$ and $\sigma_H^2 \leq L^2$. 

\end{proof}

\begin{remark}
We make the following remarks about this result. 
\begin{itemize}
    \item The number of iterations required to transition from global linear phase to fast local linear or superlinear phase $k_{\text{sup}}$ is similar to that of the deterministic sampling results established in Corollary~\ref{cor:ksup_d} (excluding the dependence on $N$). Furthermore, $k_{\text{sup}}$ is better (smaller) than the results established for uniform weighting scheme and comparable to nonuniform weighting scheme when $\tfrac{M}{\mu^{3/2}}$ is sufficiently small. However, we note that the analysis is in expectation and requires additional assumptions related to bounded moments and strong convexity of subsampled functions, which are not required in the deterministic sampling settings (see Section~\ref{section:det_analysis}).
    \item While Corollary~\ref{corr:iter_complexity_det} requires $\epsilon$ to be sufficiently small, we can establish iteration complexity results for the global phase when $\epsilon$ is large. In this the iteration complexity result is given as $K_{\epsilon} = \tilde{\mathcal{O}}\left(\kappa^2 \log\left(\tfrac{1}{\mu\sqrt{\epsilon}}\right)\right)$.
    \item This iteration complexity compared to a stochastic gradient method or an adaptive sampling gradient method has better dependence on $\epsilon$ as seen in Table~\ref{table:complexity}.  
\end{itemize}
\end{remark}
 We will now establish the total gradient evaluations required to achieve an $\epsilon$-accurate solution when the starting iterate is close enough to the optimum, $w^*$. We will only consider simpler settings common in stochastic gradient analysis where $\sigma_{1,g} = 0$ in Assumption~\ref{eq:exp_grad_var_bnd}.

\begin{corollary}
	\label{thm:grad_complexity}
	Suppose conditions of Corollary~\ref{corr:iter_complexity_det} are satisfied and $\sigma_{1,g} = 0$ in Assumption~\ref{eq:exp_grad_var_bnd}. In addition, if $w_0$ is sufficiently close to $w^*$ such that $f(w_0) - f(w^*) \leq \tfrac{L\iota_0}{\mu^2}$ and $\iota_0 \leq \tfrac{\mu^5}{324\eta^2 L M^2}$. Let $a_l = \tfrac{1}{2}$, and $|S_0| = \lceil \tfrac{81 \eta}{2}\rceil $. Then, after computing  
	\begin{align}
		\mathcal{W}_{g} = \mathcal{O}\left(\kappa^2\sigma^2_{2,g} + \tfrac{\kappa \sigma^2_{2,g}}{\mu^2\epsilon}\right)
	\end{align}
	stochastic gradients, we achieve an $\epsilon$-accurate solution. 
\end{corollary}
\begin{proof}
	From \eqref{eq:gradsamp_bnd_det}, choosing minimum number of samples $|X_k|$ at each iteration to satisfy the stochastic norm condition \cite{byrd2012sample}, we get
	\begin{align*}
		|X_k| = \left\lceil \tfrac{\sigma_{2,g}^2} {\iota_k} \right\rceil \leq \tfrac{\sigma_{2,g}^2} {\iota_k} + 1.
	\end{align*}
	The total number of gradient evaluations required to achieve an $\epsilon$-accurate solution is then given as, 
	\begin{align}
		\mathcal{W}_{g} &= \sum_{i=0}^{K_{\text{sup}}-1} |X_k| + \sum_{i=k_{\text{sup}}}^{K_{\epsilon}} |X_k| \nonumber \\ &=\tfrac{\sigma_{2,g}^2}{\iota_0}\sum_{i=0}^{K_{\text{sup}}-1} \tfrac{1}{a_g^i} + \tfrac{\sigma_{2,g}^2}{\iota_{k_{\text{sup}}}}\sum_{i=k_{\text{sup}}}^{K_{\epsilon}} \tfrac{1}{a_l^{i- k_{\text{sup}}}} + k_{\text{sup}} + 1 \nonumber \\
		&\leq \tfrac{\sigma_{2,g}^2}{\tfrac{1}{a_g} - 1}(\tfrac{1}{a_g})^{k_{\text{sup}}} + \tfrac{\sigma_{2,g}^2 a_l}{\iota_{k_{\text{sup}}}(1 - a_l)}(\tfrac{1}{a_l})^{K_{\epsilon} - k_{\text{sup}}} + k_{\text{sup}} + 1 \nonumber \\
		&\leq \underbrace{2\kappa^2 \sigma_{2,g}^2 (1 - \tfrac{1}{2\kappa^2})^{-k_{\text{sup}}}}_{\circled{1}} + \underbrace{\tfrac{\sigma_{2,g}^2 a_l}{\iota_{k_{\text{sup}}}(1 - a_l)}(\tfrac{1}{a_l})^{K_{\epsilon} - k_{\text{sup}}}}_{\circled{2}} + \underbrace{k_{\text{sup}} + 1}_{\circled{3}} \label{eq:proof_threeterm_1}
	\end{align}
	where we used $a_g = 1 - \tfrac{\mu \tilde{\mu}}{2L\tilde{L}} = 1 - \tfrac{1}{2\kappa^2}$. 
	We will now analyze the terms on the right hand side of \eqref{eq:proof_threeterm_1}. Each term requires the analysis of $k_\text{sup}$, so we proceed by analyzing that term. Using $f(w_0) - f(w^*) \leq \tfrac{L\iota_0}{\mu^2}$ and $\iota_0 \leq \tfrac{\mu^5}{324L\eta^2 M^2}$,  we get,
	\begin{equation}
		\label{eq:proof_comp_fin_c2}
		\tilde{C}_2 = \max\left\{f(w_0) - f(w^*), \tfrac{\tilde{L}\iota_0}{\mu\tilde{\mu}}\right\} =\tfrac{L\iota_0}{\mu^2} \leq \tfrac{\mu^3}{324\eta^2 M^2}.
	\end{equation}
	From \eqref{eq:ksup_exp_update}, we have that
	\begin{align}
		k_{\text{sup}} &= \bigg\lceil \max \bigg\{  \tfrac{9C_{p,s}\sqrt{\eta}}{\tilde{\mu}\sqrt{a_l}}, \tfrac{81\sigma^2_H\eta}{\tilde{\mu}^2|S_0|a_l}, 2\log_{1/\tilde{\rho}_2}\left(\tfrac{3q\sqrt{2\tilde{C}_2}}{\sqrt{\mu} \sqrt{a_l}}\right), \log_{1/a_g}\left(\tfrac{81q^2 \iota_{0}}{a_l^2\tilde{\mu}^2}\right)\bigg\}\bigg\rceil \nonumber \\
		&\leq \max \bigg\{  \tfrac{9C_{p,s}\sqrt{\eta}}{\tilde{\mu}\sqrt{a_l}}, \tfrac{81\sigma^2_H\eta}{\tilde{\mu}^2|S_0|a_l}, 2\log_{1/\tilde{\rho}_2}\left(\tfrac{3q\sqrt{2\tilde{C}_2}}{\sqrt{\mu} \sqrt{a_l}}\right), \log_{1/a_g}\left(\tfrac{81q^2 \iota_{0}}{a_l^2\tilde{\mu}^2}\right)\bigg\} + 1 \nonumber 
	\end{align}
	We will now analyze each term in the above bound. Consider,
	\begin{align}
		\tfrac{9C_{p,s}\sqrt{\eta}}{\tilde{\mu}\sqrt{a_l}} &\leq \tfrac{9M\sqrt{2\tilde{C_2}}\sqrt{\eta}}{\sqrt{\mu}\tilde{\mu}\sqrt{a_l}(1 - \sqrt{\tilde{\rho}_2})}\leq \tfrac{1}{\sqrt{\eta}(1 - \sqrt{\tilde{\rho}_2})} \leq \tfrac{2}{1 - \tilde{\rho}_2} = 4\kappa^2,
	\end{align}	
	where the first inequality is due to $C_{p,s} : = \frac{M\sqrt{2\tilde{C}_2}}{\sqrt{\mu}(1-\sqrt{\tilde{\rho}_2})}$ for strongly convex functions, the second inequality is due to \eqref{eq:proof_comp_fin_c2} and $\eta \geq 1$, the last inequality is due to the fact that $1 - \sqrt{x} \geq \tfrac{1 - x}{2}$ for any $x\in[0,1]$ and $\eta \geq 1$, and the equality is due to $\tilde{\rho}_2 = 1 - \tfrac{\mu\tilde{\mu}}{2L\tilde{L}} = 1 - \tfrac{1}{2\kappa^2}$.  
	
	Considering the second term and using $|S_0| =  \lceil \tfrac{81 \eta}{2}\rceil $ and $\sigma^2_H \leq L^2$, we get,
	\begin{align}
		\tfrac{81\sigma^2_H\eta}{\tilde{\mu}^2|S_0|a_l} \leq \tfrac{4\sigma^2_H}{\tilde{\mu}^2} \leq 4\kappa^2. 
	\end{align}	
	
	Using \eqref{eq:proof_comp_fin_c2}, $\log(1 - \tfrac{1}{2\kappa^2}) \approx -\tfrac{1}{2\kappa^2}$, and substituting $q = \frac{7M\eta}{2\tilde{\mu}}$ in the third term, we get, 
	\begin{align}
		2\log_{1/\tilde{\rho}_2}\left(\tfrac{3q\sqrt{2\tilde{C}_2}}{\sqrt{\mu} \sqrt{a_l}}\right) \leq -2\tfrac{\log(7/6)}{\log \left(\tilde{\rho_2}\right) } \approx 4\kappa^2.
	\end{align}
	Similarly, using $\log(1 - \tfrac{1}{2\kappa^2}) \approx -\tfrac{1}{2\kappa^2}$, $\iota_0 \leq \tfrac{\mu^5}{324L\eta^2 M^2}$, and substituting $q = \frac{7M\eta}{2\tilde{\mu}}$ in the fourth term, we get, 
	\begin{align}
		\log_{1/a_g}\left(\tfrac{81q^2 \iota_{0}}{a_l^2\tilde{\mu}^2}\right) = - \tfrac{\log \left(\tfrac{3969M^2\eta^2\iota_0}{\mu^4}\right)}{\log\left(1 - \tfrac{\mu \tilde{\mu}}{2L\tilde{L}}\right)}  \leq -\tfrac{\log(\tfrac{49}{4\kappa})}{\log\left(1 - \frac{\mu \tilde{\mu}}{2L\tilde{L}}\right)} \approx 5\kappa^2.
	\end{align}
    Therefore, combining all the above bounds, we get, $k_{\text{sup}} \approx 5\kappa^2 + 1$. 

    We are ready to analyze the terms in \eqref{eq:proof_threeterm_1}. For the first term, we have that
    \begin{align}\label{eq:proof_cor_fin_11}
        2\kappa^2 \sigma_{2,g}^2 \left(1 - \tfrac{1}{2\kappa^2}\right)^{-k_{\text{sup}}} = \mathcal{O}\left(\kappa^2\sigma_{2,g}^2\right),
    \end{align}
    where we used $(1 - 1/x)^{-x} \approx \exp$, for any $x > 0$ that is sufficiently large. 
    Consider, 
	\begin{align}
		\tfrac{r_{k_{\text{sup}}}^2}{\iota_{k_{\text{sup}}}} &\leq \max\left\{ \tfrac{(\E[\|w_{k_{\text{sup}}} - w^*\|])^2}{\iota_0 a_g^{k_{\text{sup}}}}, \tfrac{18}{\tilde{\mu}^2}\right\} \leq \max\left\{ \tfrac{2\tilde{C}_2\rho_2^{k_{\text{sup}}}}{\mu\iota_0 a_g^{k_{\text{sup}}}}, \tfrac{18}{\tilde{\mu}^2}\right\} \leq \max\left\{ \tfrac{2\tilde{C}_2}{\mu\iota_0}, \tfrac{18}{\tilde{\mu}^2}\right\} 
		\leq \max\left\{ \tfrac{2\kappa}{\mu^2}, \tfrac{18}{\tilde{\mu}^2}\right\} \label{eq:proof_comp_fin_rksup}
	\end{align}
	where the last inequality is due to \eqref{eq:proof_comp_fin_c2}. Now, consider the second term in \eqref{eq:proof_threeterm_1}, we get 
	\begin{align}
		\tfrac{\sigma_{2,g}^2 a_l}{\iota_{k_{\text{sup}}}(1 - a_l)}\left(\tfrac{1}{a_l}\right)^{K_{\epsilon} - k_{\text{sup}}} &= \tfrac{\sigma_{2,g}^2 a_l}{\iota_{k_{\text{sup}}}(1 - a_l)}\left(\tfrac{1}{a_l}\right)^{\left\lceil \log_{1/\sqrt{a_l}}\left(\tfrac{r_{k_{\text{sup}}}}{\sqrt{\epsilon}}\right) \right\rceil} 
		\leq \tfrac{\sigma_{2,g}^2r_{k_{\text{sup}}}^2}{\epsilon \iota_{k_{\text{sup}}}} + \tfrac{2\sigma_{2,g}^2}{\iota_{k_{\text{sup}}}} = \mathcal{O}\left(\tfrac{\kappa \sigma^2_{2,g}}{\mu^2 \epsilon}\right)\label{eq:proof_threeterm_1_2} 
	\end{align}
	where the first equality is due to \eqref{eq:kepsilob_stoc} and the last equality is due to \eqref{eq:proof_comp_fin_rksup}. Combining \eqref{eq:proof_cor_fin_11} and \eqref{eq:proof_threeterm_1_2} completes the proof. 
\end{proof}

\begin{remark}
We make the following remarks about this result and make a comparison with existing results in Table~\ref{table:complexity}. 
    \begin{itemize}
        \item We employ fixed number of stochastic Hessian samples at each iteration. Therefore, the Hessian complexity, i.e. number of Hessian computations required to get an $\epsilon$-accurate solution is same as the total number of iterations $\left(\mathcal{\tilde{O}}\left(\kappa^2 + \log\left(\tfrac{1}{\epsilon}\right)\right)\right)$.  
        \item The total number of stochastic gradients required improves upon stochastic gradient method in terms of the dependence on condition number ($\kappa$) even though the number of gradients evaluated per-iteration are increasing at each iteration. Furthermore,  we should note that our local results match with that of first-order adaptive gradient sampling methods even after employing inferior global convergence results associated with Newton's method compared to a gradient method.
        \item Although we only presented the gradient iteration complexity results here for the phase where the the starting iterate is sufficiently close to the optimal solution and $\epsilon$ is sufficiently small (see Corollaries~\ref{corr:iter_complexity_det} and~\ref{thm:grad_complexity}), we can also establish the results for the global phase where $\iota_0$ is chosen independent of the problem characteristics and $K_{\epsilon} = \tilde{\mathcal{O}}\left(\kappa^2 \log\left(\tfrac{1}{\epsilon}\right)\right)$. However such results are inferior to those of first-order adaptive gradient methods due to the artifact of global convergence analysis of Newton's method.  
    \end{itemize}
\end{remark}

\begin{table}[H]
\centering
{\renewcommand{\arraystretch}{1.5}
\begin{tabular}{|c|c|c|}
\hline
 Method         & Iteration & Gradient  \\ \hline
Stochastic Gradient \cite{bottou2007tradeoffs} & $\mathcal{O}\left(\tfrac{\kappa^2\sigma^2_{2,g}}{\mu^2\epsilon}\right)$ & $\mathcal{O}(\tfrac{\kappa^2\sigma^2_{2,g}}{\mu^2\epsilon})$ \\ \hline
Adaptive Gradient Sampling \cite{friedlander2012hybrid,byrd2012sample} & $\mathcal{O}(\kappa \log\left(\tfrac{1}{\epsilon}\right))$ & $\mathcal{O}\left(\tfrac{\kappa\sigma^2_{2,g}}{\mu^2\epsilon}\right)$  \\ \hline
This Paper (Corollary~\ref{corr:iter_complexity_det} \&~\ref{thm:grad_complexity}) & $\mathcal{\tilde{O}}\left(\kappa^2 + \log\left(\tfrac{1}{\epsilon}\right)\right)$ & $\mathcal{O}\left(\kappa^2\sigma^2_{2,g} + \tfrac{\kappa \sigma^2_{2,g}}{\mu^2\epsilon}\right)$ \\ \hline

\end{tabular}
}
\caption{Comparison of (local) iteration and gradient complexity results for strongly convex functions. Note: The complexity results in terms of expected function values ($\E[f(w_k) - f(w^*)] \leq \epsilon$) for stochastic gradient and adaptive gradient sampling summarized in \cite[Table 4.1]{byrd2012sample} have been converted to get the results in terms of $(\E[\|w_k - w^*\|])^2 \leq \epsilon$. }
\label{table:complexity}
\end{table}

%% file: sections/tikz_diagrams/analysis_roadmap_exp.tex
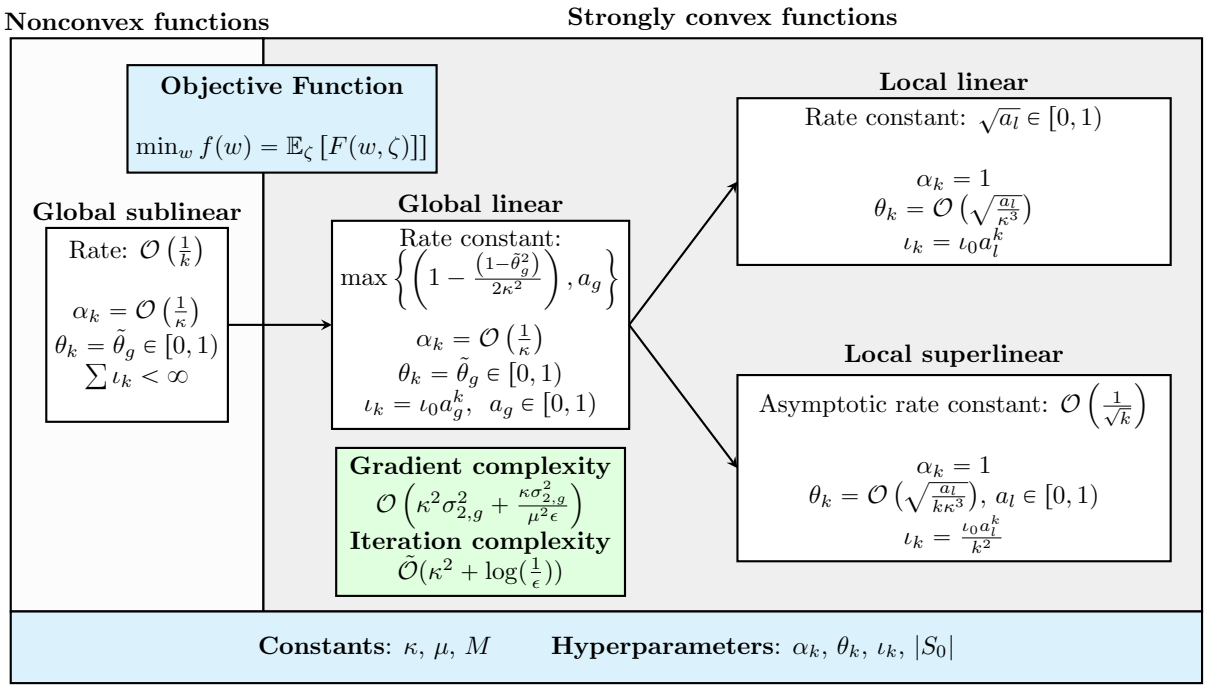
\begin{figure}[H]
\center
\begin{tikzpicture}[scale = 0.95, transform shape,every node/.style={draw,outer sep=0pt,thick},box/.style={draw, rectangle, minimum width=1cm,}]
\node[box,fill=lightgray!5] at (0,0)  [minimum width=3.5cm,minimum height=8cm, label=\textbf{Nonconvex functions}] (NCBB) {};
  
\node[bag, fill=white] (GlobalSub) at (0,0) [minimum width=2cm,minimum height=2cm,label=\textbf{Global sublinear}] {Rate: $\mathcal{O}\left(\tfrac{1}{k}\right)$\\ \enskip \\$\alpha_k = \mathcal{O}\left(\tfrac{1}{\kappa}\right)  $\\$\theta_k = \tilde{\theta}_g \in [0,1)$\\$\sum\iota_k < \infty$ \\ };

\node[box,fill=lightgray!25] at ($(NCBB.east)+(6.5,0)$)  [minimum width=13cm,minimum height=8cm, label=\textbf{Strongly convex functions}] (SCBB) {};

\node[bag, fill=white] (GlobalLin) at ($(GlobalSub.east)+(3.5,0)$) [minimum width=2cm,minimum height=2cm,label=\textbf{Global linear}] {Rate constant:\\ $\max\left\{\left(1 - \frac{\left(1-\tilde{\theta}_g^2\right)}{2\kappa^2}\right), a_g\right\}$\\\enskip \\$\alpha_k = \mathcal{O}\left(\tfrac{1}{\kappa}\right)  $\\$\theta_k = \tilde{\theta}_g \in [0,1)$\\$\iota_k = \iota_0a_g^k, \enskip a_g \in [0,1)$ };

\node[bag, fill=white] (LocalLin) at ($(GlobalLin.east)+(4.5,2)$) [minimum width=6cm,minimum height=2cm,label=\textbf{Local linear}] {Rate constant: $\sqrt{a_l}\in [0,1)$\\
\enskip\\
$\alpha_k = 1$\\
$\theta_k = \mathcal{O}\left(\sqrt{\frac{a_l}{\kappa^3}}\right)$\\$\iota_k = \iota_0a_l^k$ };

\node[bag, fill=white] (LocalSup) at ($(GlobalLin.east)+(4.5,-2)$) [minimum width=6cm,minimum height=2cm,label=\textbf{Local superlinear}] {Asymptotic rate constant: $\mathcal{O}\left(\frac{1}{\sqrt{k}}\right)$\\
\enskip\\
$\alpha_k = 1$\\
$\theta_k = \mathcal{O}\left(\sqrt{\frac{a_l}{k\kappa^3}}\right)$, $a_l \in [0,1)$\\$\iota_k = \frac{\iota_0a_l^k}{k^2}$ };

\draw[thick, -stealth] (GlobalSub.east) -- (GlobalLin.west);
\draw[thick, -stealth] (GlobalLin.east) -- (LocalLin.west);
\draw[thick, -stealth] (GlobalLin.east) -- (LocalSup.west);

\node[bag, fill=cyan!12.5] (Obj) at ($(GlobalSub.east)+(0.75,2.875)$) [minimum width=4cm,minimum height=1.5cm] {\textbf{Objective Function}\\\enskip\\$\min_w f(w) = \mathbb{E}_\zeta \left[ F(w,\zeta)] \right]$};

\node[bag, fill=green!12.5] (Obj) at ($(GlobalLin)+(0,-2.75)$) [minimum width=4cm,minimum height=1.5cm] {\textbf{Gradient complexity}\\ 
 $\mathcal{O}\left(\kappa^2\sigma^2_{2,g} + \tfrac{\kappa \sigma^2_{2,g}}{\mu^2\epsilon}\right)$\\
\textbf{Iteration complexity }\\$\mathcal{\tilde{O}}(\kappa^2 + \log(\tfrac{1}{\epsilon}))$};

\node[box,fill=cyan!12.5] at ($(NCBB.east)+(4.75,-4.5)$)  [minimum width=16.5cm,minimum height=1cm] (NCBB) {\textbf{Constants}: $\kappa$, $\mu$, $M$ \qquad  \textbf{Hyperparameters}: $\alpha_k$, $\theta_k$, $\iota_k$, $|S_0|$};

\end{tikzpicture}
\caption{Overview of the results presented in this section. We characterize the main results for global and local convergence results, and their relationship to the problem constants and algorithmic hyperparameters. 
Here $\mu$ is the Hessian spectral lower bound, $\kappa = \frac{L}{\mu}$ is a condition-number like constant, and $M$ is the Hessian Lipschitz constant.
We note that there are additional global sublinear and linear results that allow any $0<\tilde{\theta}_g<\infty$, but we do not annotate them in this figure for simplicity.}\label{fig:analysis_roadmap_exp} 
\end{figure}

%% file: sections/hessian_averaged_newton_methods.tex

In the previous sections we developed a convergence and complexity theory that establish the convergence benefits of Hessian-averaged Newton methods for the finite-sum and expectation problem settings.  However, hurdles remain to their deployment in very high-dimensional settings due to the costs of inverting Hessian matrices. In this section we discuss considerations relating to deploying Hessian-averaged Newton methods in practice. We begin by investigating a generic averaged Newton methods, which we refer to as fully-averaged Newton (FAN) for the purposes of our exposition. Due to the substantial costs of Hessian inversion this method is infeasible for even moderate problems. We propose a simple, diagonally averaged Newton method (Dan) for high-dimensional optimization problems such as those arising in machine learning (ML) settings. We investigate additional practical heuristics such as different weightings and norms that are used in the Hessian averaging, as well as practical gradient sampling strategies, which in turn lead to their own variant algorithms.

\subsection{Fully-Averaged Newton}
When the dimension of the optimization variable, $d$ is not too large, one may opt to construct a Newton-like algorithm using a weighted average of the full Hessian, which requires $d^2$ storage: 

\begin{equation} \label{eq:fan}
	\text{Fully-Averaged Newton (FAN):} \qquad p_k = \tilde{H}_k^{-1}g_k.
\end{equation}
Here $\tilde{H}_k$ is defined as in \eqref{eq:htilde_def}, or alternatively using the following formula:
\begin{equation}\label{eq:alternative_htilde}
	\tilde{H}_k = \sum_{i=0}^k\gamma_i|\nabla^2F_{S_i}(w_i)| + \mu I,
\end{equation}
where $\mu I$ can help improve the conditioning of the problem. This requires however costly eigenvalue or Cholesky decomposition at each iteration, thus incurring $\mathcal{O}(d^3)$ computation. Alternative methods to approximate inversion of the full Hessian include Krylov methods \cite[Chapter 5]{nocedal1999numerical}, however these methods do not simply extend to the averaging setting, since in this case one requires storage of the Hessian matrix in a given format, but not as a matrix-vector product callable as is necessary for Krylov methods.

\subsection{Overcoming high-dimensionality with Hessian-subspace products}

When the dimension $d$ is large, approximations of the Hessian that can be formed for less than $\mathcal{O}(d^2)$ operations and inverted for less than $\mathcal{O}(d^3)$ operations are necessary.
Randomized sketching is an easily extensible tool to construct efficient representations of matrices \cite{MartinssonTropp2020}. In this process one can construct a compressed representation of a matrix via its action on a matrix $V_r \in \mathbb{R}^{d\times r}$, where $r$ is a small number, that is often independent of $d$. In addition to diagonal approximations, other factorizations such as low rank \cite{HalkoMartinssonTropp2011} and hierarchical \cite{levitt2024linear,yesypenko2023randomized} approximations can be computed from the action of a given matrix on the $r$-dimensional subspace $V_r$.

This subspace action is easy to implement in modern ML workflows as it can be constructed from simple automatic differentiation tools around embarrassingly parallelizable linear algebra. For example, the Hessian subspace product can be computed with the same tools utilized to form the gradient; first by forming the gradient, then forming its transpose action on $V_r$, followed by taking the gradient (Jacobian) of this combined term:

\begin{equation}
	\text{Hessian subspace products:} \qquad \nabla^2 F_{S_k}(w_k) V_r = \nabla \left(\nabla F_{S_k}(w_k)^TV_r\right).
\end{equation}
E.g., forward-over reverse automatic differentiation \cite{baydin2018automatic} with vectorized GPU computing makes these Hessian approximations very computationally efficient. In our numerical tests, Hessian subspace products had approximately constant compute time until running out of GPU memory, see Figure \ref{fig:hessian_times_onGPU}. 

\begin{figure}[ht]
  \begin{center}
    \includegraphics[width=0.5\textwidth]{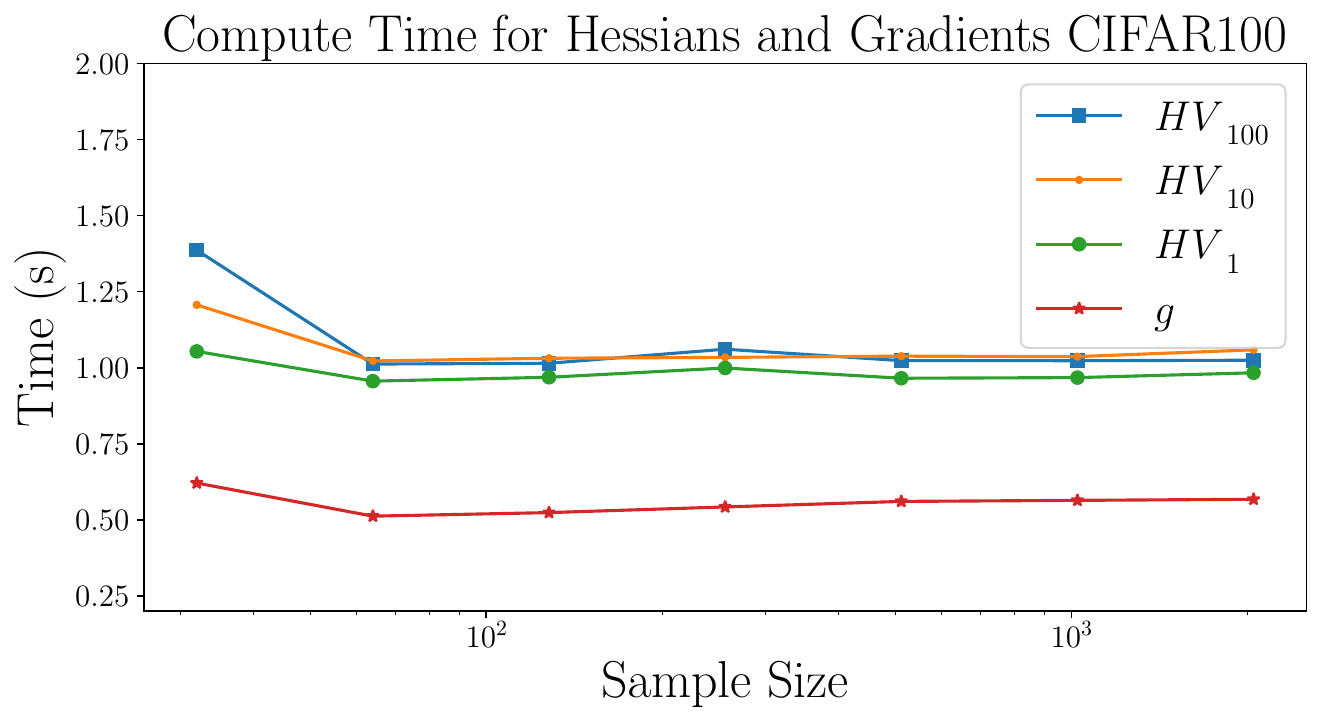}
  \end{center}
\caption{Efficiently implemented (vectorized) Hessian subspace products for varying ranks and sample sizes have approximately constant compute time until running out of GPU memory. These experiments are for a ResNet used in CIFAR100 classification, as shown in Section \ref{subsection:cifar_results}. The dimension of the neural network weights $w$ was $d =11,247,052$. This set of experiments was run on an NVIDIA L40S GPU which has 48GB of GPU RAM. }\label{fig:hessian_times_onGPU}
\end{figure}

\subsection{Diagonally-Averaged Newton }

In practical settings for stochastic optimization, such as deep learning, $d$ can be very large, and any algorithm requiring more than $\mathcal{O}(1)$ Hessian-vector products, and $\mathcal{O}(d)$ memory footprint at each iteration will be infeasible in modern compute settings, since they are typically memory bound \cite{gholami2024ai}. 

These constraints are achievable utilizing both diagonal or low rank Hessian approximations. As diagonal preconditioning of gradients (with and without momentum) dominates modern ML optimization methods, it is sensible to utilize diagonal Hessian approximations in Newton-like stochastic gradient updates that are targeted to these problems. The use of diagonal Hessian preconditioners was first introduced, to the best of our knowledge, by \cite{yao2021adahessian}, and was later used in \cite{liu2023sophia}. 

The Hessian diagonal can be computed matrix-free via randomized Hutchinson diagonal estimation 
\begin{align}\label{eq:hutchinson_diagonal_estimation}
	D_k = \text{diag}\left(\nabla^2F_{S_k}(w_k)\right)  = \mathbb{E}_{z \sim \pi_z} \left[ \text{diag}\left(\frac{z (\nabla^2F_{S_k}(w_k)z)^T}{z^Tz}\right)\right],
\end{align}
for a suitable choice of distribution $\pi_z$ \cite{dharangutte2023tight}. The diagonal matrix is trivially invertible in $d$ operations, overcoming the major computational hurdle for Newton-like methods. This approximation, and other matrix representations that are easily inverted and have $\mathcal{O}(d)$ memory footprint can be constructed via the use of randomized sketching. 

This naturally leads to the diagonally-averaged Newton method 
\begin{equation}\label{eq:dan}
	\text{Diagonally-Averaged Newton (Dan):} \qquad p_k = \big[\tilde{D}_k\big]^{-1}g_k,
\end{equation}
where in our implementation  $\tilde{D}_k = \sum_{i=1}^k \gamma_i |D_k|$, that is we approximate the diagonal of \eqref{eq:alternative_htilde} instead of \eqref{eq:htilde_def}. Diagonal approximations are useful for diagonally dominant Hessians, however these approximations may not be suitable for optimization problems with large off-diagonal components. We empirically observe that Dan and Adahessian perform well on difficult ML optimization problems.

\subsubsection{Differences with Adahessian}
Dan is similar to Adahessian \cite{yao2021adahessian} but differs in two key ways:
\begin{enumerate}
	\item Dan does not have any momentum in the gradient, so the relative performance of Dan to Adahessian helps isolate the effects of utilizing the averaged Hessian approximation in isolation from the effects of gradient momentum. 

	\item Dan utilizes a $\ell^1$ averaged approximation of the Hessian diagonal; this choice is motivated by our analysis which averages the Hessian and not its square. Adahessian and Sophia \cite{liu2023sophia} utilize an $\ell^2$ norm averaging, similar to what is used in Adam \cite{KingBa15}. We denote by $\ell^p_\gamma$ the averaging protocol:
 \begin{equation}
     \ell^p_\gamma \text{ averaging: }\qquad\sqrt[p]{\sum_{i=1}^k \gamma_i D_i^p}.
 \end{equation}
\end{enumerate}

Since Adahessian and Dan utilize different averagings, we thus propose an $\ell^2_\gamma$ modification of Dan, which we name Dan2:
\begin{equation}\label{eq:dan2}
    (\text{Dan2}): \qquad p_k = \left[\left(\sum_{i\leq k} \gamma_i D_i^2\right)^{\frac{1}{2}} \right]^{-1}g_k.
\end{equation}

\subsection{Additional algorithmic considerations}
In this section we consider additional adaptations of Dan and Dan2 that are common to other practical ML optimization methods.

\subsubsection{Infrequent Hessian computations} \label{subsection:infrequent_hessian_computations}

The per-iteration costs of the Hessian approximations can add significant additional costs relative to first-order methods. In traditional stochastic Newton methods, this burden can be lessened by (i) lower dimensional Hessian approximations (e.g., few samples for Hutchinson diagonal estimation), (ii) smaller Hessian sample size (i.e., $|S_k| < |X_k|$). In the context of Hessian averaging we can additionally lessen the burden by updating the Hessian approximation less frequently than the gradient, as is done in Adahessian. We utilize this in numerical experiments in Section \ref{subsection:cost_comparisons_cifar} where we maintain a fair cost-basis comparison between the first and second-order methods. 

\subsubsection{Non-uniform weightings}\label{subsection:weightings}

Adam, Adahessian and other popular ML optimizers utilize exponentially decaying sum averaging for their weightings, which is defined by the recurrence relationship 
\begin{equation} \label{eq:decay_weighting}
	\text{Decaying Weights:} \qquad \widehat{D}_{k+1} = \frac{\beta_2}{1 - \beta_2^k} \widehat{D}_{k} + \frac{(1 - \beta_2)}{1 - \beta_2^k}D_k,
\end{equation}
where $\beta_2\in (0,1)$ is a hyperparameter that controls the rate of decay in the averaging of past iterates, $\widehat{D}_k$ is the diagonal preconditioner being updated, and $D_k$ is the estimator of the diagonal at $w_k$. 
The benefits of this approach are that the effects of past iterates are de-emphasized, which is useful when the landscape is highly nonlinear, and the diagonal preconditioner's local information is changing rapidly. This weighting is widely used due to its ease of implementation and effectiveness.

Our local convergence analyses in Section \ref{section:det_analysis} and \ref{section:stoch_analysis} utilize uniform averaging to concentrate the Hessian statistical sampling errors. In order to prove superlinear convergence for this weighting deterministically, or in expectation, while still maintaining a fixed per-iteration Hessian sample size, other assumptions may be required. In \cite{na2023hessian} when the Hessian errors are assumed to be sub-Gaussian, superlinear local convergence rates are proven in probability for more general weightings such as \eqref{eq:decay_weighting}. In numerical experiments we demonstrate the performance of this weighting and the uniform weighting. 

\subsubsection{Gradient sample size selection}\label{subsection:stochastic_gradient_practical}

While our algorithmic framework allows for fixed per-iteration computational costs associated with the Hessian approximation, in order to achieve fast convergence one has to control the gradient statistical sampling error in relation to the \emph{true} gradient norm. This can be achieved either by geometrically increasing sample sizes (with sequence $\iota_k$), or the norm test (with sequence $\theta_k$), or a combination of both. The former (i.e., $\theta_k = 0$) is simple to implement in practice while the latter requires additional approximations as numerically verifying the bound $\mathbb{E}_k\left[\|\nabla F_{X_k}(w_k) - \nabla f(w_k)\|^2\right] \leq \theta_k^2 \|\nabla f(w_k)\|^2$ may be prohibitive due to its sampling costs. To address this challenge, one can employ an approximate norm test in practice, as was done in previous works \cite{byrd2012sample,bollapragada2018adaptive,bollapragada2018progressive}:
\begin{equation}
	\text{Approximate norm test:} \qquad \frac{1}{|S_k|}\sum_{i\in S_k} \left\|\nabla F_i(w_k) - \nabla F_{S_k}(w_k)\right\|^2  \leq \theta_k^2 \left\| \nabla F_{S_k}(w_k)\right\|^2 + \iota_k.
\end{equation}
This test approximates the expectations via Monte Carlo, and will give a rough indicator of the convergence of sample gradient to the true gradient.

In large scale subsampled optimization problems (e.g., deep learning), methods with fixed gradient sample sizes show empirically good performance.
Therefore, in numerical experiments we also consider variants of our algorithms that do not increase the gradient sample sizes, but instead use step size schedulers. This allows for direct comparison with state-of-the-art ML optimization routines. In particular Dan and Dan2 perform comparably to and often better than state-of-the-art methods in the ML optimization problems that we investigate.

%% file: sections/numerical_experiments.tex
In this section, we experiment with Hessian-averaged subsampled Newton methods on a variety of problems, separated into two classes with different metrics: (1) subsampled convex problems where we look at $f(w) - f(w^*)$, and (2) subsampled nonconvex (deep learning) problems where we are interested in the performance of the trained models on unseen data. First we investigate algorithmic trade-offs for stochastic quadratic minimization and logistic regression problems where computational costs allow us to consider full Hessian inversions (FAN). We then consider large-scale neural network problems: CIFAR[10,100] classification with ResNets and neural operator training, where the weight dimensions $d_W$ make full Hessian inversion prohibitive. In the large-scale context we investigate the performance of the Dan relative to Adam, Adahessian and SGD. 

Overall, the Hessian-averaged Newton methods were run in regimes (e.g., choice sample sizes and step sizes) that led to immediate instabilities for subsampled Newton methods not utilizing Hessian averaging. This point demonstrates the key algorithmic motivation for this type of method: to alleviate the instabilities of subsampled Newton methods at manageable costs. Additionally in the large-scale machine learning (ML) problems, Dan and Dan2 performed comparably to and often better than Adam \cite{KingBa15}, overall better than Adahessian \cite{yao2021adahessian}, and substantially better than SGD. This result suggests that the effects of Hessian averaging may be more beneficial than gradient momentum in some relevant practical settings. 

\subsection{Problem setups}
We give a high level overview of the experiments below. The approximate solution for a given method is given by $w^\dagger$, which we do not denote $w^*$ since it is not necessarily the true minimizer.

\subsubsection{Subsampled quadratic}
\begin{subequations}
First we investigate a stochastic quadratic minimization problem:
\begin{align}
    \text{(Subsampled quadratic):} \qquad &f(w) =  \mathbb{E}_{P_A,P_b}\left[\|P_AAw - P_bb\|^2 \right],\\
    \text{Evaluation criteria:} \qquad &f(w^\dagger)
\end{align}
where $P_A$, and $P_b$ are linear operators that randomly zero out entries in $A$ and $b$ respectively. That is at each iteration we randomly sample index sets for entries in $A$ and $b$ that are set to zero. This problem is a simple analogue to empirical risk minimization over a dataset.

\end{subequations}

\subsubsection{Logistic regression}

Second we consider two logistic regression for binary classification with $\ell^2$ regularization. Let $x_i$ be an input vector and $y_i\in \{-1,1\}$ be the corresponding output label 
\begin{subequations}
\begin{align}
\text{(Logistic Regression):} \qquad &f(w) = \frac{1}{n}\sum_{i=1}^n\log (1 + \exp(-y_i (w^Tx_i)) + \frac{1}{2n}{\|w\|^2}, \\
\text{Evaluation criteria:} \qquad &f(w^\dagger)
\end{align}
\end{subequations}

We note that this objective function is strongly convex \cite{berahas2020investigation}. We compare the performance of various methods on both the \texttt{ijcnn1} and \texttt{mushroom} datasets \cite{chang2011libsvm}.

\subsubsection{CIFAR[10,100] Classification}

Third, we consider the CIFAR10 and CIFAR 100 \cite{krizhevsky2009learning} classification using ResNet architectures \cite{he2016deep}, and a softmax cross-entropy loss function. Let $x_i$ be an input image, and $y_i$ be the corresponding vector label. Let $\phi(x_i,w)$ be the ResNet prediction in the pre-image of the softmax, then the problem is formulated as 
\begin{subequations}
    \begin{align}
        \text{(Softmax cross entropy):}\qquad &f(w) = -\frac{1}{n}\sum_{i=1}^n y_i^T \log\left( p(x_i,w)\right)  \text{  where  }  p_j(x_i,w) = \frac{e^{\phi_j(x_i,w)}}{\sum\limits_{i\in\text{Classes}}e^{\phi_j(x_i,w)}} \\
        \text{Evaluation criteria:} \qquad &\text{Correct classification percentage on unseen data}.
    \end{align}
\end{subequations}

\subsubsection{Parametric PDE Regression}

Finally, we consider regression problems for the approximation of parametric PDE input-output maps via neural networks (e.g., neural operators). We consider a coefficient-to-observable nonlinear reaction diffusion problem in a physical domain $\Omega \subset \mathbb{R}^2$. Here the input parameter $x \in \mathcal{X}=L^2(\Omega)$ is a heterogeneous spatially varying random field, with measure $\pi$. The PDE state $u\in \mathcal{U} = u_0 + H^1_0(\Omega)$, an affine space account for boundary conditions $u_0$. The outputs $y \in \mathbb{R}^{d_Y}$ represent finite observations of $u$ on a line in the domain.

\begin{subequations}
\begin{align}
    &\text{PDE:} \qquad\qquad\qquad\qquad -\nabla \cdot (e^x \nabla u) + cu^3  = f \qquad \text{ in } \Omega = (0,1)^2\\
    &\text{Boundary conditions:} \qquad u = 1 \text{ on } \Gamma_\text{top}, \qquad  u = 0 \text{ on } \Gamma_\text{bottom}, \qquad \nabla u \cdot \mathbf{n} &= 0 \text{ on } \Gamma_\text{sides}\\
    &\text{Parametric map:} \qquad\qquad\qquad x \mapsto y(x) = Bu(x),
\end{align}
\end{subequations}
where $\mathbf{n}$ is the unit normal to the domain $\Omega$, and $B:\mathcal{U} \rightarrow \mathbb{R}^{d_Y}$ is a linear restriction operator to points on line in the domain. The infinite-dimensional input functions are encoded to a finite input basis $\{\psi_i \in \mathcal{X}\}_{i=1}^r$ that is chosen to correspond to directions of the input space that the map is most sensitive to as in \cite{OLearyRoseberryVillaChenEtAl2022,OLearyRoseberryDuChaudhuriEtAl2021}. The corresponding coefficients $x_r\in \mathbb{R}^r$ are $(x_r)_j = (x,\psi_j)_\mathcal{X}$, where $(\cdot,\cdot)_\mathcal{X}$ is the inner product for $\mathcal{X}$. This choice of encoding makes the regression task naturally finite-dimensional, while still maintaining discretization invariance via the use of the infinite-dimensionally consistent basis vectors $\psi_i$.

The regression task is thus to learn an approximation of the map $x_r \mapsto y$ by a neural network $\phi_w$. We consider two formulations, first a parametric least squares formulation for learning the reduced coefficient map, which we refer to as the $L^2_\pi$ formulation. Second we consider a least squares formulation for learning the reduced coefficient map and its first Fr\'echet derivatives (e.g., derivative-informed neural operator (DINO) \cite{OLearyRoseberryChenVillaEtAl2022}), which we refer to as the $H^1_\pi$ formulation. 

\begin{subequations}
\label{eq:no_training_formulations}
\begin{align}
&(L^2_\pi \text{ training):} \qquad \min_w f(w) = \frac{1}{n}\sum_{i=1}^n\|y(x_r) - \phi_w(x_r)\|_2^2  \label{eq:l2_no_emp_risk}\\
&\qquad\text{Evaluation criteria:} \qquad \frac{\|y(x_r) - \phi_w(x_r)\|_2}{\|y(x_r)\|_2} \quad \text{ on unseen data } \\
&(H^1_\pi \text{ training):} \qquad \min_w f(w) = \frac{1}{n}\sum_{i=1}^n\left(\|y(x_r) - \phi_w(x_r)\|_2^2 +\|\nabla_{x_r}y(x_r)- \nabla_{x_r} \phi_w(x_r)\|_{F}^2\right)\label{eq:dino_emp_risk}\\
&\qquad\text{Evaluation criteria:} \qquad \frac{\|y(x_r) - \phi_w(x_r)\|_2}{\|y(x_r)\|_2},\frac{\|\nabla_{x_r}y(x_r) - \nabla_{x_r}\phi_w(x_r)\|_F}{\|\nabla_{x_r}y(x_r)\|_F}  \quad \text{ on unseen data }.
\end{align}
\end{subequations}

The $H^1_\pi$ formulation is particularly relevant when the surrogate is to be deployed in a setting where accurate derivatives are required, such as for the solution of optimization problems \cite{luo2023efficient}, or efficient Bayesian inference in function spaces \cite{cao2024efficient}. The $H^1_\pi$ training problem can thus be considered optimizing to optimize. We note the recent work \cite{zampini2024petscml} has also investigated the performance of second-order optimization methods for training parametric PDE surrogates.

\subsubsection{Additional details}

In order to have a one-to-one comparison of methods, our implementation of Adahessian differs slightly from the method proposed in \cite{yao2021adahessian}. First we consider versions of Adahessian that update the Hessian approximation both at each iteration, and also infrequently (the latter is proposed in \cite{yao2021adahessian}). Additionally we do not use averaging of convolution layers in the diagonal approximation, as we propose generic Hessian-vector products for Dan and Dan2 that are agnostic to the structure of what is being differentiated. Further details on the implementation details are given in Appendix \ref{appendix:numerics}, accompanied by an extended discussion of the parametric PDE problem in Appendix \ref{appendix:dino}.

\subsection{Subsampled Quadratic Minimization}\label{subsection:stochastic_quadratic}

\begin{figure}[h]
\center
\includegraphics[width = 0.7\textwidth]{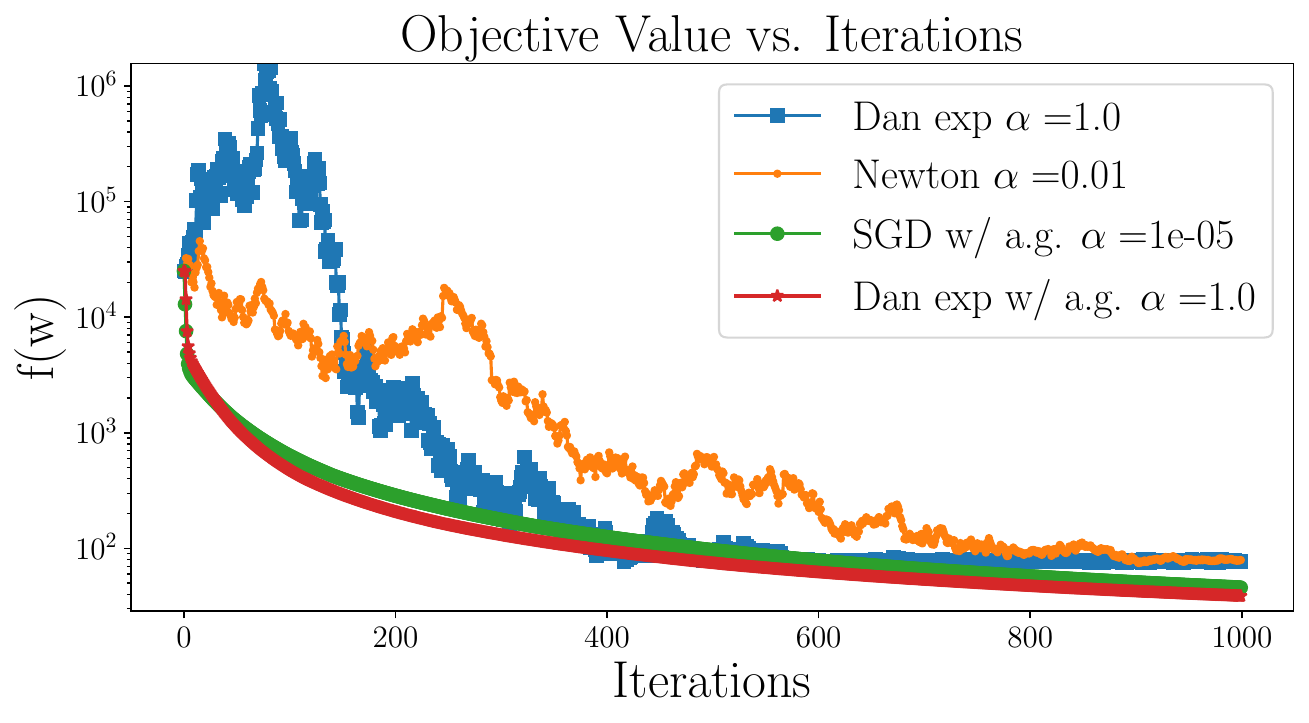}
\caption{The performance of the best 4 methods for the stochastic quadratic minimization problem. Dan with adaptive gradient sampling performed the best, both in terms of fast convergence and $f(w^\dagger)$. SGD with adaptive gradient sampling also performed well, but required significant limitations on the step size. The methods without adaptive gradient sampling plateaued with larger $f(w^\dagger)$. The averaging of the Hessian can eventually overcome stability issues after enough iterations progress to reduce the Hessian variance as seen by Dan exp $\alpha = 1.0$. When not using Hessian averaging, Newton methods required smaller steps to maintain stability. }

\label{fig:subquadresults}
\end{figure}

For this problem we compare FAN, Dan, Newton, and SGD all with and without adaptive gradient sampling (denoted ``a.g.''), which is implemented via the norm test. We do a sweep over fixed step sizes $\alpha \in[1.0,10^{-1},10^{-2},10^{-3},10^{-4},10^{-5}]$. For the averaged methods we consider both uniform and exponentially decaying weightings. Initially $P_A,P_b$ are both taken to randomly zero out $50\%$ of the entries in $A,b$ respectively. All optimizers start from the same initial guess $w_0 \sim \mathcal{N}(0,I_d{_W})$, where the stochastic gradient is approximately $13,500\%$ noisier than the true gradient. Each optimizer runs for $1,000$ iterations. 
We note that this does not constitute a fair comparison in terms of computational work, since we compare iterations. However, the methods that did not utilize adaptive gradient sampling were not making progress after $1,000$ iterations, so the gap in the objective function is indicative of superior performance for the adaptively sampled methods, and the trend would continue if run longer.
We report general trends over the $72$ different numerical runs. First all of the SGD runs without adaptive gradient sampling diverged except for $\alpha = 10^{-5}$; with adaptive gradient sampling all diverged except $\alpha \in \{10^{-4},10^{-5}\}$. 


The performance of the ten best methods is reported below in Table \ref{table:subquadresults}, and the (rolling average) of the stochastic function costs are shown in Figure \ref{fig:subquadresults}. The general trend for this problem is that the Dan methods outperformed all of the other Newton methods consistently. SGD required very small step sizes for stability, but performed well when taking many very small step sizes. The best performing methods utilized adaptive gradient sampling. Methods that did not utilize adaptive gradient sampling were able to reduce the objective function rapidly, but got stuck in neighborhoods where the objective function was not able to be reduced beyond the level of the noise. Perhaps surprising is the superior empirical performance of Dan to all other Newton methods. Dan is an economical Hessian approximation, where one might expect the omitted off-diagonal information would lead to deteriorated performance. For this problem, however, this is not the case. This bolsters the argument for this numerical approximation in high-dimensional settings where true Hessian inverse approximation is infeasible.

\begin{table}[H] 
\center
\begin{tabular}{|l|l|l|l|l|l|l|l|}
\hline
  & Method           & $\alpha$  & $f(w^\dagger)$ $\searrow$ &    & Method  & $\alpha$  & $f(w^\dagger)$ $\searrow$ \\ \hline
1 & Dan exp. w/ a.g. & 1.0       & \textbf{38.67}     & 6  & Newton  & $10^{-3}$ & 76.35     \\ \hline
2 & SGD w/ a.g.      & $10^{-5}$ & 45.69    & 7  & FAN exp & $10^{-1}$ & 76.98     \\ \hline
3 & Dan exp          & $10^{-1}$ & 73.77     & 8  & Dan exp & $10^{-1}$ & 77.01     \\ \hline
4 & Newton           & $10^{-2}$       & 74.69     & 9  & Dan exp & $10^{-2}$     & 77.11    \\ \hline
5 & SGD              & $10^{-5}$ & 75.11     & 10 & Dan uni  & $1.0$ & 77.12    \\ \hline
\end{tabular}
\caption{The ten best performing methods for the subsampled quadratic minimization, and the objective function value at their approximated optimum $w^\dagger$. The two best performing methods utilized adaptive gradient sampling, but interestingly the adaptive gradient sampling didn't lead to optimal results outside of these methods. In general Dan outperforms all of the other second-order methods, in particular the exponentially decaying weights generally outperformed the uniform averaging. }
\label{table:subquadresults}
\end{table}

For the remainder of the numerical results, we compare methods either for a fixed number of epochs or for fixed computational costs. In the case of a fixed total number of epochs, the adaptive sampling methods require fewer total computations than the same method without adaptive sampling, as these methods require fewer iterations and therefore fewer Hessian computations compared to the same method without adaptive sampling.

\subsection{Logistic Regression}

For a second numerical result we compare Dan, FAN, Newton, and SGD on two logistic regression problems. The \texttt{ijcnn1} dataset has $22$ input features and $35,000$ training samples. The \texttt{mushroom} dataset has $112$ features and $5,500$ training data. We compare fully subsampled methods (gradient and Hessian sample sizes fixed at 32), with adaptive gradient sampling methods. We run a total of 100 epochs for each method. For adaptive gradient sampling, we implement geometric growth in sample size instead of the norm test: we run 20 epochs with gradient sample sizes $|X_k| = \{32,128,512,2048,5500\}, \{32,128,512,2048,8192\}$ for \texttt{ijcnn1} and \texttt{mushroom} respectively. 
We note that in this case, since we consider epochs and not iterations, the adaptively sampled methods will require a lower total amount of computational work than the methods that do not utilize adaptive gradient sampling.
For the \text{ijcnn1} results the algorithms performed roughly the same, we show results for the step size of $\alpha = 0.1$, which are representative of other step sizes. For the \text{mushroom} dataset we show the performance of all methods as a function of the step size. The results are show below in Table \ref{table:logistic_results}.

\begin{table}[H]
\center
\begin{tabular}{|l|ll|lll|}
\hline
                        & \multicolumn{2}{c|}{\texttt{ijcnn1} $f(w^\dagger) \searrow$}                                 & \multicolumn{3}{c|}{\texttt{mushroom} $f(w^\dagger) \searrow$}                                   \\ \hline
                        & \multicolumn{1}{l|}{}                  & a.g.           & \multicolumn{1}{l|}{$\alpha$} & \multicolumn{1}{l|}{} & a.g. \\ \hline
\multirow{3}{*}{Dan}    & \multicolumn{1}{l|}{\multirow{3}{*}{}} & \multirow{3}{*}{} & \multicolumn{1}{l|}{0.1}      & \multicolumn{1}{l|}{$($\xmark$,0.123)$} & $($\xmark$,0.072)$        \\ \cline{4-6} 
                        & \multicolumn{1}{l|}{$0.398$}           &     $\mathbf{0.349}$       & \multicolumn{1}{l|}{0.01}     & \multicolumn{1}{l|}{$($\xmark$,0.123)$} & $($\xmark$,0.078)$        \\ \cline{4-6} 
                        & \multicolumn{1}{l|}{}                  &                   & \multicolumn{1}{l|}{0.001}    & \multicolumn{1}{l|}{$(0.136,0.137)$} &     $(0.105,0.102)$    \\ \hline
\multirow{3}{*}{FAN}    & \multicolumn{1}{l|}{\multirow{3}{*}{}} & \multirow{3}{*}{} & \multicolumn{1}{l|}{0.1}      & \multicolumn{1}{l|}{$0.121$} &  $\mathbf{0.071}$       \\ \cline{4-6} 
                        & \multicolumn{1}{l|}{$0.401$}           &     $\mathbf{0.349}$       & \multicolumn{1}{l|}{0.01}     & \multicolumn{1}{l|}{$0.123$} &  $0.079$       \\ \cline{4-6} 
                        & \multicolumn{1}{l|}{}                  &                   & \multicolumn{1}{l|}{0.001}    & \multicolumn{1}{l|}{$0.142$} &  $0.111$       \\ \hline
\multirow{3}{*}{Newton} & \multicolumn{1}{l|}{\multirow{3}{*}{}} & \multirow{3}{*}{} & \multicolumn{1}{l|}{0.1}      & \multicolumn{1}{l|}{$0.132$} &  $0.072$       \\ \cline{4-6} 
                        & \multicolumn{1}{l|}{$0.405$}           &     $\mathbf{0.349}$       & \multicolumn{1}{l|}{0.01}     & \multicolumn{1}{l|}{$0.132$} &  $0.072$ \\ \cline{4-6} 
                        & \multicolumn{1}{l|}{}                  &                   & \multicolumn{1}{l|}{0.001}    & \multicolumn{1}{l|}{$0.150$} &  $0.103$ \\ \hline
\multirow{3}{*}{SGD}    & \multicolumn{1}{l|}{\multirow{3}{*}{}} & \multirow{3}{*}{} & \multicolumn{1}{l|}{0.1}      & \multicolumn{1}{l|}{$0.123$} &  $0.092$ \\ \cline{4-6} 
                        & \multicolumn{1}{l|}{$0.389$}           &     $\mathbf{0.349}$       & \multicolumn{1}{l|}{0.01}     & \multicolumn{1}{l|}{$0.145$} &  $0.125$ \\ \cline{4-6} 
                        & \multicolumn{1}{l|}{}                  &                   & \multicolumn{1}{l|}{0.001}    & \multicolumn{1}{l|}{$0.198$} &  $0.201$ \\ \hline
\end{tabular}
\caption{Results for the two logistic regression problems are reported as averages over five different initial guesses. The methods performance improved uniformly with adaptive gradient sampling. For the \texttt{ijcnn1} dataset the four methods performed comparably. The symbol \xmark\enskip  denotes that these runs diverged. For the \texttt{mushroom} problem, Dan is reported for two different choices of the diagonal estimator rank, $ r \in \{1, 40\}$; for larger step sizes Dan diverged with $r = 1$, but otherwise performed comparably to FAN; this demonstrates that using more Hessian-vector products in the diagonal estimation can lead to better performance. FAN performed better than subsampled Newton without adaptive gradient sampling, and worse with it. Both Dan and FAN uniformly outperformed SGD. } \label{table:logistic_results}
\end{table}

\subsection{CIFAR[10,100] classification with ResNets}\label{subsection:cifar_results}

Next we demonstrate the performance of Dan in deep learning classication problems. We use the CIFAR10 and CIFAR 100 \cite{krizhevsky2009learning} classification using ResNet architectures \cite{he2016deep}, and a softmax cross-entropy loss function. For both problems, we use the standard 50,000 training data, and report generalization accuracy over the remaining test data. For all results we utilize a single learning rate scheduler that quarters the learning rate every 25\% of total epochs. For both problems we investigate the performance of the methods in two regimes:

\begin{enumerate}
    \item Comparison over 100 epochs. An extensive fixed data-access comparison of all methods with varying algorithmic hyperparameters. We note that this is not strictly a fair comparison, but it allows us to gain insight into the general performance of the different methods before running costlier, long trainings.

    \item Comparison of longer runs, on a fixed computational cost basis. We do a limited number of much longer training runs to compare how Dan, Dan2 and Adahessian compare to Adam.
\end{enumerate}

We note that other adaptive optimizers (e.g., RMSProp \cite{hinton2012neural} and Adagrad \cite{duchi2011adaptive}) performed substantially worse than Adam, so we omitted them. As these numerical results are very expensive to run, we rerun over two seeds and report the average results. The problems are quite computationally expensive and were run on NVIDIA A100 / L40S GPUs. 

\subsubsection{Comparison of methods over 100 epochs}

We begin with our comparison of methods over 100 epochs. For this set of results we investigate how the performance of the three Hessian-averaging based methods (Adahessian, Dan and Dan2) is effected by the number of Hessian vector products used in the Hutchinson diagonal estimation \eqref{eq:hutchinson_diagonal_estimation}. For a fair baseline of comparison, we additionally implement a version of Adahessian where a new Hessian approximation is computed at every iteration, as was discussed in Section \ref{subsection:infrequent_hessian_computations}. We explore the effects of limiting the frequency of Hessian computations in the next set of numerical results. We investigate the performance of Dan, Dan2 and SGD with and without adaptive gradient sampling. As with the last example, we implement a geometric increase of the gradient sample sizes. We first run $75$ epochs with $|X_k| = 32$, followed by five epochs each with the following sequence $X_k \in \{64,128,256,512,1024\}$. As a consequence the Dan w/ a.g. is notably substantially less expensive than regular Dan due to the lower iteration complexity in the final 25 epochs. Our results are shown below in Table \ref{table:cifar_results}.

\begin{table}[H]
\center
\begin{tabular}{|l|l|c|c|c|}
\hline
Method                       & $\alpha_0$ & Hessian rank  & CIFAR10 Accuracy $\nearrow$            & CIFAR100 Accuracy $\nearrow$ \\ \hline
\multirow{3}{*}{Adahessian}  & $0.05$    & $(1,5,10)$ & $(92.85,92.70,92.55)$            & $(71.97,71.67,72.15)$ \\ \cline{2-5} 
                             & $0.01$    & $(1,5,10)$ & $(92.47,93.20,93.21)$            & $(70.65,71.93,72.11)$ \\ \cline{2-5} 
                             & $0.001$   & $(1,5,10)$ & $(83.35,86.19,87.52)$            & $(57.49,60.10,62.82)$ \\ \hline
\multirow{3}{*}{Adam}        & $0.05$    & -          &  \xmark                          &   \xmark              \\ \cline{2-5} 
                             & $0.01$    & -          & $91.69$                          & $65.02$               \\ \cline{2-5} 
                             & $0.001$   & -          & $93.08$                          & $72.10$               \\ \hline
\multirow{3}{*}{Dan}         & $0.05$    & $(1,5,10)$ & $(91.80,92.31,90.74)$            & $(70.62,71.00,71.01)$ \\ \cline{2-5} 
                             & $0.01$    & $(1,5,10)$ & $(92.29,93.30,93.32)$            & $(71.18,72.10,71.57)^\dagger$ \\ \cline{2-5} 
                             & $0.001$   & $(1,5,10)$ & $(85.44,86.87,89.04)$            & $(60.33,61.78,64.98)$     \\ \hline
\multirow{3}{*}{Dan w/ a.g.} & $0.05$    & $(1,5,10)$ & $(92.99,92.16,92.56)$            & $(\mathbf{72.48},70.96,70.52)$ \\ \cline{2-5} 
                             & $0.01$    & $(1,5,10)$ & $(92.62,93.16,93.36)$            & $(71.11,72.03,72.22)$ \\ \cline{2-5} 
                             & $0.001$   & $(1,5,10)$ & $(82.41,86.78,89.13)$            & $(56.51,62.06,64.92)$     \\ \hline
\multirow{3}{*}{Dan2}        & $0.05$    & $(1,5,10)$ & $(92.81,92.80,92.40)$            & $(71.34,71.70,71.53)$             \\ \cline{2-5} 
                             & $0.01$    & $(1,5,10)$ & $(92.47,\textbf{93.37},93.14)$   & $(70.89,72.22,71.89)$             \\ \cline{2-5} 
                             & $0.001$   & $(1,5,10)$ & $(80.98,85.54,87.78)$                        & $(54.74,60.55,62.40)$             \\ \hline
\multirow{3}{*}{Dan2 w/ a.g.}& $0.05$    & $(1,5,10)$ & $(93.03,92.66,93.02)$            & $(72.09,72.25,72.33)$ \\ \cline{2-5} 
                             & $0.01$    & $(1,5,10)$ & $(91.90,92.99,92.96)$            & $(70.38,72.13,71.98)$ \\ \cline{2-5} 
                             & $0.001$   & $(1,5,10)$ & $(81.27,86.02,87.63)$            & $(54.74,60.55,62.40)$     \\ \hline
\multirow{3}{*}{SGD}         & $0.05$    & -          &   \xmark                         &   \xmark              \\ \cline{2-5} 
                             & $0.01$    & -          & $88.09$                          & $64.84$               \\ \cline{2-5} 
                             & $0.001$   & -          & $81.75$                          & $54.36$               \\ \hline
\multirow{3}{*}{SGD w/ a.g.} & $0.05$    & -          &   \xmark                         &   \xmark              \\ \cline{2-5} 
                             & $0.01$    & -          & $88.58$                          & $66.63$               \\ \cline{2-5} 
                             & $0.001$   & -          & $82.00$                          & $53.77$               \\ \hline
\end{tabular}
\caption{Comparison of optimization methods for CIFAR[10,100] ResNet training over 100 epochs. In both cases a Hessian-averaging based method produced the highest average generalization accuracy. Notable takeaways are that Dan with adaptive gradient led to the best result for the CIFAR100, showing the promise and competitiveness of our proposed framework in complicated deep learning tasks. Additionally, the Hessian averaging methods were able to take larger steps than the subsampled gradient methods. Adam was reliable as is usually the case, while SGD overfit substantially and produced poor  generalization accuracys. $\dagger$: See Appendix \ref{appendix:numerics_cifar}. }
\label{table:cifar_results}
\end{table}

Our proposed methods Dan and Dan2 produced the best overall results in terms of generalization accuracy. Adam and Adahessian were both competitive, while SGD was not competitive. The adaptive gradient sampling was able to improve the performance, for example Dan w/ a.g. for CIFAR100 notably produced significantly better results than the other methods. This demonstrates not only is Hessian averaging without momentum a good algorithmic building block, but additionally can lead to better results at lower costs, due to the geometric reduction in Hessian computational costs in the later epochs. Of note for other algorithmic considerations: in general the additional Hessian vector products in the diagonal estimations didn't clearly benefit performance, other than when taking small step sizes, in this case there is a clear trend that increasing the rank of the diagonal approximation led to better performance, however these methods performed substantially worth overall to the larger step sizes.




\subsubsection{Comparison of methods with respect to computational cost} \label{subsection:cost_comparisons_cifar}

In this section we investigate the performance of Adahessian, Dan and Dan2 in relation to Adam in terms of a computational cost, taking the additional Hessian computations into account. For this section we introduce a new algorithmic hyperparameter: the Hessian computation frequency, as discussed in \ref{subsection:infrequent_hessian_computations}. The modifications of Adahessian, Dan and Dan2 to compute Hessians infrequently significantly reduces the overall cost, yielding per-iteration costs closer to Adam's per-iteration cost. We only start utilizing infrequent Hessian computations after the first epoch, as empirically the Hessian statistical errors didn't concentrate fast enough to keep the iterates from diverging; a reminder of the inherent instability of subsampled Newton iterates in this regime. After the first epoch we only update the Hessian diagonal approximations every 10 iterations. We investigate a study of Hessian ranks of 1 and 5 for these methods. We introduce an epoch equivalent compute metric, as a means of putting the second-order and first-order methods on an equivalent cost basis:
\begin{equation}
    \text{(Epoch Equivalent Compute (E.E.C.))} =  \left(\underbrace{1}_\text{gradients} + \underbrace{\frac{2 \times\text{rank}}{\text{h.f.}}}_\text{Hessians} \right)\times\text{epochs} +\underbrace{ 2\times \text{rank}}_\text{first epoch Hessian}.
\end{equation}
We count one Hessian vector product as twice the cost of a gradient, when in reality it is empirically cheaper when utilizing vectorization (see Figure \ref{fig:hessian_times_onGPU}), these costs comparison are therefore conservative and make the Hessian-based methods seem more expensive than they may be given an efficient GPU implementation. An additional means of computational economy is reducing the Hessian sample size relative to the gradient sample size. We do not investigate this; as with the previous set of numerical experiments we use a sample size of 32 for both the gradients and the Hessians. In these experiments we do not investigate adaptive gradient sampling. Adam with $\alpha_0 = 10^{-3}$ is our reference method, we run it for 500, 1000 and 2000 epochs. We investigate Adahessian, Dan and Dan2 in comparison to Adam on a similar cost basis. These results are shown below in Table \ref{table:cifar_long_results}. The best individual runs over all hyperparameters and seeds for Adam, Adahessian, Dan and Dan2 are shown below in Figure \ref{fig:cifar_long}.

\begin{table}[H]
\center
\begin{tabular}{|l|l|c|c|l|c|c|ll}
\cline{1-7}
Method                      & epochs                & Hessian rank  & E.E.C.   & $\alpha_0$ & CIFAR10 & CIFAR100 &  &  \\ \cline{1-7}
\multirow{3}{*}{Adam}       &       2000            &       -       & 2000   & $0.001$   & $94.00$ & $73.06$  &  &  \\ \cline{2-7}
                            &       1000            &       -       & 1000   & $0.001$   & $93.78$ & $72.39$  &  &  \\ \cline{2-7}
                            &       500             &       -       & 500    & $0.001$   & $93.58$ & $72.39$  &  &  \\ \cline{1-7}
\multirow{4}{*}{Adahessian} & \multirow{2}{*}{1000} &  $(1,5)$      &$(1202,2010)$ & $0.05$   & (\xmark ,$93.44)$ & $(72.96,73.03)$  &  &  \\ \cline{3-7}
                            &                       &  $(1,5)$      &$(1202,2010)$ & $0.01$   & $(93.16,93.97)$ & $(69.57,72.50)$  &  &  \\ \cline{2-7}
                            & \multirow{2}{*}{500}  &  $(1,5)$      &$(602,1010)$ & $0.05$   & $(93.65,93.32)$ & $(72.30,69.42)$  &  &  \\ \cline{3-7}
                            &                       &  $(1,5)$      & $(602,1010)$ & $0.01$   & $(93.07,93.60)$ & $(70.72,71.36)$  &  &  \\ \cline{1-7}
\multirow{4}{*}{Dan}        & \multirow{2}{*}{1000} &  $(1,5)$      & $(1202,2010)$ & $0.05$   & $(91.98.90.83)$ & $(\mathbf{73.74},73.01)$  &  &  \\ \cline{3-7}
                            &                       &  $(1,5)$      & $(1202,2010)$ & $0.01$   & $(93.56,\mathbf{94.10})$  & $(71.62,72.73)$ &  \\ \cline{2-7}
                            & \multirow{2}{*}{500}  &  $(1,5)$      & $(602,1010)$ & $0.05$   & $(92.31,91.32)$ & $(73.30,73.05)$  &  &  \\ \cline{3-7}
                            &                       &  $(1,5)$      & $(602,1010)$ & $0.01$   & $(93.03,93.70)$ & $(71.30,72.30)$  &  &  \\ \cline{1-7}
\multirow{4}{*}{Dan2}       & \multirow{2}{*}{1000} &  $(1,5)$      & $(1202,2010)$ & $0.05$   & $(93.72,93.71)$ & $(73.30,72.99)$  &  &  \\ \cline{3-7}
                            &                       &  $(1,5)$      & $(1202,2010)$ & $0.01$   & $(93.07,93.85)$ & $(71.82,72.91)$  &  &  \\ \cline{2-7}
                            & \multirow{2}{*}{500}  &  $(1,5)$      & $(602,1010)$ & $0.05$   & $(93.80,93.08)$ & $(72.61,73.05)$  &  &  \\ \cline{3-7}
                            &                       &  $(1,5)$      & $(602,1010)$ & $0.01$   & $(92.80,93.43)$ & $(70.94,72.29)$  &  &  \\ \cline{1-7}
\end{tabular}
\caption{Comparison of Adam, Adahessian, Dan and Dan2 for equivalent computational cost bases. For the CIFAR10 dataset the performance of Adam was comparable to the best performing second-order methods, with Dan performing the best, but with some variance in performance. For CIFAR100 the performance of the second-order methods was drastically better than Adam. This is consistent with the results in the preceding section. }
\label{table:cifar_long_results}
\end{table}

\begin{figure}[H]
     \centering
     \begin{subfigure}[b]{0.5\textwidth}
         \centering
         \includegraphics[width=\textwidth]{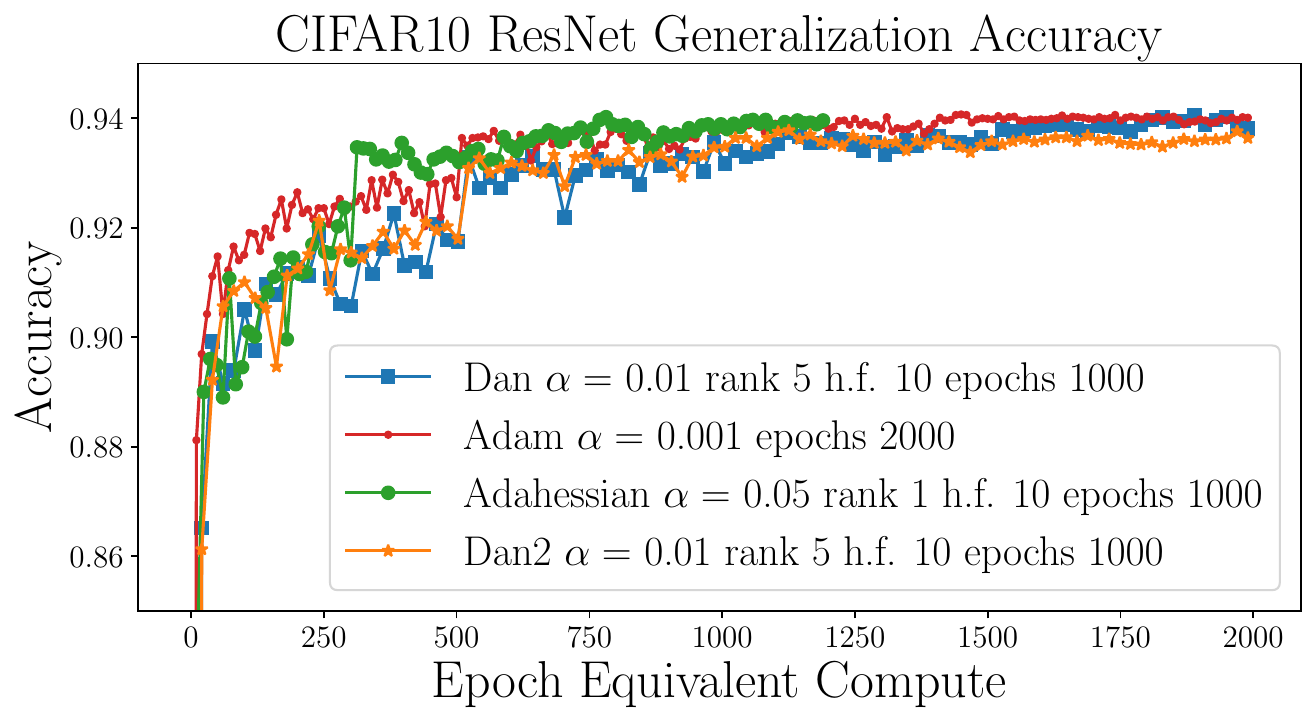}
     \end{subfigure}%
     \begin{subfigure}[b]{0.5\textwidth}
         \centering
         \includegraphics[width=\textwidth]{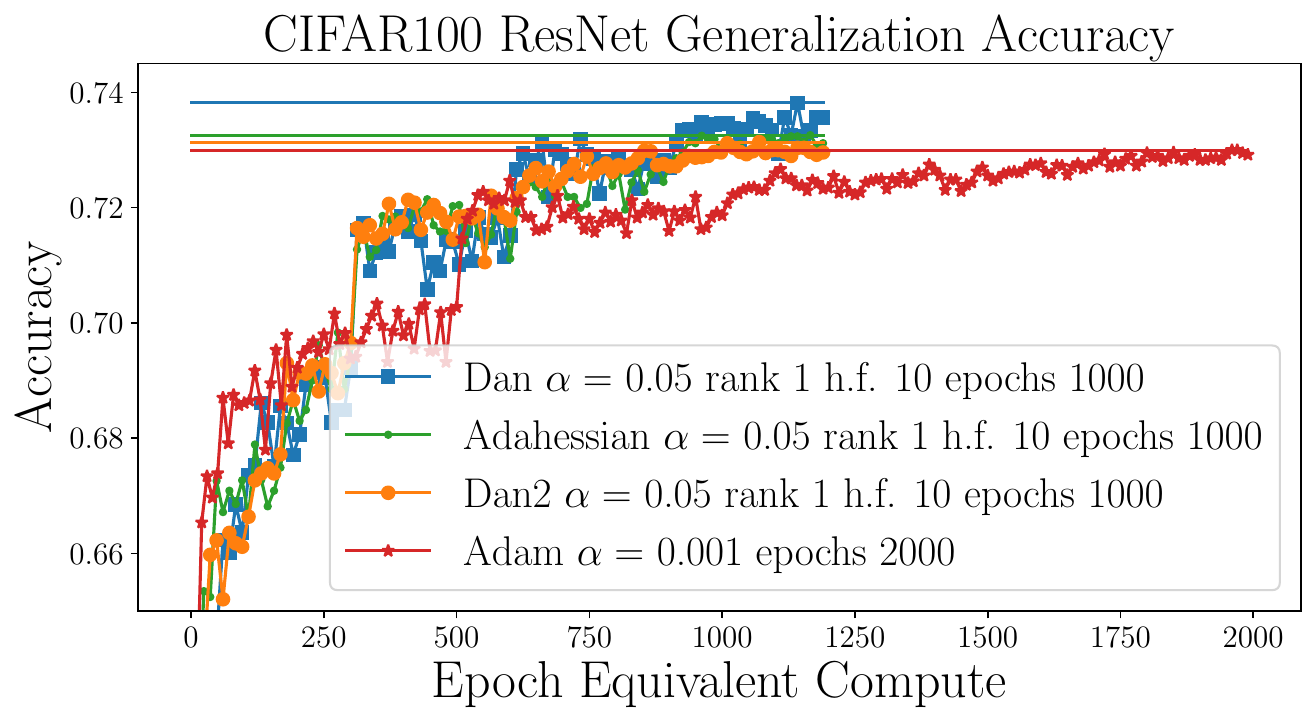}
     \end{subfigure}
        \caption{Comparison of the best runs for Adahessian, Adam, Dan and Dan2 in regards to epoch equivalent work. The best Adahessian happened to have lower Hessian ranks in the approximation than the best Dan and Dan2 for the CIFAR10, while for CIFAR100 the rank 1 methods shown all performed the best of individual runs. The CIFAR10 performance is quite similar for all four methods, while for CIFAR100 the Hessian-averaged methods substantially out-performed Adam; particularly Dan gave the best generalization accuracy (73.82\% for the run shown). }
        \label{fig:cifar_long}
\end{figure}

These results overall show the power of Hessian-averaged methods for difficult deep learning problems. While it was previously known that Adahessian performs well for these problems, these results show that the momentum in the gradient utilized by Adahessian may not be necessary to achieve good performance. Our algorithms (Dan, Dan2) do not use gradient momentum, and are competitive in challenging deep learning classification problems. 

\subsection{Parametric PDE Regression} \label{section:parametric_pde_numerics}

For a last numerical result we investigate the performance of Hessian averaged methods on some parametric PDE regression training problems; first where we learn a parametric map from input-output function data, and second where we include parametric (Fr\'{e}chet) derivatives as additional training data. These problems differ from the previous deep learning examples because (1) they are regression tasks, (2) the neural network representation is more compact, (3) there are substantially fewer sample data and (4) the derivative-learning task includes very rich information per sample. For additional details on neural operators see Appendix \ref{appendix:dino}

We train with each method for $200$ epochs on $4500$ samples of the PDE map (and its derivatives in the $H^1_\pi$ case), and use $500$ samples to compute generalization (relative) errors. We employ a one-step learning rate scheduler that reduces the step length by $10\times$ at the 150th epoch. We use a five layer feedforward network with gelu activation \cite{hendrycks2016gaussian}; the corresponding weight dimension is $d_W = 742,050$. Since the weights are lower dimensional, we numerically experiment with using larger rank Hessian approximations $k=(1,20,40)$, and demarcate the corresponding errors for these methods in tuples as such. The results are shown below in Table \ref{table:no_results}. In general for these problems, Adam performed reliably well in both the $L^2_\pi$ and $H^1_\pi$ training problems. The Hessian averaged methods performed worse than Adam on the $L^2_\pi$ problem. For the $H^1_\pi$ problem Dan and Dan2 performed about the same as Adam, with Adahessian performing slightly worse. SGD performed substantially worse than all other methods on this problem. We believe the superior performance of the Hessian averaged methods on the $H^1_\pi$ problem can be explain by the richer training data (e.g., the derivative training data) supplied for this problem; indeed the divergence experienced by Dan and Dan2 in the $L^2_\pi$ training is mitigated in the $H^1_\pi$ problem. Perhaps the additional data per iteration reduced statistical sampling errors that may have caused early iteration divergence for Dan and Dan2 in the $L^2_\pi$ training problem.

\begin{table}[]
\begin{tabular}{|l|l|l|ll|}
\hline
\multirow{2}{*}{Method}     & \multirow{2}{*}{$\alpha_0$} & $L^2_\pi$ training & \multicolumn{2}{l|}{$H^1_\pi$ training}                                               \\ \cline{3-5} 
                            &                             & ($y_r$ rel error $\searrow$ )      & \multicolumn{1}{l|}{($y_r$ rel error $\searrow$)}                    & ($\nabla_{x_r}y_r$ rel error  $\searrow$)       \\ \hline
\multirow{3}{*}{Adahessian} & $0.01$                      & $(0.366, 0.366, 0.362)$          & \multicolumn{1}{l|}{$(0.012, 0.010, $\xmark$)$}   & $(0.260, 0.258, $\xmark$)$    \\ \cline{2-5} 
                            & $0.005$                     & $(0.102, 0.097, 0.102)$         & \multicolumn{1}{l|}{$(0.010, 0.010, 0.011)$} & $(0.258, 0.258, 0.259)$ \\ \cline{2-5} 
                            & $0.001$                     & $(0.121, 0.12, 0.122)$           & \multicolumn{1}{l|}{$(0.015, 0.013, 0.012)$}  & $(0.264, 0.262, 0.260)$  \\ \hline
\multirow{3}{*}{Adam}       & $0.01$                      & $0.361$                            & \multicolumn{1}{l|}{$0.028$}                    & $0.274$                    \\ \cline{2-5} 
                            & $0.005$                     & $0.078$                            & \multicolumn{1}{l|}{$0.023$}                    & $0.264$                    \\ \cline{2-5} 
                            & $0.001$                     & $\mathbf{0.054}$                   & \multicolumn{1}{l|}{$0.008$}                    & $\mathbf{0.255}$                    \\ \hline
\multirow{3}{*}{Dan}        & $0.01$                      & $($\xmark,\xmark,\xmark$)$      & \multicolumn{1}{l|}{$(0.01, 0.008, $\xmark$)$}    & $(0.257, 0.256, $\xmark$)$   \\ \cline{2-5} 
                            & $0.005$                     & $(0.09, 0.084,$\xmark$)$         & \multicolumn{1}{l|}{$(0.009, 0.009, \mathbf{0.007})$} & $(0.257, 0.256, 0.256)$ \\ \cline{2-5} 
                            & $0.001$                     & $(0.12, 0.118, 0.114)$           & \multicolumn{1}{l|}{$(0.014, 0.013, 0.012)$} & $(0.263, 0.261, 0.260)$  \\ \hline
\multirow{3}{*}{Dan w/ a.g.}        & $0.01$                      & $($\xmark,\xmark,\xmark$)$      & \multicolumn{1}{l|}{$(0.011,0.009,\ddagger)$}    & $(0.258,0.257,\ddagger)$   \\ \cline{2-5} 
                            & $0.005$                     & $(0.108,0.088,0.078)$         & \multicolumn{1}{l|}{$(0.013,0.009,\ddagger)$} & $(0.261,0.256,\ddagger)$ \\ \cline{2-5} 
                            & $0.001$                     & $(0.145,0.134,0.126)$           & \multicolumn{1}{l|}{$(0.025,0.014,\ddagger)$} & $(0.273,0.263,\ddagger)$  \\ \hline
\multirow{3}{*}{Dan2}       & $0.01$                      & $($\xmark,\xmark,\xmark,$)$      & \multicolumn{1}{l|}{$(0.009, 0.009, 0.009)$} & $(0.257, 0.256, 0.256)$ \\ \cline{2-5} 
                            & $0.005$                     & $(0.092, 0.088, 0.084)$          & \multicolumn{1}{l|}{$(0.01, 0.009, 0.008)$}  & $(0.257, 0.257, 0.256)$ \\ \cline{2-5} 
                            & $0.001$                     & $(0.12, 0.119, 0.118)$          & \multicolumn{1}{l|}{$(0.015, 0.014, 0.012)$} & $(0.264, 0.262, 0.260)$  \\ \hline
\multirow{3}{*}{Dan2 w/ a.g.}        & $0.01$                      & $($\xmark,\xmark,\xmark$)$      & \multicolumn{1}{l|}{$(0.01,0.009,\ddagger)$}    & $(0.259,0.256,\ddagger)$   \\ \cline{2-5} 
                            & $0.005$                     & $(0.134,0.088,0.082)$         & \multicolumn{1}{l|}{$(0.014,0.009,\ddagger)$} & $(0.262,0.257,\ddagger)$ \\ \cline{2-5} 
                            & $0.001$                     & $(0.144,0.137,0.131)$           & \multicolumn{1}{l|}{$(0.027,0.013,\ddagger)$} & $(0.284,0.263,\ddagger)$  \\ \hline
\multirow{3}{*}{SGD}        & $0.01$                      & $0.283$                            & \multicolumn{1}{l|}{$0.046$}                    & $0.313$                    \\ \cline{2-5} 
                            & $0.005$                     & $0.154$                            & \multicolumn{1}{l|}{$0.042$}                    & $0.310$                    \\ \cline{2-5} 
                            & $0.001$                     & $0.204$                            & \multicolumn{1}{l|}{$0.056$}                    & $0.355$                    \\ \hline
\multirow{3}{*}{SGD w/ a.g.}        & $0.01$                      & $0.157$                            & \multicolumn{1}{l|}{$0.057$}                    & $0.057$                    \\ \cline{2-5} 
                            & $0.005$                     & $0.206$                            & \multicolumn{1}{l|}{$0.073$}                    & $0.403$                    \\ \cline{2-5} 
                            & $0.001$                     & $0.297$                            & \multicolumn{1}{l|}{$0.118$}                    & $0.534$                    \\ \hline
\end{tabular}
\caption{Results of neural operator training for the reaction diffusion PDE problem. Hessian averaged methods use $(1,20,40)$ vectors for diagonal estimation.  For the $L^2_\pi$ training problem, Adam reliably performed the best; while Dan, Dan2 ran into some stability issues for $\alpha_0 = 0.01$. When Dan and Dan2 did not diverge they performed better than Adahessian which performed better than SGD. For the $H^1_\pi$ training problem, the Hessian-averaged Newton methods performed much better; for these problems, Dan, Dan2 and Adam all performed somewhat comparably, with Adahessian performing slightly worse and SGD performing substantially worse. The $H^1_\pi$ training problem has much richer information per datum, perhaps leading to more well-informed subsampled Hessian approximations, leading to better performance. $\ddagger$: these runs failed to complete in the final iterations due to out of memory error, this could be overcome with a more efficient implementation, such an implementation is however outside of the scope of this work.  }\label{table:no_results}
\end{table}

\subsection{Summary of findings}

We summarize our findings as follows

\begin{itemize}
    \item Hessian averaging can overcome the instabilities of subsampled Newton methods and still produce good Hessian-like operators for generating optimization iterates.

    \item Dan and Dan2 (as with Adahessian) are competitive with Adam in per-iteration computational work, particularly when reducing the frequency of Hessian computations. All three were emprirically superior to SGD in the experiments that we conducted. Additionally, efficiently implemented Hessian approximations using GPU vectorized computations leads runtime performance that is independent of the subspace rank $r$ before GPU memory is exhausted.

    \item Dan and Dan2 had competitive and often better performance than state-of-the-art methods Adam and Adahessian in our experiments. Using adaptive gradient sampling can simultaneously improve performance (e.g., generalization accuracy) while also reducing total computation as the Hessian computations per-epoch are reduced as the gradient sample size is increased. 
    
\end{itemize}

%% file: sections/conclusions.tex

In this work we have advanced Hessian-averaged Newton methods with adaptive gradient sampling as a compelling class of fully-inexact 
methods with significant advantages both theoretically and practically. Our global and local convergence theory demonstrates that these methods are capable of fast local convergence for fixed per-iteration Hessian computations, when utilizing 
generalized norm test 
to determine gradient sample sizes. When utilizing deterministic sampling strategies without replacement, we developed a local superlinear convergence rate, that improved the best existing rate from $\mathcal{O}\left(\frac{1}{\sqrt{k}}\right)$ to $\mathcal{O}\left( \frac{1}{k} \right)$, while simultaneously relaxing assumptions of gradient exactness and maintaining similar global to local transition phase iteration complexity. We additionally extended our analysis to the stochastic  gradient sampling setting, establishing a full convergence theory in expectation, matching the $\mathcal{O}\left( \frac{1}{k} \right)$ rate of \cite{na2023hessian,jiang2024stochastic}, albeit using additional assumptions. Furthermore, we establish local gradient complexity results matching those of adaptive first-order gradient methods and improving upon stochastic gradient methods. 

From a practical standpoint, we advanced Hessian-averaging as a variance reduction strategy that reduced stochasticity-driven instabilities of second-order methods. From a computational cost perspective, we emphasized how matrix-free Hessian approximations can be efficiently computed in modern computing frameworks for $\mathcal{O}(1)$ per-iteration Hessian-vector products and computational time, only requiring $\mathcal{O}(d)$ memory. We introduce the efficient diagonally-averaged Newton methods (Dan and Dan2) as practical extensions of the algorithmic framework that we investigate in this work. In numerical experiments, we demonstrated that these methods are not only computationally competitive with first-order methods, but consistently produce competitive and often superior generalization accuracies in complex deep learning tasks.

%% file: sections/acknowledgments.tex
This work was partially supported by the National Science Foundation under award DMS-2324643. 
The authors would like to thank Albert Berahas and Shagun Gupta for feedback on some of the writing. 
The authors would like to thank Omar Ghattas and Umberto Villa for access to computing resources.

%% file: sections/appendix_analysis.tex
\subsection{Deterministic Sampling Bounds}\label{appendix:analysis}
Given Assumption \ref{assum:bnd_var}, we have the following bound used in Lemma \ref{lem:samplesizes},

\begin{equation}
    \|g_k - \nabla f(w_k)\|^2 \leq 4 \left(\frac{N- |X_k|}{N} \right) (\beta_{1,g}\|\nabla f(w_k)\|^2 + \beta_{2,g}).
\end{equation}
The bound is stated in the beginning of Section 3.1 in \cite{friedlander2012hybrid}, we restate it here for completeness of our presentation:

\begin{align} \label{eq:friedlander_schmidt_gbound}
    \|g_k - \nabla f(w_k)\|^2 &= \left\|\left( \frac{N - |X_k|}{N|X_k|} \right)\sum_{i\in X_k}\nabla F_i(w_k) - \frac{1}{N}\sum_{i\in [N]\setminus X_k} \nabla F_i(w_k)\right\|^2 \nonumber\\ 
    &\leq \left[\left( \frac{N - |X_k|}{N|X_k|} \right)\left\|\sum_{i\in X_k}\nabla F_i(w_k)\right\| + \frac{1}{N} \left\|\sum_{i\in [N]\setminus X_k} \nabla F_i(w_k)\right\|\right]^2 \nonumber \\
    &\leq \left[\left( \frac{N - |X_k|}{N|X_k|} \right)\sum_{i\in X_k}\|\nabla F_i(w_k)\| + \frac{1}{N} \sum_{i\in [N]\setminus X_k}\| \nabla F_i(w_k)\|\right]^2 \nonumber \\
    &\leq  4 \left( \frac{N - |X_k|}{N} \right)(\beta_{1,g}\|\nabla f(w_k)\|^2 + \beta_{2,g}),
\end{align}
where $[N] := \{1,2,\cdots,N\}$.

Similarly, given  Assumption \ref{assum:hessian_var}, we have the following bound used in Lemma \ref{lem:sampling error},

\begin{equation}
    \|H_k - \nabla^2 f(w_k)\|^2 \leq 4 \left(\frac{N- |S_k|}{N} \right) (\beta_{1,H}\|\nabla f(w_k)\|^2 + \beta_{2,H}).
\end{equation}
The bound is stated in the beginning of Section 3.1 in \cite{friedlander2012hybrid}, we restate it here for completeness of our presentation:

\begin{align} \label{eq:friedlander_schmidt_Hbound}
    \|H_k - \nabla^2 f(w_k)\|^2 &= \left\|\left( \frac{N - |S_k|}{N|S_k|} \right)\sum_{i\in S_k}\nabla^2 F_i(w_k) - \frac{1}{N}\sum_{i\in [N]\setminus S_k} \nabla^2 F_i(w_k)\right\|^2 \nonumber\\ 
    &\leq \left[\left( \frac{N - |S_k|}{N|S_k|} \right)\left\|\sum_{i\in S_k}\nabla^2 F_i(w_k)\right\| + \frac{1}{N} \left\|\sum_{i\in [N]\setminus S_k} \nabla F_i(w_k)\right\|\right]^2 \nonumber \\
    &\leq \left[\left( \frac{N - |S_k|}{N|S_k|} \right)\sum_{i\in S_k}\|\nabla^2 F_i(w_k)\| + \frac{1}{N} \sum_{i\in [N]\setminus S_k}\| \nabla^2 F_i(w_k)\|\right]^2 \nonumber \\
    &\leq  4 \left( \frac{N - |S_k|}{N} \right)(\beta_{1,H}\|\nabla^2 f(w_k)\|^2 + \beta_{2,H}).
\end{align}

%% file: sections/appendix_numerics.tex
\subsection{Numerical Results}\label{appendix:numerics}

In this appendix we overview additional details of the numerical results. Most implementation detail questions can be answered by reviewing the code in the accompanying repository \url{github.com/tomoleary/hessianaveraging} \cite{hessianaveraging}. The code was implemented in \texttt{jax} \cite{jax2018github}, and the numerical results were run on servers with NVIDIA A100 and L40S GPU. Access to large memory GPU may be required to run some of the results in the manuscript. 

\subsubsection{Stochastic Quadratic Minimization}

In section \ref{subsection:stochastic_quadratic} we consider a subsampled quadratic minimization problem:
\begin{equation}
    \text{Subsampled quadratic:} \qquad \min_w f(w) =  \mathbb{E}_{P_A,P_b}\left[\|P_AAw - P_bb\|^2 \right],
\end{equation}
where $P_A$, and $P_b$ randomly zero out a certain number of entries in $A$ and $b$ respectively; this problem is a simple analogue to empirical risk minimization over a dataset. 

In order to investigate adaptive gradient sampling, we use different $P_A,P_b$ for the Hessian and gradient calculations, and in order to satisfy the norm test we reduce the number of zero entries in $P_A,P_b$ in order to satisfy the norm test. The true $A$ matrix is taken to be a positive definite matrix with spectrum $\lambda_i = 10^{-4} + (0.1i)^\frac{3}{2}$, with $d = 100$. In this case the Hessian condition number $\kappa(A^TA) \approx 10^6$, which gives a restrictive Lipschitz condition for gradient descent. When employing the norm test, we take $\theta_k = 0.5$ to be constant, so we are limited to the (fast) linear local convergence regime of our theory.

\subsubsection{CIFAR[10,100] classification with ResNets} \label{appendix:numerics_cifar}
In section \ref{subsection:cifar_results} we investigate image classification with CIFAR[10,100] datasets. We utilize a ResNet architecture based on \cite{jruseckas}, similar but not identical to those utilized in \cite{he2016deep}. Our results are able to achieve similar accuracies to typical ResNet architectures, but do not reference established benchmarks. All training runs for a given seed are run from the same initial guess. We use learning rate schedulers that reduce the learning rate by a factor of four every 25\% of epochs in order to obtain more practical performance. The architectures used for CIFAR10 and CIFAR100 are nearly identical, and only differ in the final layers, which map to $\mathbb{R}^{10},\mathbb{R}^{100}$, respectively. The weight dimensions are correspondingly $d = 11,200,882$ and $d = 11,247,052$. The details of the network architecture are much easier to define in code, and are taken from \cite{jruseckas}, so we refer the interested reader to the repository \cite{hessianaveraging} for the implementation, and encourage them to run the code. Access to GPUs with large RAM will be necessary, all results we used were run on NVIDIA A100 (40 and 80 GB) and L40S (48 GB). For all cases except one we report averages over \texttt{jax} seeds 0 and 1. 

\textbf{($\mathbf{\dagger}$)}: Note from Table \ref{table:cifar_results}: in the case of the CIFAR100 Dan with $\alpha = 0.01$ all methods suffered from early iteration divergence using seed 0. Averaged over seeds 0 and 1 the corresponding accuracies of this method goes from the reported $(71.18,72.10,71.57)$ down to $(48.68,62.52,64.03)$. This result is an anomaly due to unlucky initial guess and/or sampling order, which was resolved by running a different seed, however we document this issue for full transparency and reproducibility. 

\subsubsection{Parametric PDE Learning Numerical Details}

Note that the loss function in \eqref{eq:dino_emp_risk} requires one derivative of the neural network, and the Hessian-vector products used in all optimization problem require two more. For this reason with consider $C^3$ continuous activation functions, and thus uze the Gaussian error linear unit, \texttt{gelu} activation function. The networks are generic feedforward multi-layer perceptrons that have inputs of $200$, followed by five laters with dimension $400$, which then gets reduced to the output, which has dimension $50$.

\subsection{A Note on (Derivative-Informed) Parametric PDE Learning}\label{appendix:dino}

We give a brief synopsis of parametric PDE learning, which is of great interest to the authors, in order to give additional context to the numerical results in section \ref{section:parametric_pde_numerics}. Learning parametric PDE maps via neural network representations has become a research topic of great interest in recent years. In the typical setup there is a parameter function $x \in \mathcal{X}$, which is mapped out to an output $y \in \mathcal{Y}$ implicitly, through the solution of an expensive-to-evaluate PDE model; this map is then $x\mapsto y(x)$. The parametric PDE learning problem is typically motivated through computationally expensive tasks such as Bayesian inverse problems, optimization problems under uncertainty, optimal design, optimal experimental design, rare event estimation, all of which have very large computational costs through iteration and sampling complexities. The goal of the parametric PDE learning problem is to construct a surrogate for the parametric PDE maps showing up in the aforementioned tasks that can be substituted for direct forward simulation within these algorithms, and specifically to do so at a lower end-to-end cost, including accounting for the costs of sampling the PDE map to obtain training data. 

The spaces $\mathcal{X},\mathcal{Y}$ are generally separable Banach spaces, but for the remainder of this presentation we will assume them to be separable Hilbert spaces. The inputs $x \in \mathcal{X}$ are equipped with an input distribution $\pi$. We seek to construct and train a neural network approximation $y_w(x) \approx y(x)$. In this setting we consider $\mathcal{X}$ to represent an infinite-dimensional space, while $\mathcal{Y}$ is either an infinite-dimensional space (e.g., to represent the PDE state) or a finite-dimensional vector-valued function on the state (as was the case in the numerical results in section \ref{section:parametric_pde_numerics}). The so-called neural operator formulation is to formulate both the approximation and the training problem in a function space setting (following the so-called ``optimize-then-discretize'' approach) \cite{kovachki2023neural}\footnote{We note that the term neural operator is typically reserved for circumstances that both $x$ and $y$ represent infinite-dimensional functions, however the general framework is useful in cases where one or both are functions, as it may lead to discretization-dimension independent representations \cite{huang2024operator}.}. A typical formulation is in the parametric Bochner space $L^2_\pi = L^2(\mathcal{X},\pi;\mathcal{Y})$, i.e., to formulate the following optimization problem
\begin{equation} \label{eq:inf_dim_l2no}
    \min_w \left(\|y- y_w\|_{L^2_\pi}^2 =  \mathbb{E}_{\pi}\left[\|y - y_w\|_\mathcal{Y}^2 \right] \right).
\end{equation}
By first formulating the neural network training problem in the function space setting one can derive neural network architectures that respect the continuum limit of the PDE map, and lead to efficient statistical learning formulations. Popular examples of this general approach include PCANet, Fourier Neural Operator (FNO) and DeepONet \cite{kovachki2023neural,HesthavenUbbiali18,LiKovachkiAzizzadenesheliEtAl2020a,LuJinKarniadakis2019}. These architectural representations utilize appropriate basis representations (e.g., proper orthogonal decomposition (POD), Fourier basis etc.) that allow for efficient finite dimensional approximations that have approximation properties independent of the discretization dimension of the problem. In the numerical results we considered we utilize an architecture that restricts the input function $x$ to the dominant eigenfunctions of the expected sensitivity operator. For example when the inputs are distributed with Gaussian measure, $x \in \mathcal{N}(\overline{x},\mathcal{C}_x)$, we compute the dominant eigenfunctions $\psi_i$ of the following generalized eigenvalue problem for the eigenpairs $\lambda_i,\psi_i$ with $\lambda_i \geq \lambda_j$ for $i<j$:

\begin{equation}
    \mathbb{E}_\pi [D_x y ^* D_x y] \psi_i = \lambda_i\mathcal{C}_x^{-1} \psi_i,
\end{equation}
where $\cdot^*$ denotes the adjoint of $D_x y$. The architectures are referred to as derivative-informed projected neural networks DIPNets \cite{OLearyRoseberryVillaChenEtAl2022,OLearyRoseberryDuChaudhuriEtAl2021}, and are motivated by the existence of input-reduced approximations $y_r$ of the map with bounds satisfying
\begin{equation}
    \|y - y_r \circ \Psi_r \Psi_r^*\|_{L^2_\pi}^2 \leq \sum_{i\geq r}\lambda_i.
\end{equation}

The form of $y_r$ is a conditional expectation ridge function that marginalizes out the orthogonal complement to the subspace spanned matrix of eigenfunctions $\Psi_r = [\psi_1,\dots,\psi_r]$, with $r\in \mathbb{N}$. The bound is derived using the Poincar\'e inequality (in this case for Gaussian measures) \cite{bogachev1998gaussian,zahm2020gradient}. In our numerical results the output dimension was reduced via pointwise evaluation of the PDE state at finite points in the domain (as is relevant to inverse problems), however other dimension strategies such as POD could have also been employed to similar effect. The empirical risk minimization problem associated with \eqref{eq:inf_dim_l2no} in this specific architecture leads to the efficient finite-dimensional reduced basis coefficient learning problem in \eqref{eq:l2_no_emp_risk}.

Accurately trained neural operators have been deployed to solve complex inference and uncertainty propagation tasks that would have been out of reach when using a traditional forward simulation \cite{pathak2022fourcastnet}. However when they are deployed in the context of an optimization problem, the $L^2_\pi$ formulation is insufficient. Suppose we have an optimization problem of the form:
\begin{equation} \label{eq:generic_pde_constrained_opt}
    \min_x f(y(x),x),
\end{equation}
which is solved via gradient-based methods. Problems of the form \eqref{eq:generic_pde_constrained_opt} include traditional PDE-constrained optimization problems (additionally including uncertainty when $f$ is a risk measure over additional parameters in the PDE system), but also captures other tasks such as variational inference, e.g., evidence-based lower bound optimization (ELBO) \cite{ruthotto2021introduction}. We will show in the following proposition that training only on the function values and not also the derivative (e.g., the $L^p_\pi$ parametric Bochner space formulations) may be insufficient to ensure accurate gradients when substituting $y_w$ for $y$. In the following we use $D_x$ to denote the (total) Fr\'echet derivative of the objective function $f$ with respect to the input function $x$, while we denote partial derivatives by $\partial_x$.

\begin{proposition}{Error bound for parametric PDE gradients (similar to Proposition 3.1 in \cite{luo2023efficient}).}

Supposed the function $f$ has a Lipschitz partial (Fr\'echet) derivatives w.r.t both $x$ and $y$ with constant $L_f$, then we have the following bound:

\begin{align}
    \|D_x f(y(x),x) - D_x f(y_w(x),x)\| \leq \nonumber \\
    L_f(1 + \|D_x y(x)\|)\underbrace{\|y(x) - y_w(x)\|}_\text{function error} + \|\partial_y f(y_w(x),x)\|\underbrace{\|D_x y(x) - D_x y_w(x)\|}_\text{derivative error}.
\end{align}

\end{proposition}

\begin{proof}
We have the following bounds
\begin{align}
    \|D_x f(y(x),x) - D_x f(y_w(x),x)\| \leq \left\|\partial_x f(y(x),x) - \partial_x f(y_w(x),x) \right\| + \nonumber \\
    \quad \|\partial_{y}f(y(x),x) - \partial_{y}f(y_w(x),x)\|\|D_x y(x)\| + \|\partial_{y}f(y(x),x)\|\|D_x y(x) - D_x y_w(x) \|\nonumber \\
    = L_f(1 + \|D_x y(x)\|)\|y(x) - y_w(x)\| + \|\partial_y f(y_w(x),x)\|\|D_x y(x) - D_x y_w(x)\|.
\end{align}
\end{proof}

This bound shows that the derivative of the PDE map also needs to be controlled in addition to the function error, in order to obtain accurate gradients, which are required for the solution of optimization problems of the form \eqref{eq:generic_pde_constrained_opt}. For this reason the derivative-informed operator learning formulation was introduced in \cite{OLearyRoseberryChenVillaEtAl2022}, which proposes the operator learning in the parametric Sobolev space $H^1_\pi = H^1(\mathcal{X},\pi;\mathcal{Y})$, i.e. to solve the following optimization problem to train the surrogate:
\begin{equation} \label{eq:inf_dim_dino}
    \min_w \left( \|y - y_w\|_{H^1_\pi}^2 = \mathbb{E}_\pi \left[\|y - y_w\|_{\mathcal{Y}}^2 + \|D_x y - D_x y_w\|^2_{HS(\mathcal{X},\mathcal{Y}} \right] \right),
\end{equation}
where $\|\cdot\|_{HS(\mathcal{X},\mathcal{Y})}$ denotes the Hilbert Schmidt norm for linear operators $A:\mathcal{X} \rightarrow \mathcal{Y}$. Utilizing the linear reduced basis architecture discussed above the empirical risk minimization analogue of \eqref{eq:inf_dim_dino} takes the form of the efficient reduced basis coefficient derivative learning problem in \eqref{eq:dino_emp_risk}.

Neural operators trained in this formulation are referred to as derivative-informed neural operators (DINOs). They have favorable cost accuracy tradeoff over traditional formulations for learning the function in $L^2_\pi$; this includes the sampling costs for all PDE solves including the derivative computations using adjoints and directional sensitivities. This phenomenon is apparent in the numerical results shown in Section \ref{section:parametric_pde_numerics}. Additionally, they produce better gradients leading to better PDE-constrained optimization \cite{luo2023efficient}, and better (Gauss--Newton) Hessians used in the efficient solution of Bayesian inverse problems \cite{cao2024efficient}.